\documentclass[a4paper,11pt,openright,twoside]{article}

\usepackage{color}
\usepackage{latexsym}
\usepackage{amssymb,amsbsy,amsmath,amsfonts,amssymb,amscd}
\usepackage{amsthm}
\usepackage{upref}
\usepackage{epsfig, graphicx}
\usepackage[hidelinks]{hyperref}
\usepackage{enumerate}
\usepackage{mathrsfs}
\usepackage{color}
\usepackage{ stmaryrd }
\usepackage{ dsfont }
\usepackage{mathtools}
\mathtoolsset{showonlyrefs}
\usepackage{pgf,tikz}
\usetikzlibrary{arrows}
\usepackage{soul}

\usepackage{float}
\usepackage{tikz}
\usepackage[all]{xy}
\usepackage{wrapfig}
\usepackage{enumitem}
\usepackage{multirow, array} 
\usepackage{setspace}
\usepackage{appendix}
\usepackage{fancyhdr}

\usepackage{xfrac}    
\usepackage{faktor}

\usepackage{bbm}
\def\1{\mathbbm{1}}

\newcommand{\commentout}[1]{}
\newcommand{\R}{\mathbb{R}}
\newcommand{\N}{\mathbb{N}}
\newcommand{\Z}{\mathbb{Z}}
\newcommand {\e}  {\mathrm{e}}
\newcommand {\diff} {\mathrm{d}}
\newcommand {\Chi} {{\bf \raise 2pt \hbox{$\chi$}} }

\newcommand{\norme}[1]{\left\lVert#1\right\rVert}
\newcommand{\pscal}[2]{\left\langle ~#1~\middle\vert~#2~\right\rangle}

\DeclareMathOperator\erf{erf}

\newtheorem{theorem}{Theorem}[section]
\newtheorem{maintheorem}[theorem]{Main Theorem}
\newtheorem{lemma}[theorem]{Lemma}
\newtheorem{hyp}[theorem]{Assumption}

\newtheorem{remark}[theorem]{Remark}
\newtheorem{example}[theorem]{Example}

\newtheorem{corollary}[theorem]{Corollary}

\numberwithin{equation}{section}

\makeatletter
\def\@cite#1#2{\textup{[{#1\if@tempswa, #2\fi}]}}
\makeatother

\title{\Large \bf Noise-driven bifurcations in a nonlinear Fokker--Planck system describing stochastic neural fields}

\author{Jos\'e A. Carrillo\thanks{Mathematical Institute, University of Oxford, Oxford OX2 6GG, UK (carrillo@maths.ox.ac.uk)} \and Pierre Roux\thanks{Mathematical Institute, University of Oxford, Oxford OX2 6GG, UK (pierre.roux@maths.ox.ac.uk)} \and Susanne Solem\thanks{Department of Mathematics, Norwegian University of Life Sciences, NO-1433 \AA s, Norway (susanne.solem@nmbu.no)}}

\date{\today}

\begin{document}

\maketitle
\allowdisplaybreaks

\begin{abstract}
    \noindent The existence and characterisation of noise-driven bifurcations from the spatially homogeneous stationary states of a nonlinear, non-local Fokker--Planck type partial differential equation describing stochastic neural fields is established. The resulting theory is extended to a system of partial differential equations modelling noisy grid cells. It is shown that as the noise level decreases, multiple bifurcations from the homogeneous steady state occur. Furthermore, the shape of the branches at a bifurcation point is characterised locally.
    The theory is supported by a set of numerical illustrations of the condition leading to bifurcations, the patterns along the corresponding local bifurcation branches, and the stability of the homogeneous state and the most prevalent pattern: the hexagonal one.
\end{abstract}

\pagenumbering{arabic}

\section{Introduction}
We establish the existence of noise-driven bifurcations from the spatially homogeneous steady states of the partial differential equation (PDE)

\begin{equation}\label{eqn:2}
	\tau\dfrac{\partial \rho}{\partial t} = -\dfrac{\partial }{\partial s}\left[  \left(  \Phi\left( \int_{\mathbb T^d} W(x-y)\int_0^{+\infty}s\rho(y,s,t)\diff s \diff y + B\right) -s  \right)\rho(x,s,t)   \right] + \sigma \dfrac{\partial^2 \rho}{\partial s^2},
\end{equation}
which aims to describe a network of noisy neurons.
In this nonlinear Fokker--Planck type model, $\rho(x,s,t)$ is at time $t$ the probability density of neurons at location $x\in \mathbb T^d$, with activity level $s\in[0,+\infty)$. Here, $\mathbb T^d$ is the $d$-dimensional square torus of length $L$ endowed with the quotient topology, $\mathbb T^d := \sfrac{\R^d}{( L \Z)^d}$, as the neurons activity levels are assumed to be spatially connected on a toroidal geometry. The interactions between the neurons are modeled through the periodic connectivity function $W$ weighting the averaging in space of the mean activity levels of the neurons. The network of neurons is assumed to receive a constant external input through $B$, and the  total input signal to each neuron is then controlled through the modulation function $\Phi$. The parameter $\tau$ is the system's relaxation time, and $\sigma$ determines the noise level in the network. To prohibit the activity level from becoming negative, a no-flux boundary condition is prescribed at $s=0$:
\begin{equation*}
	\left( \Phi\left( \int_{\mathbb T^d} W(x-y)\int_0^{+\infty}s\rho(y,s,t)\diff s \diff y + B\right)\rho(x,s,t) - \sigma \dfrac{\partial \rho}{\partial s}(x,s,t) \right)\bigg|_{s=0} = 0,
\end{equation*}
for all $(x,t)\in\mathbb T^d\times\R_+$. Last, we make the assumption that the neurons are distributed homogeneously in space, thereby imposing that any initial datum $\rho^0$ satisfies
\begin{equation}\label{eqn:mass_x}
\forall x\in\mathbb T^d, \qquad \int_{0}^{+\infty} \rho^0(x,s)\diff s = \dfrac1{L^d}.
\end{equation}
Since there is no movement of neurons in space and owing to the no-flux boundary condition, it can be checked that for any smooth solution this property propagates in time to the density $\rho(x,\cdot,t)$.

Our results also apply to the following generalisation of \eqref{eqn:2} :
\begin{align}\label{eq:4PDE}
\tau \frac{\partial \rho^\beta}{\partial t} =
-\frac{\partial}{\partial s}\Bigg(
\Big[\Phi^\beta(x,t) -s\Big] \rho^\beta
\Bigg) + \sigma \frac{\partial^2 \rho^\beta}{\partial s^2},
\end{align}
where $\Phi^\beta(x,t)$ is given by
\begin{align}\label{eq:phi}
\Phi^\beta(x,t) = \Phi \left(\frac{1}{4}\sum_{\beta'=1}^4 \int_{\mathbb T^d} W^{\beta'}(x-y) \int_{0}^\infty s \rho^{\beta'} (y,s,t)\, \diff s \diff y + B^\beta(t) \right),
\end{align}
for $\beta =1,2,3,4$, with corresponding no-flux boundary and mass normalisation conditions. In the case $d=2$, \eqref{eq:4PDE} models a network of noisy grid cells \cite{CHS}. Grid cells, discovered in \cite{gridcells}, are neurons which play a pivotal role in spatial representation by firing (emitting spikes) in a hexagonal pattern as a mammal moves around in an open environment, see \cite{tenyears} for a short summary. In addition to the mechanisms already incorporated in \eqref{eqn:2}, \eqref{eq:4PDE} also splits the neurons into four different groups, each having their orientation preference in physical space: $\beta =1,2,3,4$ (north, west, south, east). To model the observed behaviour of grid cells, an orientation preference dependent shift $r^\beta$ is included in the connectivity $W^\beta(x-y)=W(x-y-r^\beta)$, and \eqref{eq:4PDE} is connected with the movement of a mammal through the orientation and time dependent input $B^\beta$. We note however, that time dependent behaviour is beyond the scope of the present manuscript. In the rest of the manuscript we therefore assume $B^\beta(t) = B$, where $B$ is a constant, as was done in \cite{CHS}.

In \cite{tenyears} understanding the effects of noise on networks of grid cells was emphasized as a challenge, and this sparked the development and initial study of \eqref{eq:4PDE} in \cite{CHS}. The model \eqref{eq:4PDE}, which is based on the ODE models of \cite{coueyetal,burakfiete}, has been rigorously shown to be the mean-field limit of a network of stochastic grid cells with Gaussian independent noise \cite{CCS21} by adapting Sznitman's coupling method \cite{Sznitman1991}. Furthermore, novel evidence pointing towards a toroidal connectivity for grid cells has recently been provided in \cite{Gardner2022} by using topological data analysis tools.

The authors in \cite{CHS} showed the existence and uniqueness of spatially homogeneous stationary states for any given noise strength $\sigma >0$ to equations \eqref{eqn:2} and \eqref{eq:4PDE} under suitable assumptions on the model parameters. Notice that \eqref{eqn:2} and \eqref{eq:4PDE} share the same spatially homogeneous solutions. This one-parameter family of spatially homogeneous solutions will be denoted by $(\rho_\infty^\kappa, \kappa)$ where $\kappa=1/\sigma$ for ease of notation. The main contribution in \cite[Theorem 3.6]{CHS} implies a linear stability criteria for $\rho_\infty^\kappa$ as a solution to \eqref{eqn:2}. More precisely, $\rho_\infty^\kappa$ is linearly asymptotically stable in $L^2\big({\mathbb T}^d  \times \R_+\big)$ for \eqref{eqn:2} as long as
\begin{align} \label{eq:linearstabilitycondintro}
\dfrac{\tilde W (k)}{\Theta(k)} < \frac{1}{\kappa L^{\frac d2}  (\Phi_0^\kappa)' \int_0^\infty (s-L^d\bar{\rho}_\infty^\kappa)^2 \rho_\infty^\kappa \diff s},    
\end{align}
for all $k \in \mathbb{N}$ with $ (\Phi_0^\kappa)':=\Phi'(W_0 \bar\rho_\infty^\kappa + B)$, $\bar\rho_\infty^\kappa$ the mean in $s$ of $\rho_\infty^\kappa$, $W_0$ the integral of $W$ over $\mathbb T^d$, and $\tilde W (k)$ the Fourier mode of the periodic connectivity kernel $W$. Here, $\Theta(k)$ is a normalization factor for the Fourier coefficient of a periodic function in $L^2\big({\mathbb T}^d\big)$ whose precise definition is found in Section \ref{sec:prelim}. Moreover, \cite[Section 4]{CHS} provides numerical evidence of the existence of bifurcation branches emanating from the curve $(\rho_\infty^\kappa, \kappa)$ for \eqref{eq:4PDE} leading to hexagonal patterns in the case of a radially symmetric connectivity kernel $W$, which are reminiscent of similar stationary patterns for the Wilson--Cowan equations \cite{EC1979}, see \cite{Murray03}.

In this work, we obtain sharp conditions for the appearance of local bifurcation branches from $(\rho_\infty^\kappa, \kappa)$ based on the Fourier modes of the connectivity kernel $W$ for both \eqref{eqn:2} and \eqref{eq:4PDE}. Our main result, stated in Main Theorem \ref{thm:mainmain}, shows that for any value of $\kappa$ such that there exists a unique Fourier mode $k^*$ of $W$ with
\begin{align} \label{eq:branchintro}
\dfrac{\tilde W (k^*)}{\Theta(k^*)} = q(\bar \rho_\infty^\kappa, \kappa):=\dfrac{1}{L^{\frac d2} (\Phi_0^\kappa)' \left( \dfrac1{L^d} - L^d\bar \rho_\infty^\kappa  \left(\bar \rho_\infty^\kappa - \dfrac{ \Phi_0^\kappa}{L^d}\right)\kappa \right)},    
\end{align}
 leads to a local bifurcation branch of spatially in-homogeneous stationary solutions to \eqref{eqn:2}. This result is illustrated in Fig. \ref{fig:intro-illustration} for a particular choice of the parameters in the model, see Section 5 for more details. We will show in Section 2 that $q(\bar \rho_\infty^\kappa, \kappa)$, the right hand side of \eqref{eq:branchintro}, coincides with the right hand side of \eqref{eq:linearstabilitycondintro}. As a byproduct, Main Theorem \ref{thm:mainmain} shows that the linear asymptotic stability condition \eqref{eq:linearstabilitycondintro} is sharp under suitable assumptions for a large family of connectivity functions $W$. In particular, we only require that the connectivity kernel $W$ is coordinate-wise even as opposed to the common assumption of a radially symmetric kernel in neural field models. The study of kernels being less symmetric is relevant from a neuroscience perspective as connections between neurons can deteriorate, or be locally disturbed by activity outside the network.

Analogous results with a slightly different condition than \eqref{eq:branchintro}, involving the shifts $r^\beta$, hold for \eqref{eq:4PDE}, see Section 4. More precisely, we show that a rigorous study of the bifurcations of the one component model \eqref{eqn:2} is sufficient to understand the bifurcations of the four component model \eqref{eq:4PDE}, and that the conditions leading to bifurcations for \eqref{eq:4PDE} are a simple modification of the conditions for \eqref{eqn:2}.

\begin{figure}[ht]
    \centering
    \includegraphics[width=0.45\textwidth]{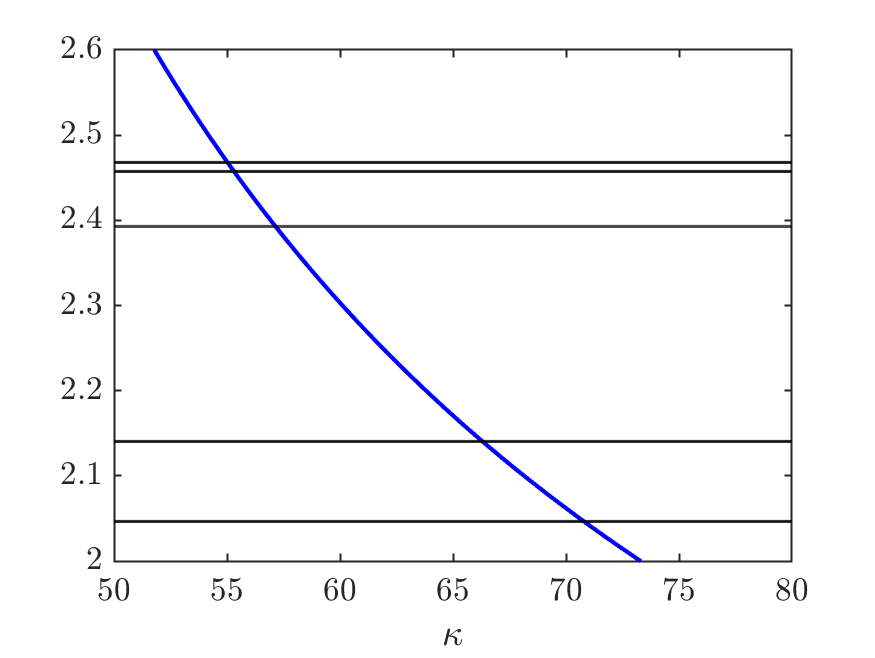}
     \includegraphics[width=0.45\textwidth]{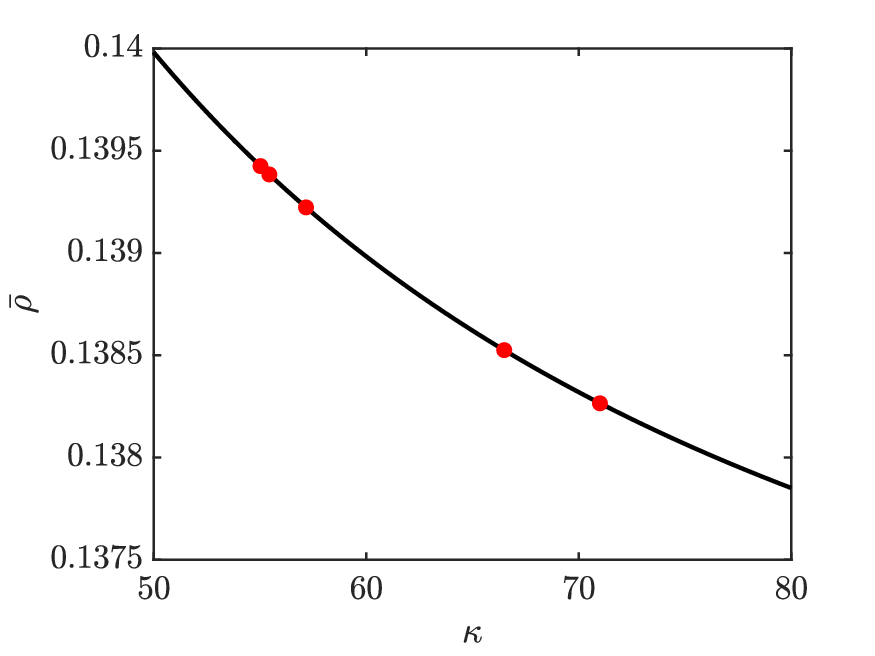}
    \caption{Plots illustrating Main Theorem \ref{thm:mainmain} for the same choice of parameters as in the left plot of Figure \ref{fig:radsym-vs-nonradsym}. Left: illustration of the bifurcation condition \eqref{eq:branchintro} as crossings between the Fourier modes ${\tilde W (k^*)}/{\Theta(k^*)}$ of the connectivity function $W$ (black, horizontal lines) and the function $q(\bar \rho_\infty^\kappa, \kappa)$, right: corresponding bifurcation points (red dots) along the spatially homogeneous branch (black line).}
    \label{fig:intro-illustration}
\end{figure}

The local bifurcation analysis is obtained by adapting the strategy of \cite{CGPS20} developed for nonlinear nonlocal McKean--Vlasov equations on the torus. The authors in \cite{CGPS20} used the classical approach of Crandall--Rabinowitz local bifurcation theorems in the right functional setting to show branching from constant stationary states. However, in our setting we have to deal with spatially homogeneous steady states that are not constants in $s$. We therefore apply Crandall--Rabinowitz arguments to the averages in the activity level of stationary states instead. We will see that the stationary states $\rho_\infty^\kappa$ implicitly depend on the bifurcation parameter $\kappa$, which leads to cumbersome expressions for the required Fr\'echet derivatives. 

Another difficulty is that there is no gradient flow structure for \eqref{eqn:2}, or in other words, no variational formulation in terms of steepest descent of free energies as in \cite{CGPS20}. Therefore, although the numerical bifurcation diagrams of \cite{CHS} suggest the existence of discontinuous phase transitions, we cannot adapt the tools of \cite{CGPS20} to rigorously discover them. We thus mainly focus on obtaining all local bifurcation branches emanating from the family $(\rho_\infty^\kappa, \kappa)$ in this work. 

The strategy of \cite{CGPS20} does enable us to establish the condition leading to bifurcation points \eqref{eq:branchintro} and the existence of certain branches, but is not sufficient for characterising multiple branches emanating from the same point when the bifurcation problem is equivariant, for example when the connectivity kernel is radially symmetric. For this, we need the equivariant branching lemma (see for example \cite[Theorem 2.3.2]{CL-equivariantbif} and the classification of 2D and 3D planforms in \cite{Dionne}). A recent relevant paper in this direction is \cite{FSV2022} where bifurcations with respect to a gain parameter of a Wilson--Cowan-type equation \cite{WC1, WC2} for the primary visual cortex are studied.

Let us finally mention that noise driven patterns and bifurcations in related computational neuroscience models have been studied by several authors \cite{bressloff2012, TouboulPhysD,KE13,K14,MB, TouboulSIAM, B19, BAC19,RP19}, see \cite{FTC09,TCL16} for noise in the connectivity, largely from a stochastic viewpoint. Moreover, a rigorous derivation of mesoscopic models by coarsening from stochastic differential equations via mean-field limits have also been developed in \cite{MS02, FTC09,FI15,TouboulPhysD,TouboulSIAM,CT18}. These works deal with spatially extended systems of neural networks modelled by their voltage with random connectivity interactions using large deviation principles \cite{AG95,G97}.

The outline of the manuscript is as follows. We start by introducing the setting in Section \ref{sec:prelim}, before providing a summary of the main results concerning the bifurcations of \eqref{eqn:2}. We end the section by linking it to the linear stability condition provided for \eqref{eq:4PDE} in \cite{CHS}. Section \ref{sec:bifurcations} is devoted to proving the main results, which includes establishing some results on the spatially homogeneous stationary state itself. An outline of the generalisation of the main results to the four component model \eqref{eq:4PDE} is provided in Section \ref{sec:4component}. The theoretical results are then illustrated through a series of plots in Section \ref{sec:numerics}, before providing a set of future perspectives in Section \ref{sec:future}.


\section{Preliminaries, main results, and relation to stability}\label{sec:prelim}
In this section we set the stage for establishing the existence of noise-driven bifurcations of \eqref{eqn:2} and \eqref{eq:4PDE}. First, we provide details concerning the spaces we will work in, before stating the main result for \eqref{eqn:2} in Main Theorem \ref{thm:mainmain}. We also discuss how these results are linked to the stability conditions for the Mckean-Vlasov equation in \cite{CGPS20} and the grid cell model in \cite{CHS}.

\subsection{Characterisation of the stationary states}

Henceforth, we denote $\N=\{0,1,\dots\}$, $\R_+^*=(0,+\infty)$,
\[  \bar \rho(x) = \int_{0}^{+\infty} s\rho(x,s)\diff s  \]
the activity average of $\rho$ at $x \in \mathbb T^d $, and
\[   W_0 = \int_{\mathbb T^d} W(x) \diff x.   \]

By setting the left hand side of \eqref{eqn:2} to zero and integrating with respect to $s$, we obtain that the stationary states satisfy
\[   \sigma \partial_s \rho(x,s) =   - \left(s - \Phi\left( W \ast \bar \rho (x) + B \right) \right) \rho(x,s),  \]
using the no-flux boundary condition. Here, $u\ast v : x\mapsto \int_{\mathbb T^d} u(x-y)v(y)\diff y$ denotes the convolution of $u$ and $v$ on $\mathbb T^d$. Thus, the stationary states must solve
\begin{equation}\label{eq:stationarystateequ}
\rho(x,s) = \dfrac{1}{Z_\rho} \exp\left( - \dfrac{ \big( s - \Phi\left(  W\ast \bar \rho (x) + B \big) \right)^2}{2\sigma} \right), 
\end{equation}
with, owing to \eqref{eqn:mass_x},
\[  Z_\rho = L^d \int_{0}^{+\infty} \exp\left( - \dfrac{ \big( s - \Phi\left(  W\ast \bar \rho (x) + B \big) \right)^2}{2\sigma} \right) \diff s. \]

For the sake of simplicity, let us denote
\[  \kappa = \dfrac{1}{\sigma}, \qquad \Phi_{\bar\rho}(x) = \Phi(W \ast \bar\rho(x) + B),\qquad  \Phi'_{\bar \rho}(x) = \Phi'(W \ast \bar\rho(x) + B).  \]

Then, any stationary state must be a zero of the functional
\begin{equation*}
	\mathcal G(\rho ,\kappa)  = \rho -\dfrac{1}{Z_\rho}\e^{ -\kappa \frac{ \left( s - \Phi_{\bar\rho} \right)^2}{2}} ,
\end{equation*}
with
\[   Z_\rho = L^d \int_{0}^{+\infty} \e^{ -\kappa \frac{ \left( s - \Phi_{\bar\rho} \right)^2}{2}} \diff s. \]
Consequently, the average $x\mapsto\bar \rho(x)$ is a zero of the functional
\begin{equation}\label{eq:onecompfunctional}
	\bar{\mathcal G}(\bar \rho ,\kappa)  = \bar\rho -\dfrac{1}{Z_\rho} \int_0^{+\infty} s \e^{ -\kappa \frac{ \left( s - \Phi_{\bar\rho} \right)^2}{2}}\diff s,
\end{equation}
with
	\[   Z_\rho =  L^d \int_{0}^{+\infty} \e^{ -\kappa \frac{ \left( s - \Phi_{\bar\rho} \right)^2}{2}} \diff s. \] 
A quick check will confirm that we can fully characterise the stationary states of \eqref{eqn:2} through considering the zeros of the functional for the averages, $\bar{\mathcal G}(\bar \rho ,\kappa)$.

Finally, the short calculation
	\begin{align*}  \rho(x,0) &= - \int_{0}^{+\infty} \dfrac{\partial \rho}{\partial s}(x,s) \diff s = \kappa \int_{0}^{+\infty} (s-\Phi_{\bar \rho}(x) )\rho(x,s) \diff s \\
	& = \kappa \Big( \underbrace{\int_{0}^{+\infty} s \rho(x,s)\diff s}_{ = \bar\rho(x)} - \Phi_{\bar\rho}(x) \underbrace{\int_{0}^{+\infty} \rho(x,s)\diff s }_{=\frac1{L^d}}  \Big),
	\end{align*}
yields the following relation between the value at $s=0$ of the stationary states and the mean values,
	\begin{equation}\label{eqn:rho_0}
	\rho(x,0) = \left(\bar \rho(x) - \dfrac{\Phi_{\bar\rho}(x)}{L^d}  \right)\kappa.
	\end{equation}
This relation will appear throughout the manuscript.

\subsection{Hilbert spaces and Fourier bases}\label{sec:basis}

Throughout this manuscript, we use either the space $L^2(\mathbb T^d)$ or the Hilbert space $L^2_S(\mathbb T^d)$ of coordinate-wise even $L^2$-functions on $\mathbb T^d := \sfrac{\R^d}{( L \Z)^d}, L>0$, that is to say
\[ L^2_S(\mathbb T^d) = \left\{ \ u\in L^2(\mathbb T^d) \ \big| \ \forall i\in\llbracket 1,d \rrbracket,\ u(x_1,\dots,-x_i,\dots,x_d) = u(x_1,\dots,x_i,\dots,x_d)\ \ \text{a.e. in } \mathbb T^d \ \right\}.  \]
Similarly, for any $m\in \N$, we define the coordinate-wise even Sobolev space $H^m_S(\mathbb T^d) = H^m(\mathbb T^d)\cap L^2_S(\mathbb T^d)$. Following the reasons stated in \cite[Section 2.2]{FSV2022}, we will consider a connectivity kernel $W\in H^m_S(\mathbb T^d)$, with $2m > d$. The advantage is twofold. First, the smoothing property of the convolution and Sobolev embeddings imply that $W\ast h$ will be continuous for all $h\in L^2(\mathbb T^d)$. Second, it avoids pathological connectivity kernels with singularities.

While working in the space $L^2_S(\mathbb T^d)$, we use the Hilbert basis
$
	\left(\omega_k\right)_{k\in \N^d}
$
defined by
\begin{equation}\label{eqn:HilbertElements}  \omega_k(x) = \dfrac{\Theta(k)}{L^{\frac d2}} \prod_{i=1}^{d} \omega_{k_i}(x_i),  
\end{equation}
with
\[
\omega_{k_i}(x) = \left\{\begin{array}{ll}
    \cos\left(\dfrac{2\pi k_i}{L} x_i\right) & \mathrm{if}\ k_i > 0, \\
    1 & \mathrm{if}\ k_i=0,
\end{array}\right.
\qquad \mathrm{and} \qquad \Theta(k) = \prod_{i=1}^{d} \sqrt{2 - \delta_{k_i,0}},
\]
where the Kronecker delta $\delta_{i,j}$ is 0 if $i\neq j$ and 1 if $i=j$. We assume that the connectivity function $W$ lies in $L^2_S(\mathbb T^d)$. The Fourier modes of $W$ are
\[\tilde  W(k) = \pscal{ W }{ \omega_k} = \int_{\mathbb T^d} W(x)\omega_k(x)\diff x , \quad k\in\N^d,  \]
Moreover, since $W\in L_S^2(\mathbb T^d)$, for all $g\in L_S^2(\mathbb T^d)$,
\begin{equation*}
W\ast g (x) = \int_{\mathbb T^d} W(x-y) g(y) \diff y = \dfrac{L^\frac d2}{\Theta(k)} \sum_{k\in \N^d} \tilde W(k) \tilde g(k) \omega_k(x). 
\end{equation*}
In particular, for all $k\in\N^d$,
\begin{equation}\label{trigo}
	W\ast \omega_k (x) =  \dfrac{L^{\frac d2}\tilde W(k)}{\Theta(k)} \omega_k(x) .
\end{equation}

When we work in the larger space $L^2(\mathbb T^d)$, we extend the above Hilbert basis to multi-indexes in $\Z^d$: we consider the base $	\left(\omega_k\right)_{k\in \Z^d} $ defined by \eqref{eqn:HilbertElements},
\[
\omega_{k_i}(x) = \left\{\begin{array}{ll}
    \cos\left(\dfrac{2\pi k_i}{L} x_i\right) & \mathrm{if}\ k_i > 0, \\
    1 & \mathrm{if}\ k_i=0,\\
    \sin\left(\dfrac{2\pi k_i}{L} x_i\right) & \mathrm{if}\ k_i < 0. \\
\end{array}\right.
\qquad \mathrm{and} \qquad \Theta(k) = \prod_{i=1}^{d} \sqrt{2 - \delta_{k_i,0}}.
\]
The formula for the convolution with basis vectors then becomes
\begin{equation}\label{trigo2}
	W\ast \omega_k (x) =  \dfrac{L^{\frac d2}\tilde W(|k|)}{\Theta(|k|)}  \omega_k(x), \quad |k| = (|k_1|,\dots,|k_d|) .
\end{equation}

\subsection{Assumptions and main theorem}

We assume the following in order to perform the bifurcation analysis.	
\begin{hyp}[On $\Phi$, $W$ and $B$]\label{as:1} There exists $m\in\N$, such that $2m > d$ and $W\in H^m_S(\mathbb T^d)$. The function $\Phi$ is continuous on $\R$, $C^3$ on $\R^*$, increasing on $\R_+$ and $\Phi\equiv 0$ on $\R_-$. Moreover, $B > 0$ and  $W_0 = \int_{\mathbb T^d} W(y)\diff dy < 0$.
\end{hyp}	
We remind the reader that the existence and uniqueness of spatially homogeneous stationary states of \eqref{eqn:2} was proven in \cite[Prop. 3.1]{CHS} and hold under (less restrictive assumptions than) Assumption \ref{as:1}.
	
Let $\rho_\infty^\kappa: s\mapsto \rho_\infty^\kappa(s)$ be the constant in space steady state associated with the parameter $\kappa>0$. Note that in particular $ \bar{\mathcal G}(\bar \rho_\infty^\kappa,\kappa)=0$. Whenever necessary, the dependence on $\kappa$ of $\rho_\infty$ will be highlighted with a superscript $\kappa$. Otherwise we will drop it. 
Denote
\[  
 \Phi_0^\kappa = \Phi(W_0 \bar\rho_\infty^\kappa + B),\qquad  (\Phi^\kappa_0)' = \Phi'(W_0 \bar\rho_\infty^\kappa + B), \qquad (\Phi^\kappa_0)'' = \Phi''(W_0 \bar\rho_\infty^\kappa + B).   
\]
Then, applying \eqref{eqn:rho_0} to the spatially homogeneous steady state $\rho=\rho_\infty^\kappa$, we get the relation
	\begin{equation}\label{eqn:rho_inf_0}  \rho_\infty^\kappa(0) = \left(\bar \rho_\infty^\kappa - \dfrac{ \Phi_0^\kappa}{L^d}  \right)\kappa.  
	\end{equation}
This relation is used throughout the manuscript.

We are now ready to state the main result of this work. The different parts of the proofs are found in Section \ref{sec:bifurcations}.

\begin{maintheorem}\label{thm:mainmain} Grant Assumption \ref{as:1}, and let $\kappa^*\in\left(\tfrac{2|W_0|^2}{L^{2d}\pi B^2},+\infty\right)$. Assume that there exists a unique $k^*\in\N^d$ such that
    \begin{equation}\label{eqn:W_tilde_equal_0}
        \dfrac{\tilde W (k^*)}{\Theta(k^*)} = \dfrac{1}{L^{\frac d2} (\Phi_0^{\kappa^*})' \left( \dfrac1{L^d} - L^d\bar \rho_\infty^{\kappa^*}  \left(\bar \rho_\infty^{\kappa^*} - \dfrac{ \Phi_0^{\kappa^*}}{L^d}\right)\kappa^* \right)}.
    \end{equation}
    Assume also that $ (\Phi_0^{\kappa^*})''\geq 0$. Then $(\bar\rho_\infty^{\kappa^*},\kappa^*)$ is a bifurcation point of $\bar{\mathcal G}(\bar\rho,\kappa)=0$. 
    
    Moreover, in a neighbourhood $U\times (\kappa^*-\varepsilon,\kappa^*+\varepsilon)$ of $(\bar\rho_\infty^{\kappa^*},\kappa^*)$ in $L^2_S(\mathbb T^d)\times \R_+^*$, the average of the stationary state is either of the form $(\bar \rho_\infty^\kappa,\kappa)$ or on the non-homogeneous solution curve
    \begin{equation*}
        \{\ (\bar \rho_{\kappa(z)},\kappa(z)) \ | \ z\in(-\delta,\delta), \ (\bar \rho_{\kappa(0)},\kappa(0))=(\bar \rho_\infty^{\kappa^*},\kappa^*),\ \delta>0 \ \},
    \end{equation*}
     with
    \begin{equation*}
        \bar \rho_{\kappa(z)}(x) = \bar \rho_\infty^{\kappa(z)} + z \omega_{k^*}(x) + o(z), \qquad x\in \mathbb T^d.
    \end{equation*}
\end{maintheorem}

The uniqueness condition on $k^*$ can be relaxed: if the bifurcation problem exhibits symmetries and $k^*$ is unique up to appropriate geometrical transformations, equivariant bifurcation theory provides the same result. We treat this general case in Section \ref{sec:high_dim}.

 More information about the shape of the branch can be found in Theorem \ref{thm:main_2} for the case of a unique $k^*$ and in Section \ref{sec:high_dim} for the case where it is unique up to symmetries. A relaxed convexity assumption on $\Phi$ can be found in Remark \ref{rem:convexity}.

The right hand side of \eqref{eqn:W_tilde_equal_0} is strictly positive, as stated in the following lemma.
\begin{lemma}\label{lm:g_eta}
We have the relation
\[ \dfrac1{L^d} - L^d\bar \rho_\infty  \left(\bar \rho_\infty - \dfrac{ \Phi_0^\kappa}{L^d}\right)\kappa = \frac{1}{L^d} g\left(\sqrt{\frac{\kappa}{2}} \Phi_0^\kappa\right), \]
where
    \[
g(\eta):= 1-\frac{2}{\sqrt{\pi}}\frac{\exp (-\eta^2)}{1+\erf (\eta)}\left[\frac{1}{\sqrt{\pi}}\frac{\exp (-\eta^2)}{1+\erf (\eta)}+\eta\right],
\]
with $\eta = \sqrt{\frac{\kappa}{2}} \Phi_0^\kappa$. Furthermore, the function $\eta\mapsto g(\eta)$ is increasing and for all $\eta \in \R_+$,
\[ 1-\frac{2}{\pi} \leqslant g(\eta) < 1. \]
\end{lemma}
The function $\erf(z)$ is the standard error function, $\erf(z):=\tfrac{2}{\sqrt\pi}\int_{0}^{z} \e^{-s^2}\diff s$, and the proof of the lemma can be found in the appendix. As the right-hand side of \eqref{eqn:W_tilde_equal_0} is positive, there exists no bifurcations of the form described in the main theorem if all the Fourier modes of $W$ are negative. A similar conclusion was drawn in \cite{CGPS20} for the McKean--Vlasov equation, however with opposite sign. They introduce the notion of H-stability: a function $W\in L^2(\mathbb T^d)$ is said to be H-stable if all its Fourier modes are non-negative.
The authors of \cite{CGPS20} prove, under mild assumptions, that for H-stable potentials there exists a unique stationary state which is globally asymptotically stable. In the setting of this manuscript, if $-W$ is H-stable, the hypotheses of Theorem \ref{thm:mainmain} cannot be satisfied.  Proving that the constant in $x$ stationary state $\rho_\infty$ is asymptotically stable for the time-dependent system \eqref{eqn:2} when $-W$ is H-stable is an open question and the energy method utilised in \cite{CGPS20} is not applicable here.
However, there exists a result on the linear stability of \eqref{eqn:2}.

\subsection{Linear stability}\label{sec:linearstability}
In \cite[Theorem 3.6]{CHS} it is shown that the homogeneous in space stationary state $\rho_\infty^\kappa$ is linearly asymptotically stable in $L^2\big({\mathbb T}^d  \times \R_+\big)$ as long as 
\begin{align} \label{eq:linearstabilitycond1}
\dfrac{\tilde W (k)}{\Theta(k)} < \frac{1}{\kappa  (\Phi_0^\kappa)' L^{\frac d2} \int_0^\infty (s-L^d\bar{\rho}_\infty^\kappa)^2 \rho_\infty^\kappa \diff s}    
\end{align}
for all $k \in \mathbb{N}$. In the following lemma, we show that when replacing the inequality with an equality in \eqref{eq:linearstabilitycond1}, yielding the value of $\kappa$ where $\rho_\infty$ is no longer linearly asymptotically stable, \eqref{eq:linearstabilitycond1} is identical to the bifurcation condition \eqref{eqn:W_tilde_equal_0}. For many problems, linear stability of a stationary state $\bar\rho_\infty$ is equivalent to the spectrum of $D_{\bar\rho} \bar {\mathcal G} (\bar\rho_\infty,\kappa)$ being located in the left complex half-space (see the discussion in \cite[Sec. I.7]{Kielhofer2012}). For some equations, for example semilinear equations, it is even possible to get nonlinear stability out of the study of the spectrum of  $D_{\bar\rho} \bar {\mathcal G} (\bar\rho_\infty,\kappa)$. However, as \eqref{eqn:2} is both nonlinear and non-local, we cannot apply this general theory. The rigorous connection between the eigenvalues and linear stability for \eqref{eqn:2} is provided in the following lemma. 

\begin{lemma}\label{rem:linear} The linear asymptotic stability of $\rho_\infty$ in $L^2\big(\mathbb T^d  \times \R_+\big)$ is lost at the smallest $\kappa$ for which there exists a $k\in\N^d$ such that $\tfrac{\tilde W (k)}{\Theta(k)}$ satisfies \eqref{eqn:W_tilde_equal_0}.
\end{lemma}
\begin{proof}
The result follows by checking that the right hand side of \eqref{eq:linearstabilitycond1} is identical the right hand side of \eqref{eqn:W_tilde_equal_0}. Remember that $\int_0^\infty \rho_\infty \diff s = \frac{1}{L^d}$. Then, the denominator of the right hand side in \eqref{eq:linearstabilitycond1} is
\begin{align*}
 L^{\frac d2} \kappa \int_0^\infty (s-L^d\bar{\rho}_\infty)^2 \rho_\infty \diff s  & = L^{\frac d2} \kappa \left( \int_0^\infty s^2 \rho_\infty \diff s - 2L^d \bar{\rho}_\infty \int_0^\infty s \rho_\infty \diff s + L^{2d}\bar{\rho}_\infty ^2 \int_0^\infty \rho_\infty \diff s\right) \\ 
 & =L^{\frac d2} \left(\kappa \int_0^\infty s^2 \rho_\infty \diff s - L^d \kappa \bar{\rho}_\infty^2 \right) \\
 & = L^{\frac d2}\left( 2 \int_0^\infty \tfrac{\kappa}{2}(s- \Phi_0^\kappa)^2 \rho_\infty \diff s + \kappa 2 \Phi_0^\kappa\bar{\rho}_\infty -\frac{\kappa}{L^d} (\Phi_0^\kappa)^2 - L^d\kappa \bar{\rho}_\infty^2\right).
\end{align*}
Remembering the expression for $\rho_\infty$ in \eqref{eq:stationarystateequ} and for $\rho_\infty(0)$ in \eqref{eqn:rho_inf_0}, we compute the integral as follows: 
\begin{align*}
 2 \int_0^\infty \tfrac{\kappa}{2}(s- \Phi_0^\kappa)^2 \rho_\infty \diff s & = \sqrt{\frac{2}{\kappa}} \frac{1}{2} \left(\frac{\sqrt{\pi}}{Z_\rho} \left(1 + \erf \left(\tfrac{\sqrt{\kappa} \Phi_0^\kappa}{\sqrt{2}} \right) \right) - 2\sqrt{\frac{\kappa}{2}} \Phi_0^\kappa \rho_\infty (0) \right) \\
 & =  \sqrt{\frac{2}{\kappa}} \frac{1}{2}\left(\frac{\sqrt{2\kappa}}{L^d} - 2\sqrt{\frac{\kappa}{2}} \Phi_0^\kappa \kappa \left(\bar{\rho}_\infty - \frac{ \Phi_0^\kappa}{L^d} \right) \right) \\
 & = \frac{1}{L^d} -  \Phi_0^\kappa \kappa \left(\bar{\rho}_\infty - \frac{ \Phi_0^\kappa}{L^d} \right).
\end{align*}
Thus, 
\begin{align*}
 L^{\frac d2} \kappa \int_0^\infty (s-\bar{\rho}_\infty)^2 \rho_\infty \diff s 
 & = L^{\frac d2}\left(\frac{1}{L^d} -  \Phi_0^\kappa \kappa \left(\bar{\rho}_\infty - \frac{ \Phi_0^\kappa}{L^d} \right)  + \kappa 2 \Phi_0^\kappa\bar{\rho}_\infty -\frac{\kappa}{L^d} (\Phi_0^\kappa)^2 -L^d\kappa \bar{\rho}_\infty^2\right) \\
 & = L^{\frac d2}\left(\frac{1}{L^d} - L^d\bar{\rho}_\infty \left(\bar{\rho}_\infty-\frac{ \Phi_0^\kappa}{L^d}\right) \kappa \right),
\end{align*}
which, when multiplied by $ (\Phi_0^\kappa)'$ is identical to the denominator of the right hand side of \eqref{eqn:W_tilde_equal_0}.
\end{proof}

Note that under Assumption \ref{as:1}, if $-W$ is H-stable, $\rho_\infty$ is a linearly stable stationary state of \eqref{eqn:2} for any $\kappa>0$.

\section{Bifurcation analysis}\label{sec:bifurcations}
This section is devoted to proving the main Theorem \ref{thm:mainmain}. We start by assuming that the connectivity $W$ is coordinate-wise even on the square, and look for coordinate-wise even, non-homogeneous in space stationary states following \cite{CGPS20}. In Section \ref{sec:high_dim} we consider possible additional symmetries of $W$. But first we provide results concerning the asymptotic behaviour of the homogeneous in space stationary state.

\subsection{Asymptotic behaviour of the homogeneous in space stationary state}

\begin{lemma}\label{lm:kappa_small}
	Grant Assumption \ref{as:1}. Then, we have
	 \[\forall\kappa\in\left(0,\dfrac{2|W_0|^2}{L^{2d}\pi B^2}\right], \qquad \bar{\rho}_\infty^\kappa = \frac1{L^d}\sqrt{\dfrac{2}{\kappa\pi}}, \qquad \bar{\rho}_\infty^\kappa \geqslant \dfrac{B}{|W_0|}, \qquad  \Phi_0^\kappa =  (\Phi_0^{\kappa})' = 0. \]
	and
	 \[ \forall\kappa\in\left(\dfrac{2|W_0|^2}{L^{2d}\pi B^2},+\infty\right), \qquad \bar{\rho}_\infty^\kappa < \dfrac{B}{|W_0|}. \]
\end{lemma}

\begin{proof}
	For a given value $\kappa$, by Proposition 3.1 of \cite{CHS}, we know that there exists a unique stationary state which is constant in space and thus a unique associated mean value $\bar\rho_\infty^\kappa$.
	If
	\begin{equation}\label{condkappa} \kappa \leqslant \dfrac{2|W_0|^2}{L^{2d}\pi B^2},      \end{equation}
	then we have
	\[   \frac1{L^d}\sqrt{\dfrac{2}{\kappa\pi}}  \geqslant \dfrac{B}{|W_0|}    \qquad \mathrm{and} \qquad  \Phi\left(W_0 \frac1{L^d}\sqrt{\dfrac{2}{\kappa\pi}}  + B \right) = 0. \]
	Then, if $ \Phi_0^\kappa=0$,
	\[  \bar \rho_\infty^\kappa = \dfrac{\displaystyle \int_0^{+\infty} s \e^{-\frac\kappa2 s^2}\diff s}{L^d\sqrt{\frac{\pi}{2\kappa}}(1+\mathrm{erf}(0))} = \dfrac{1}{L^d\kappa} \sqrt{\dfrac{2\kappa}{\pi}} = \frac1{L^d}\sqrt{\dfrac{2}{\kappa \pi}}.    \]
	Hence, if \eqref{condkappa} holds, then the unique constant in space smooth stationary state satisfies
	\[  \bar \rho_\infty^\kappa = \frac1{L^d}\sqrt{\dfrac{2}{\kappa\pi}}.  \]

	Now, let $\kappa\in \left( \tfrac{2|W_0|^2}{L^{2d}\pi B^2},+\infty\right)$. Assume that 
	\[ \bar{\rho}_\infty^\kappa \geqslant \dfrac{B}{|W_0|}.  \]
	Then $ \Phi_0^\kappa=0$ and we have
	\[   \bar \rho_\infty^\kappa =\frac1{L^d} \sqrt{\dfrac{2}{\kappa \pi}} = \frac1{L^d}\sqrt{\dfrac2\pi} \sqrt{\dfrac{1}{\kappa}} <\frac1{L^d} \sqrt{\dfrac2\pi} \sqrt{\dfrac{L^{2d}\pi B^2}{2 |W_0|^2}} = \dfrac{B}{|W_0|}, \]
	which is a contradiction.
\end{proof}

\begin{lemma}\label{lm:der_bar_rho}
    Grant Assumptions \ref{as:1}. Then the function $\kappa\mapsto \bar\rho_\infty$ is differentiable and for all $\kappa\in(0,+\infty)$,
    \[ \dfrac{\diff \bar\rho_\infty}{\diff\kappa}(\kappa)  = -  \dfrac{1}{2\kappa } \dfrac{(1 + L^d \Phi_0^\kappa \bar\rho_\infty\kappa)\left(\bar \rho_\infty - \dfrac{ \Phi_0^\kappa}{L^d}\right)}{\left(1 -  (\Phi_0^\kappa)' \left(\dfrac1{L^d} - L^d\bar\rho_\infty\left(\bar\rho_\infty-\dfrac{ \Phi_0^\kappa}{L^d}\right)\kappa\right)W_0\right)}.    \]
\end{lemma}

\begin{proof}
    The mean $\bar\rho_\infty$ of the constant stationary state is a solution to
    \begin{equation*}
        \tilde G(\bar\rho_\infty,\kappa) = 0,
    \end{equation*}
    where for all $y\in \R_+^*$
    \[  \tilde G(y,\kappa) = y -\dfrac{1}{Z_\rho} \int_0^{+\infty} s \e^{ -\kappa \frac{ \left( s - \Phi(W_0 y + B) \right)^2}{2}}\diff s, \qquad Z_\rho =  L^d \int_{0}^{+\infty} \e^{ -\kappa \frac{ \left( s - \Phi(W_0 y + B) \right)^2}{2}} \diff s.  \]
    We have
    \begin{equation*}
        \dfrac{\partial \tilde G}{\partial y}(\bar\rho_\infty,\kappa) = \left(1 -  (\Phi_0^\kappa)' \left(\dfrac1{L^d} - L^d\bar\rho_\infty\left(\bar\rho_\infty-\dfrac{ \Phi_0^\kappa}{L^d}\right)\kappa\right)W_0\right),
    \end{equation*}
    and
    \begin{equation*}
         \dfrac{\partial \tilde G}{\partial \kappa}(\bar\rho_\infty,\kappa) =  \dfrac{\partial \bar{\mathcal G}}{\partial \kappa}(\bar\rho_\infty,\kappa) = \dfrac{1}{2\kappa } (1 + L^d \Phi_0^\kappa \bar\rho_\infty\kappa)\left(\bar \rho_\infty - \dfrac{ \Phi_0^\kappa}{L^d}\right).
    \end{equation*}
    By the implicit function theorem, if
    \[ \dfrac{\partial \tilde G}{\partial y}(\bar\rho_\infty,\kappa) \neq 0, \]
    then
    \begin{equation*}
        \dfrac{\diff \bar\rho_\infty}{\diff\kappa}(\kappa) = -\dfrac{\dfrac{\partial \tilde G}{\partial \kappa}(\bar\rho_\infty,\kappa)}{\dfrac{\partial \tilde G}{\partial y}(\bar\rho_\infty,\kappa)}.
    \end{equation*}
    But by Assumption \ref{as:1}, $W_0<0$ and by Lemma \ref{lm:g_eta},
    \[ \dfrac1{L^d} - L^d\bar\rho_\infty\left(\bar\rho_\infty-\dfrac{ \Phi_0^\kappa}{L^d}\right)\kappa > 0,  \]
    so $\dfrac{\partial \tilde G}{\partial y}(\bar\rho_\infty,\kappa) > 0$.
    Hence the result.
\end{proof}

\begin{lemma}\label{lm:inf_bound_rho}
	Grant Assumption \ref{as:1}. Then, there exists a value $\rho^*\in\R_+^*$ such that for all $\kappa\in\R_+^*$, $\bar\rho_\infty^\kappa \geqslant \rho^*$ and $\lim_{\kappa\to+\infty} \bar\rho_\infty^\kappa = \rho^*$.
\end{lemma}
 
\begin{proof}
	Assume there exists a sequence $\kappa_n$ such that
	\[ \lim_{n\to+\infty} \bar\rho_\infty^{\kappa_n} = 0.  \]
	Then we have
	\[ \lim_{n\to+\infty}  \Phi_0^{\kappa_n} = \Phi(B) > 0.  \]
	But recall that
	\[   \rho_\infty^{\kappa_n}(0) = \left(\bar\rho_\infty^{\kappa_n}-\dfrac{ \Phi_0^{\kappa_n}}{L^d}\right)\kappa_n.  \]
	For $n$ large enough, $ \rho_\infty^{\kappa_n}(0) < 0 $ which is impossible. Moreover, by Lemma \ref{lm:der_bar_rho} the function $\kappa\mapsto \bar\rho_\infty^\kappa$ is decreasing.
	Hence the result.
\end{proof}
 
\begin{lemma}\label{lm:equiv}
	Grant Assumption \ref{as:1}. Then, the quantity $\bar \rho_\infty - \dfrac{ \Phi_0^\kappa}{L^d}$
	converges exponentially fast towards 0 when $\kappa$ tends to $+\infty$. Moreover, there exists $\Phi_* > 0$ such that $\lim_{\kappa\to+\infty} \Phi_0^\kappa(\kappa) =  \Phi_*$ and \[ \bar \rho_\infty - \dfrac{ \Phi_0^\kappa}{L^d}\,\underset{\kappa\to+\infty}{\sim}\,
	\dfrac{\e^{-\frac{\kappa}2\Phi_*^2}}{L^d\sqrt{2\pi\kappa}}. \]
\end{lemma}

\begin{proof}
	By \eqref{eqn:rho_bar_eta},
	\begin{equation}\label{eqn:equivalent}
	    \bar \rho_\infty - \dfrac{ \Phi_0^\kappa}{L^d} = \dfrac{\sqrt2 \e^{-\frac\kappa2  (\Phi_0^\kappa)^2}}{ L^d \sqrt{\pi\kappa} \left(1+\mathrm{erf}\left(\tfrac{ \Phi_0^\kappa\sqrt{\kappa}}{\sqrt2}\right)\right)}.
	\end{equation}
	Since $ \Phi_0^\kappa$ is bounded, we have
	\begin{equation}\label{eqn:first_conv_0}  \lim_{\kappa\to+\infty} \left(\bar \rho_\infty - \dfrac{ \Phi_0^\kappa}{L^d}\right) = 0.  \end{equation}
	If $\liminf_{\kappa \to +\infty}  \Phi_0^\kappa = 0$,
	then $\liminf_{\kappa \to +\infty} \bar\rho_\infty = 0$, which contradicts Lemma \ref{lm:inf_bound_rho}. Therefore, the convergence in \eqref{eqn:first_conv_0} is exponentially fast.
	
	Last, by Lemma \ref{lm:der_bar_rho}, $\bar \rho_\infty$ is decreasing, thereby $ \Phi_0^\kappa$ is non-increasing. But we also know that $ \Phi_0^\kappa$ is bounded from above by $\Phi(B)$. Hence, there exists a positive limit $\Phi_* > 0$ for $ \Phi_0^\kappa$ when $\kappa\to+\infty$. Using this fact and the properties of the error function in \eqref{eqn:equivalent}, we obtain the desired equivalent.
\end{proof}

\begin{remark}
    Since $\Phi_* = L^d \rho^*$, given a function $\Phi$ it may be possible to find those values explicitly. For example, if $\Phi(x)= \max(x,0)$ is the ReLU function, then
    \[ W_0 \rho^* + B = L^d \rho^*,   \]
    and thus
    \[ \rho^* = \dfrac{B}{L^d+|W_0|},\qquad \Phi_* = \dfrac{L^d B}{L^d+|W_0|}.    \]
\end{remark}

\subsection{Fréchet derivatives of the functional}
In order to establish the existence of bifurcating branches from the branch of the spatially homogeneous stationary state, we need to compute a few Fréchet derivatives. 
\begin{lemma}
	Grant Assumption \ref{as:1}. Then $\bar G$ is $C^3$ Fréchet differentiable on $L^2(\mathbb T^d)\times (0,+\infty)$. Let $h_1\in L^2(\mathbb T^d)$, the first order Fréchet derivatives of $\bar{\mathcal G}$ at a point $(\bar \rho , \kappa)$ are
	\begin{equation*}
		D_{\bar \rho} \bar{\mathcal G}(\bar \rho 	,\kappa)[h_1]\ =\ h_1 - \Phi'_{\bar\rho}  \left (\dfrac 1 {L^d} - L^d\big( \bar \rho-\bar{\mathcal G}(\bar \rho ,\kappa)\big) \left(\bar \rho - \dfrac{\Phi_{\bar\rho}}{L^d}  \right)\kappa\right)  \, W \ast h_1
	\end{equation*}
	and
	\begin{equation}\label{eqn:D_kappa_full}
		D_\kappa \bar{\mathcal G}(\bar\rho ,\kappa) =\dfrac{1}{2\kappa }\bigg(  \bar{\mathcal G}(\bar \rho,\kappa) 	+ \left(1 + L^d\Phi_{\bar\rho}\big(\bar\rho -\bar{\mathcal G}(\bar\rho ,\kappa)\big)\kappa \right)\left(\bar \rho - \dfrac{\Phi_{\bar\rho}}{L^d}\right) \bigg) .
	\end{equation}
	Let $\rho_\infty^\kappa$ be a smooth constant in space steady state of \eqref{eqn:2} associated with a parameter $\kappa$. Then,
	\begin{equation*}
		D_{\bar \rho} \bar{\mathcal G}(\bar \rho_\infty^\kappa ,\kappa)[h_1]\ =\ h_1 -  (\Phi_0^\kappa)'  \left (\dfrac 1 {L^d} - L^d\bar \rho_\infty^\kappa  \left(\bar \rho_\infty^\kappa - \dfrac{ \Phi_0^\kappa}{L^d}\right)\kappa \right)  \, W \ast h_1,
	\end{equation*}
	and
	\begin{equation*}
		D_\kappa \bar{\mathcal G}(\bar \rho_\infty^\kappa,\kappa) = \dfrac{1}{2\kappa }\left(1 + L^d \Phi_0^\kappa\bar \rho_\infty^\kappa\kappa \right)\left(\bar \rho_\infty^\kappa - \dfrac{ \Phi_0^\kappa}{L^d}\right) .
	\end{equation*}
	
\end{lemma}

\begin{proof}
	 To prove smoothness of the functional, let us remark first that by Taylor expansion of $\Phi$, which is $C^3$, we have for all $h\in L^2(\mathbb T^d)$,
    \[  \Phi(W\ast(\bar\rho + h) + B) =  \Phi_{\bar\rho} + \Phi_{\bar\rho}' W\ast h + \frac12\Phi_{\bar\rho}'' (W\ast h )^2 + \frac16\Phi_{\bar\rho}''' (W \ast h)^3 + o\left( \norme{ (W \ast h)^3  }  \right).\]
    Note then, that for all $i\in\N$, $i\geq 1$, since $W, W\ast h \in H^k(\mathbb T^d) \subset L^\infty(\mathbb T^d)$ and 
    \[ \norme{(W \ast h)^i}_{L^2}  = \sqrt{ \int_{\mathbb T^d} \left(\int_{\mathbb T^d} W(x-y) h(y) \diff y\right)^{2i} \diff x  } \leqslant \norme{W}_{L^\infty}^i \norme{h}_{L^1}^i \leqslant L^{\frac {id}2} \norme{W}_{L^\infty}^i \norme{h}_{L^2}^i. \]
    Hence the map $\bar\rho\mapsto \Phi(W\ast\bar\rho + B)$ is $C^3$ Fréchet differentiable.
 Denote $J(\bar\rho,\kappa | s) = \e^{ -\kappa \frac{ \left( s - \Phi_{\bar\rho} \right)^2}{2}}$.
    Since, by Sobolev embedding and properties of the convolution, $W\ast \bar \rho \in C^0(\mathbb T^d)$ and
    \[ \norme{W\ast \bar\rho}_{L^\infty} \leq \norme{W}_{L^\infty} \norme{ \bar \rho}_{L^1} \leqslant L^{\frac d2}\norme{W}_{L^\infty} \norme{ \bar \rho}_{L^2}, \]
    then for any $s,\kappa\in\R_+$, the functional $J(\cdot,\kappa|s)$ is locally bounded in $L^2(\mathbb T^d)$ and thus the functionals
    \[ (\bar\rho,\kappa) \mapsto \int_0^{+\infty} s J(\bar\rho,\kappa)\diff s \quad \mathrm{and} \quad (\bar\rho,\kappa) \mapsto \int_{0}^{+\infty} J(\bar\rho,\kappa) \diff s \]
    are Fréchet differentiable by dominated convergence. Hence, $\bar{\mathcal G}$ is $C^3$ Fréchet differentiable by arithmetic operations where all products and quotients involve strictly positive $C^3$ functions.

	Note that
	\[   \dfrac{1}{Z_\rho} \int_0^{+\infty} s \exp\left( - \kappa\dfrac{ \big(s -   \Phi(  W \ast \bar \rho + B )\big)^2 }{2} \right)\diff s =  \bar \rho-\bar{\mathcal G}(\bar \rho ,\kappa). \]

	Let $h_1\in L^2(\mathbb T^d)$. 
	The first derivative in $\bar \rho$ is
	\begin{multline*}
		D_{\bar \rho} \bar{\mathcal G}(\bar \rho 	,\kappa)[h_1]\ =\ h_1 - \dfrac{\Phi_{\bar\rho}'}{Z_\rho} \int_{0}^{+\infty} s e^{- \kappa\frac{ (s -   \Phi_{\bar\rho})^2 }{2}} \kappa (s - \Phi_{\bar\rho} ) \diff s\, W \ast h_1 \\ + L^d( \bar \rho-\bar{\mathcal G}(\bar \rho ,\kappa)) \dfrac{\Phi_{\bar\rho}'}{Z_\rho}  \int_0^{+\infty} e^{- \kappa\frac{ (s -   \Phi_{\bar\rho})^2 }{2}}  \kappa (s - \Phi_{\bar\rho} )  \diff s\, W \ast h_1
	\end{multline*}
	By integration by parts 
	\[ \int_{0}^{+\infty} s e^{- \kappa\frac{ (s -   \Phi_{\bar\rho})^2 }{2}} \kappa (s - \Phi_{\bar\rho} ) \diff s = \dfrac{Z_\rho} {L^d},  \]
	and by direct integration
	\[ \int_0^{+\infty} e^{- \kappa\frac{ (s -   \Phi_{\bar\rho})^2 }{2}}  \kappa (s - \Phi_{\bar\rho} )  \diff s = e^{-\frac\kappa2 \Phi_{\bar\rho}^2}.    \]
	Hence, using that
	\[ \dfrac{e^{-\frac\kappa2 \Phi_{\bar\rho}^2}}{Z_\rho} =   \rho(x,0) = \left(\bar \rho(x) - \dfrac{\Phi_{\bar\rho}(x)}{L^d}  \right)\kappa,   \]
	we obtain
	\begin{equation*}
		D_{\bar \rho} \bar{\mathcal G}(\bar \rho 	,\kappa)[h_1]\ =\ h_1 - \Phi_{\bar\rho}'  \left (\dfrac 1 {L^d} - L^d( \bar \rho-\bar{\mathcal G}(\bar \rho ,\kappa)) \left(\bar \rho - \dfrac{\Phi_{\bar\rho}}{L^d}  \right)\kappa\right)  \, W \ast h_1
	\end{equation*}

	The first derivative in $\kappa$ is
	\begin{align*}
		D_\kappa \bar{\mathcal G}(\bar\rho ,\kappa) =& \dfrac{1}{Z_\rho}\bigg(  \int_0^{+\infty}  \dfrac s 2 \e^{- \kappa\frac{ (s -   \Phi_{\bar\rho})^2 }{2}} \big(s -   \Phi_{\bar\rho}\big)^2 \diff s \\
  &-\ L^d\big(\bar\rho -\bar{\mathcal G}(\bar\rho ,\kappa)\big) \int_{0}^{+\infty}\dfrac 12 \e^{- \kappa\frac{ (s -   \Phi_{\bar\rho})^2 }{2}} \big(s -   \Phi_{\bar\rho}\big)^2 \diff s \bigg).
	\end{align*}
		We can write it in the form
	\begin{equation*}
		D_\kappa \bar{\mathcal G}(\bar \rho,\kappa)=  	\dfrac{1}{2Z_\rho}\int_0^{+\infty}   e^{- \kappa\frac{ (s -   \Phi_{\bar\rho}(x))^2 }{2}} (s -   \Phi_{\bar\rho}(x))^2  \big(s- L^d \big(\bar\rho -\bar{\mathcal G}(\bar\rho ,\kappa)\big) \big)\diff s .
	\end{equation*}
	By integration by part, we have
	\begin{align*}
		D_\kappa \bar{\mathcal G}(\bar \rho,\kappa) = \, &  \dfrac{1}{2\kappa Z_\rho}\bigg( \left[ -e^{- \kappa\frac{ (s -   \Phi_{\bar\rho}(x))^2 }{2}} (s -   \Phi_{\bar\rho}(x))\big(s-L^d\big(\bar\rho -\bar{\mathcal G}(\bar\rho ,\kappa)\big)\big)   \right]^{+\infty}_0\\ &
		+  \int_0^{+\infty}   e^{- \kappa\frac{ (s -   \Phi_{\bar\rho}(x))^2 }{2}} \big(2s -   \Phi_{\bar\rho}(x)- L^d \big(\bar\rho -\bar{\mathcal G}(\bar\rho ,\kappa)\big) \big)\diff s \bigg).
	\end{align*}
	Moreover,
	\begin{align*}
		\dfrac{1}{Z_\rho}\bigg[ -e^{- \kappa\frac{ (s -   \Phi_{\bar\rho}(x))^2 }{2}} \big(s -   \Phi_{\bar\rho})&(s-L^d\big(\bar\rho -\bar{\mathcal G}(\bar\rho ,\kappa)\big)\big)   \bigg]^{+\infty}_0 \\ &=   L^d \Phi_{\bar\rho}(x)  \big(\bar\rho -\bar{\mathcal G}(\bar\rho ,\kappa)\big)\frac{e^{-\frac\kappa2  \Phi_{\bar\rho}(x)^2}}{Z_\rho} \\
		& =  L^d \Phi_{\bar\rho}(x)  \big(\bar\rho -\bar{\mathcal G}(\bar\rho ,\kappa)\big) \rho(x,0)\\
		& =  L^d\Phi_{\bar\rho}(x)  \big(\bar\rho(x) -\bar{\mathcal G}(\bar\rho ,\kappa)(x)\big)\left(\bar \rho(x) - \dfrac{\Phi_{\bar\rho}(x)}{L^d}  \right)\kappa,
	\end{align*}
	and using the definition of $Z_\rho$,
	\begin{align*} 
		\dfrac{1}{Z_\rho}\int_0^{+\infty}   e^{- \kappa\frac{ (s -   \Phi_{\bar\rho})^2 }{2}} \big(2s -   \Phi_{\bar\rho}-  L^d\big(\bar\rho -\bar{\mathcal G}(\bar\rho ,\kappa)\big) \big)\diff s &= 2 \bar\rho - \dfrac{1}{L^d} \big( \Phi_{\bar\rho} + L^d\big(\bar\rho -\bar{\mathcal G}(\bar\rho ,\kappa)\big)\big) \\
  &=  \bar\rho - \dfrac{\Phi_{\bar\rho}}{L^d} +\bar{\mathcal G}(\bar\rho ,\kappa).
	\end{align*}
	We obtain
	\begin{align*}
		D_\kappa \bar{\mathcal G}(\bar \rho,\kappa) & =  	\dfrac{1}{2\kappa }\bigg( \bar{\mathcal G}(\bar \rho,\kappa)  + \bar \rho - \dfrac{\Phi_{\bar\rho}}{L^d} + 
		L^d\Phi_{\bar\rho} \big(\bar\rho -\bar{\mathcal G}(\bar\rho ,\kappa)\big) \left(\bar \rho - \dfrac{\Phi_{\bar\rho}}{L^d}\right)\kappa \bigg)\\
		& = \dfrac{1}{2\kappa }\bigg(  \bar{\mathcal G}(\bar \rho,\kappa) 	+ \left(1 + L^d\Phi_{\bar\rho}\big(\bar\rho -\bar{\mathcal G}(\bar\rho ,\kappa)\big)\kappa \right)\left(\bar \rho - \dfrac{\Phi_{\bar\rho}}{L^d}\right) \bigg) .
	\end{align*}

	If we apply the first order derivatives on the constant stationary state $(\bar \rho_\infty^\kappa,\kappa)$ we obtain
	\begin{equation*}
		D_{\bar \rho} \bar{\mathcal G}(\bar \rho_\infty^\kappa ,\kappa)[h_1]\ =\ h_1 -  (\Phi_0^\kappa)'  \left (\dfrac 1 {L^d} - L^d\bar \rho_\infty^\kappa  \left(\bar \rho_\infty^\kappa - \dfrac{ \Phi_0^\kappa}{L^d}\right)\kappa \right)  \, W \ast h_1,
	\end{equation*}
	and
	\begin{equation*}
		D_\kappa \bar{\mathcal G}(\bar \rho_\infty^\kappa,\kappa) = \dfrac{1}{2\kappa } \left(1 + L^d \Phi_0^\kappa\bar \rho_\infty^\kappa\kappa \right)\left(\bar \rho_\infty^\kappa - \dfrac{ \Phi_0^\kappa}{L^d}\right).
	\end{equation*}
\end{proof}

We continue by differentiating the formula \eqref{eqn:D_kappa_full} with respect to $\bar\rho$ in order to obtain the cross derivative $D^2_{\bar\rho\kappa} \bar{\mathcal G}(\bar \rho,\kappa)$.

\begin{lemma}\label{lm:cross}
	For all $h_1\in L^2(\mathbb T^d)$, the second order Fréchet cross derivative of $\bar{\mathcal G}$ at a point $(\bar \rho , \kappa)$ is
	\begin{align*}
		D^2_{\bar\rho\kappa} \bar{\mathcal G}(\bar \rho,\kappa) [h_1] = \, & \dfrac1{2\kappa} \bigg(  D_{\bar \rho} \bar{\mathcal G}(\bar \rho,\kappa)[h_1]
		+ \left(1+ L^d \Phi_{\bar\rho}(\bar \rho - \bar{\mathcal G}(\bar \rho,\kappa))\kappa\right)\left(h_1 - \dfrac{\Phi'_{\bar\rho}}{L^d} W\ast h_1\right) \\
		 & +L^d\kappa\left( \Phi_{\bar\rho}' (\bar \rho - \bar{\mathcal G}(\bar \rho,\kappa)) W\ast h_1 + \Phi_{\bar\rho}(h_1-D_{\bar \rho} \bar{\mathcal G}(\bar \rho,\kappa)[h_1]) \right)\left(\bar\rho - \dfrac{\Phi_{\bar\rho}}{L^d}\right) \bigg).
	\end{align*}
	Let $\rho_\infty^\kappa$ be a smooth constant in space steady state of \eqref{eqn:2} associated with a parameter $\kappa$. Then,
	\begin{align*}
		D^2_{\bar\rho\kappa} \bar{\mathcal G}(\bar \rho_\infty^\kappa,\kappa) [h_1] = \,& \dfrac1{2\kappa} \bigg( D_{\bar \rho} \bar{\mathcal G}(\bar \rho_\infty^\kappa,\kappa)[h_1]
		+ \left(1+ L^d \Phi_0^\kappa\bar \rho_\infty^\kappa\kappa\right)\left(h_1 - \dfrac{\Phi'_0}{L^d} W\ast h_1\right) \\
		& +L^d\kappa\left(  (\Phi_0^\kappa)' \bar \rho_\infty^\kappa W\ast h_1 +  \Phi_0^\kappa(h_1-D_{\bar \rho} \bar{\mathcal G}(\bar \rho_\infty^\kappa,\kappa)[h_1]) \right)\left(\bar \rho_\infty^\kappa - \dfrac{ \Phi_0^\kappa}{L^d}\right) \bigg).
	\end{align*}
\end{lemma}

In the same way, we can compute the second and third order derivatives in $\bar\rho$. Denote
\begin{align*}
   \Phi''_{\bar\rho} &= \Phi''( W\ast \bar\rho + B ), \qquad\qquad\qquad (\Phi_0^\kappa)'' = \Phi''( W_0 \bar\rho_\infty^\kappa + B ),\\ \Phi'''_{\bar\rho}  &= \Phi'''( W\ast \bar\rho + B ),\qquad \mathrm{and} \qquad (\Phi^\kappa_0)''' = \Phi'''( W_0 \bar\rho_\infty^\kappa + B ).
\end{align*}
We have the following lemma.

\begin{lemma}\label{lm:third}
	Let $\rho_\infty^\kappa$ be a smooth constant in space steady state of \eqref{eqn:2} associated with a parameter $\kappa$. Then, for all $h_1, h_2, h_3\in L^2(\mathbb T^d)$,
	\begin{align*}
	   D^2_{\bar\rho \bar\rho} \bar{\mathcal G}(\bar \rho_\infty^\kappa,&\kappa) [h_1, h_2] \\
	   = &\, - (\Phi^\kappa_0)''  \left (\dfrac 1 {L^d} - L^d\bar \rho_\infty^\kappa \left(\bar \rho_\infty^\kappa - \dfrac{ \Phi_0^\kappa}{L^d}  \right)\kappa\right)  \, W \ast h_1  \, W \ast h_2 \\
	   & + (\Phi^\kappa_0)' L^d \left(  
	   \big(h_2 - D_{\bar\rho} \bar{\mathcal G}(\bar \rho_\infty^\kappa,\kappa)[h_2]\big) \left(\bar \rho_\infty^\kappa - \dfrac{ \Phi_0^\kappa}{L^d} \right) + \bar \rho_\infty^\kappa \left(h_2 - \dfrac{ (\Phi_0^\kappa)'}{L^d}\, W\ast h_2 \right)
	   \right)\kappa\, W \ast h_1,
	\end{align*}
	and
	\begin{align*}
	  D^3_{\bar\rho \bar\rho  \bar\rho} &\bar{\mathcal G}(\bar \rho_\infty^\kappa,\kappa) [h_1, h_2, h_3] \\
	  = &\, - (\Phi^\kappa_0)'''  \left (\dfrac 1 {L^d} - L^d\bar \rho_\infty^\kappa \left(\bar \rho_\infty^\kappa - \dfrac{ \Phi_0^\kappa}{L^d}  \right)\kappa\right)  \, W \ast h_1  \, W \ast h_2 \, W \ast h_3 \\
	   &+ (\Phi^\kappa_0)'' L^d \left(  
	   \big(h_3 - D_{\bar\rho} \bar{\mathcal G}(\bar \rho_\infty^\kappa,\kappa)[h_3]\big) \left(\bar \rho_\infty^\kappa - \dfrac{ \Phi_0^\kappa}{L^d} \right) + \bar \rho_\infty^\kappa\left(h_3 - \dfrac{ (\Phi_0^\kappa)'}{L^d}\, W\ast h_3 \right)
	   \right)\kappa\, W \ast h_1  \, W \ast h_2\\
	   	& + (\Phi^\kappa_0)'' L^d \left( 
	   \big(h_2 - D_{\bar\rho} \bar{\mathcal G}(\bar \rho_\infty^\kappa,\kappa)[h_2]\big) \left(\bar \rho_\infty^\kappa - \dfrac{ \Phi_0^\kappa}{L^d} \right) 
	   + \bar \rho_\infty^\kappa \left(h_2 - \dfrac{ (\Phi_0^\kappa)'}{L^d}\, W\ast h_2 \right)
	   \right)\kappa\, W \ast h_1 \, W\ast h_3 \\
	   	 & + (\Phi^\kappa_0)' L^d \Bigg[ 
	   	  - D^2_{\bar\rho\bar \rho} \bar{\mathcal G}(\bar \rho_\infty^\kappa,\kappa)[h_2,h_3]\left(\bar \rho_\infty^\kappa - \dfrac{ \Phi_0^\kappa}{L^d} \right) +
	   \big(h_2 - D_{\bar\rho} \bar{\mathcal G}(\bar \rho_\infty^\kappa,\kappa)[h_2]\big) \left(h_3 - \dfrac{ (\Phi_0^\kappa)'}{L^d}\, W\ast h_3 \right)  \\
	   	&  +  \big(h_3 - D_{\bar\rho} \bar{\mathcal G}(\bar \rho_\infty^\kappa,\kappa)[h_3]\big)\left(h_2 - \dfrac{(\Phi^\kappa_{0})'}{L^d}\, W\ast h_2 \right) - \bar \rho_\infty^\kappa \dfrac{ (\Phi_0^\kappa)''}{L^d}\, W\ast h_2 \, W\ast h_3\Bigg]
	   \kappa\, W \ast h_1 .
	\end{align*}
\end{lemma}

\begin{proof}
For all $h_1, h_2, h_3\in L^2(\mathbb T^d)$, the second and third order Fréchet derivative in $\bar\rho$ of $\bar{\mathcal G}$ at a point $(\bar \rho , \kappa)$ are
	\begin{align*}
	   D^2_{\bar\rho \bar\rho} &\bar{\mathcal G}(\bar \rho,\kappa) [h_1, h_2] \\
	   =&\, - \Phi''_{\bar\rho}  \left (\dfrac 1 {L^d} - L^d\big( \bar \rho-\bar{\mathcal G}(\bar \rho ,\kappa)\big) \left(\bar \rho - \dfrac{\Phi_{\bar\rho}}{L^d}  \right)\kappa\right)  \, W \ast h_1  \, W \ast h_2 \\
	   &+ \Phi'_{\bar\rho} L^d \left(  
	   \big(h_2 - D_{\bar\rho} \bar{\mathcal G}(\bar \rho,\kappa)[h_2]\big) \left(\bar \rho - \dfrac{\Phi_{\bar\rho}}{L^d} \right) + \big( \bar \rho-\bar{\mathcal G}(\bar \rho ,\kappa)\big)\left(h_2 - \dfrac{\Phi_{\bar\rho}'}{L^d}\, W\ast h_2 \right)
	   \right)\kappa\, W \ast h_1,
	\end{align*}
	and
	\begin{align*}
	   D^3_{\bar\rho \bar\rho \bar\rho} & \bar{\mathcal G}(\bar \rho,\kappa) [h_1, h_2, h_3] \\
	   =& \, - \Phi'''_{\bar\rho}  \left (\dfrac 1 {L^d} - L^d\big( \bar \rho-\bar{\mathcal G}(\bar \rho ,\kappa)\big) \left(\bar \rho - \dfrac{\Phi_{\bar\rho}}{L^d}  \right)\kappa\right)  \, W \ast h_1  \, W \ast h_2 \, W \ast h_3 \\
	   &+ \Phi''_{\bar\rho} L^d \left(  
	   \big(h_3 - D_{\bar\rho} \bar{\mathcal G}(\bar \rho,\kappa)[h_3]\big) \left(\bar \rho - \dfrac{\Phi_{\bar\rho}}{L^d} \right) + \big( \bar \rho-\bar{\mathcal G}(\bar \rho ,\kappa)\big)\left(h_3 - \dfrac{\Phi_{\bar\rho}'}{L^d}\, W\ast h_3 \right)
	   \right)\kappa\, W \ast h_1  \, W \ast h_2\\
	   	& + \Phi''_{\bar\rho} L^d  
	   \big(h_2 - D_{\bar\rho} \bar{\mathcal G}(\bar \rho,\kappa)[h_2]\big) \left(\bar \rho - \dfrac{\Phi_{\bar\rho}}{L^d} \right) 
	   \kappa\, W \ast h_1 \, W\ast h_3 \\
	   	 & + \Phi'_{\bar\rho} L^d \left( 
	   	  - D^2_{\bar\rho\bar \rho} \bar{\mathcal G}(\bar \rho,\kappa)[h_2,h_3]\left(\bar \rho - \dfrac{\Phi_{\bar\rho}}{L^d} \right) +
	   \big(h_2 - D_{\bar\rho} \bar{\mathcal G}(\bar \rho,\kappa)[h_2]\big) \left(h_3 - \dfrac{\Phi_{\bar\rho}'}{L^d}\, W\ast h_3 \right) \right)
	   \kappa\, W \ast h_1  \\
	   	& + \Phi''_{\bar\rho} L^d  \big( \bar \rho-\bar{\mathcal G}(\bar \rho ,\kappa)\big)\left(h_2 - \dfrac{\Phi_{\bar\rho}'}{L^d}\, W\ast h_2 \right)
	   \kappa\, W \ast h_1 \, W\ast h_3 \\
	   &	  + \Phi'_{\bar\rho} L^d \left( \big(h_3 - D_{\bar\rho} \bar{\mathcal G}(\bar \rho,\kappa)[h_3]\big)\left(h_2 - \dfrac{\Phi_{\bar\rho}'}{L^d}\, W\ast h_2 \right) - \big( \bar \rho-\bar{\mathcal G}(\bar \rho ,\kappa)\big) \dfrac{\Phi_{\bar\rho}''}{L^d}\, W\ast h_2 \, W\ast h_3\right)
	   \kappa\, W \ast h_1 .
	\end{align*}
Letting $\bar \rho = \bar \rho_\infty ^\kappa$ yields the result.
\end{proof}

\subsection{Proof of the main results}

We are now ready to rigorously state and prove the main results stated in Theorem \ref{thm:mainmain}. For convenience, the main theorem is split into two parts: Theorem \ref{thm:main} concerning the bifurcation points, and Theorem \ref{thm:main_2} characterising the corresponding bifurcation branches. The case of higher dimensional kernels of the $\bar\rho$ first derivative and additional symmetries is discussed in Subsection \ref{sec:high_dim}.  

We define the functional
	    \begin{equation}\label{eq:onecompfunctionalH}
	        \begin{array}{rccl}
	            \mathcal H : & L_S^2(\mathbb T^d)\times \R_+^* & \to & L_S^2(\mathbb T^d)  \\
	                         & (\bar \rho,\kappa) &\mapsto & \bar {\mathcal G} (\bar \rho_\infty^\kappa + \bar\rho, \kappa).
	        \end{array}
	    \end{equation}
	    which is the one we will apply the Crandall--Rabinowitz theorem to. For completeness, the Crandall--Rabinowitz theorem in the present context is stated in the appendix (Theorem \ref{thm:CR}). Note that for all $\kappa\in\R_+^*$, $\mathcal H(0,\kappa)=0$ and that for $\kappa>\tfrac{2|W_0|^2}{L^{2d}\pi B^2}$, owing to Assumption \ref{as:1}, the functional $\mathcal H$ satisfies all the smoothness hypotheses required. The Fréchet derivatives of $\mathcal H$ are
	    \begin{align*}
	            D_{\bar\rho} \mathcal H (\bar \rho, \kappa) & =  D_{\bar\rho} \bar{\mathcal G}(\bar \rho_\infty^\kappa + \bar \rho,\kappa), \\
	            D_{\kappa} \mathcal H (\bar \rho, \kappa) & =  D_{\kappa} \bar{\mathcal G}(\bar \rho_\infty^\kappa + \bar \rho,\kappa) + D_{\bar\rho} \bar{\mathcal G}(\bar \rho_\infty^\kappa+\bar\rho,\kappa)\left[\dfrac{\diff \bar\rho_\infty}{\diff \kappa}(\kappa)\right],\\
	            D^2_{\bar\rho \kappa} \mathcal H (\bar \rho, \kappa) & =  D^2_{\bar\rho\kappa} \bar{\mathcal G}(\bar \rho_\infty^\kappa + \bar \rho,\kappa) + D^2_{\bar\rho\bar\rho} \bar{\mathcal G}(\bar \rho_\infty^\kappa + \bar \rho,\kappa) \left[\dfrac{\diff \bar\rho_\infty}{\diff \kappa}(\kappa), \cdot\right]  ,\\
	            D_{\bar \rho \bar \rho} \mathcal H (\bar \rho, \kappa) & =  D^2_{\bar\rho \bar\rho} \bar{\mathcal G}(\bar \rho_\infty^\kappa + \bar \rho,\kappa), \\
	            D_{ \bar \rho \bar \rho \bar \rho} \mathcal H (\bar \rho, \kappa) & =  D^3_{\bar\rho \bar\rho \bar\rho} \bar{\mathcal G}(\bar \rho_\infty^\kappa + \bar \rho,\kappa).
	    \end{align*}
	    In particular, we have
\begin{align}\label{eqn:der_H_0}
	            D_{\bar\rho} \mathcal H (0, \kappa) & =  D_{\bar\rho} \bar{\mathcal G}(\bar \rho_\infty^\kappa,\kappa), \nonumber \\
	            D_{\kappa} \mathcal H (0, \kappa) & =  D_{\kappa} \bar{\mathcal G}(\bar \rho_\infty^\kappa,\kappa) + D_{\bar\rho} \bar{\mathcal G}(\bar \rho_\infty^\kappa,\kappa)\left[\dfrac{\diff \bar\rho_\infty}{\diff \kappa}(\kappa)\right],\nonumber \\
	            D^2_{\bar\rho \kappa} \mathcal H (0, \kappa) & =  D^2_{\bar\rho\kappa} \bar{\mathcal G}(\bar \rho_\infty^\kappa,\kappa)+ D^2_{\bar\rho\bar\rho} \bar{\mathcal G}(\bar \rho_\infty^\kappa,\kappa) \left[\dfrac{\diff \bar\rho_\infty}{\diff \kappa}(\kappa), \cdot\right],\\
	            D_{\bar \rho \bar \rho} \mathcal H (0, \kappa) & =  D^2_{\bar\rho \bar\rho} \bar{\mathcal G}(\bar \rho_\infty^\kappa ,\kappa), \nonumber \\
	            D_{ \bar \rho \bar \rho \bar \rho} \mathcal H (0, \kappa) & =  D^3_{\bar\rho \bar\rho \bar\rho} \bar{\mathcal G}(\bar \rho_\infty^\kappa,\kappa) \nonumber .
	    \end{align}

\begin{theorem}\label{thm:main}
	Grant Assumption \ref{as:1}. Let $\kappa\in\left(\tfrac{2|W_0|^2}{L^{2d}\pi B^2},+\infty\right)$. If there exists a unique $k^*\in\N^d$ such that
    \begin{equation}\label{eqn:W_tilde_equal}
        \dfrac{\tilde W (k^*)}{\Theta(k^*)} = \dfrac{1}{L^{\frac d2} (\Phi_0^\kappa)' \left( \dfrac1{L^d} - L^d\bar \rho_\infty^\kappa  \left(\bar \rho_\infty^\kappa - \dfrac{ \Phi_0^\kappa}{L^d}\right)\kappa \right)}.
    \end{equation}
    and $ (\Phi_0^\kappa)''\geq 0$, then $(\bar\rho_\infty^\kappa,\kappa)$ is a bifurcation point of $\bar{\mathcal G}(\bar \rho_\infty^\kappa,\kappa)=0$.
\end{theorem}

\begin{proof}

		\noindent \emph{i)} We first rewrite
		\[  	D_{\bar\rho} \mathcal H(0,\kappa) = I -  T,   \]
		with $ T : L^2_S(\mathbb S) \to L^2_S(\mathbb S)$ defined by
		\[   T h_1 =  (\Phi_0^\kappa)'  \left (\dfrac1{L^d} - 
		L^d\bar \rho_\infty^\kappa  \left(\bar \rho_\infty^\kappa - \dfrac{ \Phi_0^\kappa}{L^d}\right)\kappa \right)\, W \ast h_1. \]
		First, we have to check that the operator $D_{\bar\rho} \mathcal H(0,\kappa)$ is not the identity, that is to say $T\neq 0$. We know that $ (\Phi_0^\kappa)'\neq 0$ by Assumption \ref{as:1} and Lemma \ref{lm:kappa_small}. Hence, by Lemma \ref{lm:g_eta}, $T\neq 0$.
		
		\noindent \emph{ii)} Let the basis of $L^2_S(\mathbb{T}^d )$ be as described in Section \ref{sec:basis}. By \eqref{trigo}, for all $k\in\N^d$, we have
		\begin{align*}  \norme{ T \omega_k}_2^2 & = \norme{    (\Phi_0^\kappa)'\left (\dfrac1{L^d} - 
		L^d\bar \rho_\infty^\kappa  \left(\bar \rho_\infty^\kappa - \dfrac{ \Phi_0^\kappa}{L^d}\right)\kappa \right) \dfrac{\tilde W(k)}{\Theta(k)} \omega_k(\cdot)   }_2^2\\ 
		& = \left|  (\Phi_0^\kappa)'\left (\dfrac1{L^d} - 
		L^d\bar \rho_\infty^\kappa  \left(\bar \rho_\infty^\kappa - \dfrac{ \Phi_0^\kappa}{L^d}\right)\kappa \right) \dfrac{\tilde W(k)}{\Theta(k)}\right|. \end{align*}
		Hence, the Hilbert--Schmidt norm
		 \[  \norme{ T}_{HS}^2 = \sum_{k\in\N} \norme{ T \omega_k}_2^2,\]
		is finite and $ T$ is a Hilbert--Schmidt operator. Therefore, $D_{\bar\rho} \mathcal H(0,\kappa)$ is a Fredholm operator.
		
		\noindent \emph{iii)} Note that the mapping $z\mapsto I + z W\ast\cdot$ is norm-continous; indeed, for all $z_1,z_2\in\R$,
		\begin{equation*}\norme{(I + z_1 W\ast\cdot) - (I + z_2W\ast\cdot)} = | z_1-z_2 | \norme{W\ast \cdot}. \end{equation*}
		Hence, taking \[z= - (\Phi_0^\kappa)'\left (\dfrac1{L^d} - 
		L^d\bar \rho_\infty^\kappa  \left(\bar \rho_\infty^\kappa - \dfrac{ \Phi_0^\kappa}{L^d}\right)\kappa \right),\] 
		and applying the invariance by homotopy of the index of Fredholm operators, we have 
		$\mathrm{Ind}(D_{\bar\rho} \mathcal H(0,\kappa)) = \mathrm{Ind}(I)=0$.
		
		\noindent \emph{iv)} We now diagonalise $D_{\bar \rho} \mathcal H(0,\kappa)$ with respect to the orthonormal basis defined in Section \ref{sec:basis}:
		\[
		D_{\bar\rho} \mathcal H(0,\kappa)[\omega_{k}](x) =  \left\{\begin{array}{ll}
			 1 - L^\frac d2 (\Phi_0^\kappa)'  \left(\dfrac1{L^d} - L^d\bar \rho_\infty^\kappa  \left(\bar \rho_\infty^\kappa - \dfrac{ \Phi_0^\kappa}{L^d}\right)\kappa \right)W_0,\qquad & if\ k=0,\\
		     \left( 1 - \dfrac{L^{\frac d2}F(k)}{\Theta(k)}  \left (\dfrac1{L^d} - L^d\bar \rho_\infty^\kappa  \left(\bar \rho_\infty^\kappa - \dfrac{ \Phi_0^\kappa}{L^d}\right)\kappa \right) \right) \omega_k(x), \qquad & otherwise,
		\end{array}
		\right.
		\]
		with $F(k) =  (\Phi_0^\kappa)'  \tilde W(k)$. Since $W_0<0$, and by Lemma \ref{lm:g_eta},
		\[ 1 - L^\frac d2 (\Phi_0^\kappa)'  \left(\dfrac1{L^d} - L^d\bar \rho_\infty^\kappa  \left(\bar \rho_\infty^\kappa - \dfrac{ \Phi_0^\kappa}{L^d}\right)\kappa \right)W_0 > 0. \]
		By assumption \eqref{eqn:W_tilde_equal_0}, the kernel of $D_{\bar\rho} \mathcal H(0,\kappa)$ is thus one dimensional.
		
		\noindent\emph{v)} Last, we need to check that if $v\in \mathrm{Ker}\big(D_{\bar \rho} \mathcal H(0,\kappa)\big)$ and $\norme{v}=1$, then $D^2_{\bar\rho\kappa} \mathcal H(0,\kappa) [v]\notin \mathrm{Im}\big(D_{\bar \rho} \mathcal H(0,\kappa)\big)$.
		Since $W$ is symmetric, the operator $I-T$ is self-adjunct. Then, since $\mathrm{Im}(I-T)$ is closed, we have
		\[ \mathrm{Im}(I-T) = \mathrm{Ker}( (I-T)^*)^\perp = \mathrm{Ker}(I-T)^\perp.  \]
		As a result, we just need to check that 
		\[ \pscal{D^2_{\bar\rho\kappa} \mathcal H(0,\kappa) [v]}{ v } \neq 0.\]
		According to \eqref{eqn:der_H_0},
		\begin{equation*}
		    D^2_{\bar\rho \kappa} \mathcal H (0, \kappa)[v]  =  D^2_{\bar\rho\kappa} \bar{\mathcal G}(\bar \rho_\infty^\kappa,\kappa)[v]+ D^2_{\bar\rho\bar\rho} \bar{\mathcal G}(\bar \rho_\infty^\kappa,\kappa) \left[\dfrac{\diff \bar\rho_\infty}{\diff \kappa}(\kappa), v\right].
		\end{equation*}
		Applying the formula in Lemma \ref{lm:cross}, and noticing that $v\in \mathrm{Ker}\big(D_{\bar \rho} \bar{\mathcal G}(\bar\rho_\infty^\kappa,\kappa)\big)$, we obtain
		\begin{equation*}
		D^2_{\bar\rho\kappa} \bar{\mathcal G}(\bar \rho_\infty^\kappa,\kappa) [v] = 	\dfrac{1}{2\kappa}\left(1+ L^d \Phi_0^\kappa\bar \rho_\infty^\kappa\kappa\right)\left(v - \dfrac{(\Phi^\kappa_0)'}{L^d} W\ast v\right) + \dfrac{L^d}{2} 
		\left(  (\Phi_0^\kappa)' \bar \rho_\infty^\kappa W\ast v +  \Phi_0^\kappa v \right)\left(\bar \rho_\infty^\kappa - \dfrac{ \Phi_0^\kappa}{L^d}\right).
	    \end{equation*}
		Since $v\in \mathrm{Ker}\big(D_{\bar \rho} \bar{\mathcal G}(\bar \rho_\infty,\kappa)\big)$, we have
		\[ W\ast v =  \dfrac{1}{ (\Phi_0^\kappa)' \left (\dfrac{1}{L^d} - L^d\bar \rho_\infty^\kappa\left(\bar \rho_\infty^\kappa - \dfrac{ \Phi_0^\kappa}{L^d}\right)\kappa \right)}v.    \]
		Hence,
		\begin{align}\label{eqn:brick_1}  \dfrac{1}{2\kappa}&\left(1+ L^d \Phi_0^\kappa\bar \rho_\infty^\kappa\kappa\right)\left(v - \dfrac{(\Phi^\kappa_0)'}{L^d} W\ast v\right) \nonumber \\
		& = -\dfrac{1}{2\kappa}\left(1+ L^d \Phi_0^\kappa\bar \rho_\infty^\kappa\kappa\right) \dfrac{L^{2d}\bar \rho_\infty^\kappa\left(\bar \rho_\infty^\kappa - \dfrac{ \Phi_0^\kappa}{L^d}\right)\kappa}{L^d\left(\dfrac{1}{L^d} - L^d\bar \rho_\infty^\kappa\left(\bar \rho_\infty^\kappa - \dfrac{ \Phi_0^\kappa}{L^d}\right)\kappa\right)} v\\ 
		&= -\dfrac{L^d}{2}\left(1+ L^d \Phi_0^\kappa\bar \rho_\infty^\kappa\kappa\right) \dfrac{\bar \rho_\infty^\kappa\left(\bar \rho_\infty^\kappa - \dfrac{ \Phi_0^\kappa}{L^d}\right)}{\dfrac{1}{L^d} - L^d\bar \rho_\infty^\kappa\left(\bar \rho_\infty^\kappa - \dfrac{ \Phi_0^\kappa}{L^d}\right)\kappa} v, \nonumber
			\end{align}  
		and
		\begin{align}\label{eqn:brick_2}  \dfrac{L^d}{2} & 
		\left(  (\Phi_0^\kappa)' \bar \rho_\infty^\kappa W\ast v +  \Phi_0^\kappa v \right)\left(\bar \rho_\infty^\kappa - \dfrac{ \Phi_0^\kappa}{L^d}\right)\nonumber \\ &= \dfrac{L^d}{2}\left( \dfrac{\bar\rho_\infty^\kappa}{\dfrac{1}{L^d} - L^d\bar \rho_\infty^\kappa\left(\bar \rho_\infty^\kappa - \dfrac{ \Phi_0^\kappa}{L^d}\right)\kappa} +  \Phi_0^\kappa \right)\left(\bar \rho_\infty^\kappa - \dfrac{ \Phi_0^\kappa}{L^d}\right)  v.
		\end{align}
		Collecting \eqref{eqn:brick_1} and \eqref{eqn:brick_2}, we obtain
		\begin{align*} 
  D^2_{\bar\rho\kappa} \bar{\mathcal G}(\bar \rho_\infty^\kappa,\kappa) [v] & = \dfrac{L^d}{2}\left( \dfrac{\bar\rho_\infty^\kappa - \left(1+ L^d \Phi_0^\kappa\bar \rho_\infty^\kappa\kappa\right)\bar\rho_\infty^\kappa}{\dfrac{1}{L^d} - L^d\bar \rho_\infty^\kappa\left(\bar \rho_\infty^\kappa - \dfrac{ \Phi_0^\kappa}{L^d}\right)\kappa} +  \Phi_0^\kappa \right)\left(\bar \rho_\infty^\kappa - \dfrac{ \Phi_0^\kappa}{L^d}\right)  v\\ \nonumber
		& = \dfrac{L^d \Phi_0^\kappa}{2}\left( 1 - \dfrac{ L^d(\bar \rho_\infty^\kappa)^2\kappa}{\dfrac{1}{L^d} - L^d\bar \rho_\infty^\kappa\left(\bar \rho_\infty^\kappa - \dfrac{ \Phi_0^\kappa}{L^d}\right)\kappa} \right)\left(\bar \rho_\infty^\kappa - \dfrac{ \Phi_0^\kappa}{L^d}\right)  v
		.
		\end{align*}
		On the other hand, we have
		\begin{align*}
	       D^2_{\bar\rho \bar\rho} \bar{\mathcal G}(\bar \rho_\infty^\kappa,\kappa)& \left[\dfrac{\diff \bar\rho_\infty}{\diff \kappa}(\kappa), v\right] \\ = & \, \underbrace{- (\Phi^\kappa_0)''  \left (\dfrac 1 {L^d} - L^d\bar \rho_\infty^\kappa \left(\bar \rho_\infty^\kappa - \dfrac{ \Phi_0^\kappa}{L^d}  \right)\kappa\right)  \, W_0  \dfrac{\diff \bar\rho_\infty}{\diff \kappa}(\kappa) \, W \ast v}_{=:A_1} \\
	       & + \underbrace{(\Phi^\kappa_0)' L^d \left(  
	       \big(v - D_{\bar\rho} \bar{\mathcal G}(\bar \rho_\infty^\kappa,\kappa)[v]\big) \left(\bar \rho_\infty^\kappa - \dfrac{ \Phi_0^\kappa}{L^d} \right) + \bar \rho_\infty^\kappa \left(v - \dfrac{ (\Phi_0^\kappa)'}{L^d}\, W\ast v \right)
	       \right)\kappa\, W_0 \dfrac{\diff \bar\rho_\infty}{\diff \kappa}(\kappa) }_{=:A_2},
	    \end{align*}
		where
		\begin{equation*}
		    A_1 = - \dfrac{ (\Phi_0^\kappa)''W_0}{ (\Phi_0^\kappa)'}  \dfrac{\diff \bar\rho_\infty}{\diff \kappa}(\kappa),
		\end{equation*}
		and
		\begin{gather*}
		\begin{aligned}
		    A_2 & = (\Phi^\kappa_0)' L^d \left(  
	       \left(\bar \rho_\infty^\kappa - \dfrac{ \Phi_0^\kappa}{L^d} \right)v + \bar \rho_\infty^\kappa \left(v - \dfrac{ (\Phi_0^\kappa)'}{L^d}\, W\ast v \right)
	       \right)\kappa\, W_0 \dfrac{\diff \bar\rho_\infty}{\diff \kappa}(\kappa)\\
	       & = (\Phi^\kappa_0)' L^d \left(  
	       \left(\bar \rho_\infty^\kappa - \dfrac{ \Phi_0^\kappa}{L^d} \right)v + \bar \rho_\infty^\kappa  \dfrac{L^d\bar \rho_\infty^\kappa\left(\bar \rho_\infty^\kappa - \dfrac{ \Phi_0^\kappa}{L^d}\right)\kappa}{\dfrac{1}{L^d} - L^d\bar \rho_\infty^\kappa\left(\bar \rho_\infty^\kappa - \dfrac{ \Phi_0^\kappa}{L^d}\right)\kappa}
	       \right)\kappa\, W_0 \dfrac{\diff \bar\rho_\infty}{\diff \kappa}(\kappa)\, v \\
	       & = (\Phi^\kappa_0)' L^d \left( 1 - \dfrac{L^d(\bar \rho_\infty^\kappa)^2\kappa}{\dfrac{1}{L^d} - L^d\bar \rho_\infty^\kappa\left(\bar \rho_\infty^\kappa - \dfrac{ \Phi_0^\kappa}{L^d}\right)\kappa} \right) \left(\bar \rho_\infty^\kappa - \dfrac{ \Phi_0^\kappa}{L^d}\right) W_0 \dfrac{\diff \bar\rho_\infty}{\diff \kappa}(\kappa)\, v
	       .
	       \end{aligned}
		\end{gather*}
		Therefore,
		\begin{align}\label{eqn:no_more_idea_for_labels_3}
	       D^2_{\bar\rho \kappa} \mathcal H(0,\kappa) \left[ v\right] = &  \left( \underbrace{ - \dfrac{ (\Phi_0^\kappa)''W_0}{ (\Phi_0^\kappa)'}  \dfrac{\diff \bar\rho_\infty}{\diff \kappa}(\kappa)}_{\leqslant 0} \right. 
	       \\ & \left. +  \underbrace{L^d\left(\dfrac{ \Phi_0^\kappa}{2} +  (\Phi_0^\kappa)' W_0 \dfrac{\diff \bar\rho_\infty}{\diff \kappa}(\kappa) \right)}_{>0}\! \underbrace{ \left( 1 - \dfrac{L^d(\bar \rho_\infty^\kappa)^2\kappa}{\tfrac{1}{L^d} - L^d\bar \rho_\infty^\kappa\left(\bar \rho_\infty^\kappa - \tfrac{ \Phi_0^\kappa}{L^d}\right)\kappa} \right)}_{=:A_3}\! \underbrace{\left(\bar \rho_\infty^\kappa - \dfrac{ \Phi_0^\kappa}{L^d}\right)}_{>0}   \right)  v. \nonumber
		\end{align}
		Hence, if $A_3$ is of negative sign, then $D^2_{\bar\rho\kappa} \mathcal H (0,\kappa) [v] \neq 0$. Moreover, $A_3<0$ if and only if
	    \begin{equation*}
	     L^d(\bar \rho_\infty^\kappa)^2\kappa > \dfrac{1}{L^d} - L^d\bar \rho_\infty^\kappa\left(\bar \rho_\infty^\kappa - \dfrac{ \Phi_0^\kappa}{L^d}\right)\kappa.
	    \end{equation*}
	    If we denote $\eta = \sqrt{\frac{\kappa}{2}} \Phi_0^\kappa$, then by Lemma \ref{lm:g_eta},
	    \[  \dfrac{1}{L^d} - L^d\bar \rho_\infty^\kappa\left(\bar \rho_\infty^\kappa - \dfrac{ \Phi_0^\kappa}{L^d}\right)\kappa = \frac{1}{L^d}\left( 1-\frac{2}{\sqrt{\pi}}\frac{\exp (-\eta^2)}{1+\erf (\eta)}\left[\frac{1}{\sqrt{\pi}}\frac{\exp (-\eta^2)}{1+\erf (\eta)}+\eta\right]\right) = \dfrac{g(\eta)}{L^d} ,\]
	    and by \eqref{eqn:rho_bar_eta},
	    \begin{equation*}  
     L^d (\bar\rho_\infty^\kappa)^2 \kappa = \dfrac{2}{L^d} \left( \dfrac{1}{\sqrt\pi}\frac{\exp (-\eta^2)}{1+\erf (\eta)} + \eta   \right)^2.
     \end{equation*}
	    Hence, $A_3 < 0$ if and only if	    
	    \begin{equation*}
	        w(\eta) := \left( \dfrac{1}{\sqrt\pi}\frac{\exp (-\eta^2)}{1+\erf (\eta)} + \eta    \right)\left( \dfrac{2}{\sqrt\pi}\frac{\exp (-\eta^2)}{1+\erf (\eta)} + \eta \right)  > \dfrac12,
	    \end{equation*}
	    where the function $w$ is the same as defined in \eqref{eqn:def_h}. In the proof of Lemma \ref{lm:g_eta}, we obtained that
	    \[w(0)=\dfrac{2}{\pi}>\dfrac12,\]
	    and that $w$ is an increasing function on $\R_+$.
	    Therefore $A_3<0$, which in turn implies
		\[ \pscal{D^2_{\bar\rho\kappa} \mathcal H(0,\kappa) [v]}{ v } \neq 0.\]
\end{proof}

\begin{remark}\label{rem:convexity}
    The assumption $ (\Phi_0^\kappa)'' \geq 0$ in Theorem \ref{thm:main} can be relaxed. Indeed, in \eqref{eqn:no_more_idea_for_labels_3} 
    \begin{align*}
     A_4:= - \dfrac{ (\Phi_0^\kappa)''W_0}{ (\Phi_0^\kappa)'}  \dfrac{\diff \bar\rho_\infty}{\diff \kappa}(\kappa) + L^d\left(\dfrac{ \Phi_0^\kappa}{2} +  (\Phi_0^\kappa)' W_0 \dfrac{\diff \bar\rho_\infty}{\diff \kappa}(\kappa) \right)  A_3 \left(\bar \rho_\infty^\kappa - \dfrac{ \Phi_0^\kappa}{L^d}\right) < 0,
    \end{align*}
    is sufficient for $ D^2_{\bar\rho \kappa} \mathcal H(0,\kappa) \left[ v\right] \neq 0$. By Lemma \ref{lm:g_eta} and \eqref{eqn:rho_bar_eta},
    \begin{align*}
        A_3 = 2\left(\frac{1}{2}-\frac{(f(\eta)+\eta))^2}{g(\eta)} \right) \leq 2 \left(\frac{1}{2}-\frac{f(0)^2}{g(0)} \right) = \frac{\pi-4}{\pi-2},
    \end{align*}
    as it can be checked that $\tfrac{(f(\eta)+\eta))^2}{g(\eta)}$ is an increasing function in $\eta$. Thus, 
    \begin{align*}
        A_4 & \leq - \dfrac{ (\Phi_0^\kappa)''W_0}{ (\Phi_0^\kappa)'}  \dfrac{\diff \bar\rho_\infty}{\diff \kappa}(\kappa) + L^d\left(\dfrac{ \Phi_0^\kappa}{2} +  (\Phi_0^\kappa)' W_0 \dfrac{\diff \bar\rho_\infty}{\diff \kappa} (\kappa)\right)  \left(\bar \rho_\infty^\kappa - \dfrac{ \Phi_0^\kappa}{L^d}\right) \frac{\pi-4}{\pi-2}\\
        & \leq - \dfrac{ (\Phi_0^\kappa)''W_0}{ (\Phi_0^\kappa)'}  \dfrac{\diff \bar\rho_\infty}{\diff \kappa}(\kappa) + L^d\dfrac{ \Phi_0^\kappa}{2} \left(\bar \rho_\infty^\kappa - \dfrac{ \Phi_0^\kappa}{L^d}\right) \frac{\pi-4}{\pi-2},
    \end{align*}
    as, from the expression of $\tfrac{\diff \bar\rho_\infty}{\diff \kappa}$ in Lemma \ref{lm:der_bar_rho}, $\tfrac{\diff \bar\rho_\infty}{\diff \kappa}$ can range from $-\infty$ to $0$. By writing out the remaining $\tfrac{\diff \bar\rho_\infty}{\diff \kappa}$, gathering terms, and remembering that $\left(\bar \rho_\infty - \tfrac{ \Phi_0^\kappa}{L^d}\right)>0$ in the above inequality, we find that
    \begin{align*}
    \dfrac{ (\Phi_0^\kappa)''W_0}{ (\Phi_0^\kappa)'}  \dfrac{1}{2\kappa } \dfrac{(1 + L^d \Phi_0^\kappa \bar\rho_\infty^\kappa\kappa)}{1 -  (\Phi_0^\kappa)'W_0 \tfrac{g(\eta)}{L^d}} + L^d\dfrac{ \Phi_0^\kappa}{2} \frac{\pi-4}{\pi-2}  & \leq \dfrac{ (\Phi_0^\kappa)''W_0}{ (\Phi_0^\kappa)'}  \dfrac{1}{2\kappa } \dfrac{(1 + L^d \Phi_0^\kappa \bar\rho_\infty^\kappa\kappa)}{1 - \tfrac{ (\Phi_0^\kappa)'W_0}{L^d}(1-\tfrac{2}{\pi})} + L^d\dfrac{ \Phi_0^\kappa}{2} \frac{\pi-4}{\pi-2} \\
    & \leq \dfrac{ (\Phi_0^\kappa)''W_0}{ (\Phi_0^\kappa)'}  \dfrac{L^d}{2\kappa} \dfrac{\left(1 +  \Phi_0^\kappa \left( \Phi_0^\kappa \kappa + \sqrt{\tfrac{2}{\pi \kappa}}\right)\right)}{L^d -  (\Phi_0^\kappa)'W_0\left(1-\tfrac{2}{\pi}\right)} + L^d\dfrac{ \Phi_0^\kappa}{2} \frac{\pi-4}{\pi-2}\\ & < 0
    \end{align*}
will yield $A_4 < 0$. For the first inequality we again used the lower bound on $g(\eta)$ in Lemma \ref{lm:g_eta}, and for the second the relation \eqref{eqn:rho_0} together with a maximum of $\rho_\infty^\kappa(0)$. By inserting the lower bound of $\kappa$ from Lemma \ref{lm:kappa_small} and rearranging, we see that requiring
    \begin{align*}
         (\Phi_0^\kappa)'' > 2 \frac{4-\pi}{\pi-2} \frac{ \Phi_0^\kappa (\Phi_0^\kappa)'}{L^{2d}\pi B^2}\frac{L^d- (\Phi_0^\kappa)'W_0\left(1-\frac{2}{\pi} \right)}{1+ \Phi_0^\kappa\left( \Phi_0^\kappa + \frac{L^d B}{|W_0|}\right)} W_0,
    \end{align*}
    is sufficient for $ D^2_{\bar\rho \kappa} \mathcal H(0,\kappa) \left[ v\right] \neq 0$.
\end{remark}

\begin{theorem}[Characterisation of the branch]\label{thm:main_2}
    Let the assumptions of Theorem \ref{thm:main} hold. Then at the bifurcation point $(\bar\rho_\infty^{\kappa^*},\kappa^*)$, there exists a continuously differentiable curve of non-homogeneous solutions to $\bar{\mathcal{G}}(\bar \rho_\infty^\kappa, \kappa)=0$ ( $\bar{\mathcal{G}}$ is defined in \eqref{eq:onecompfunctional}),
    \begin{equation}\label{eq:nontrivial_rho}
        \{\ (\bar \rho_{\kappa(z)},\kappa(z)) \ | \ z\in(-\delta,\delta), \ (\bar \rho_{\kappa(0)},\kappa(0))=(\bar \rho_\infty^{\kappa^*},\kappa^*),\ \delta>0 \ \},
    \end{equation}
       such that, for all $x\in \mathbb T^d$,
    \begin{equation*}
        \bar \rho_{\kappa(z)}(x) = \bar \rho_\infty^{\kappa(z)} + z \omega_{k^*}(x) + o(z),
    \end{equation*}
    where $o(z)\in \mathrm{span}(\omega_{k^*})^\perp$. In a neighbourhood $U\times (\kappa^*-\varepsilon,\kappa^*+\varepsilon)$ of $(\bar\rho_\infty^{\kappa^*},\kappa^*)$ in $L^2_S(\mathbb T^d)\times \R_+^*$, all the solutions are either of the form $(\bar \rho_\infty^\kappa,\kappa)$ or on the non-homogeneous curve \eqref{eq:nontrivial_rho}. Moreover, the function $z\mapsto \kappa(z)$ satisfies $\kappa'(0)=0$, and
    \begin{equation*}
        \kappa''(0) = -\dfrac13 \dfrac{\mathcal K_1   }{ \mathcal K_2},
    \end{equation*}
    with
    \begin{align*}
	  \mathcal K_1 = \Bigg( - \dfrac{(\Phi^\kappa_0)'''}{(\Phi^\kappa_0)'} 
	   &+ 2(\Phi^\kappa_0)'' L^d 
	   \mathcal K_3 
	   \rho_\infty^\kappa(0) 
	   	  + (\Phi^\kappa_0)' L^d\Bigg[
	   	  -\left(  \dfrac{2L^{\frac d2}\Theta(k^*)}{\tilde W(k^*)} - 2\dfrac{ (\Phi_0^\kappa)'}{L^d} + \bar\rho_\infty^\kappa(\Phi^\kappa_0)'' L^{\frac d2} \dfrac{\tilde W(k^*)}{\Theta(k^*)} \right)\kappa\\
	   	  &+ \frac{\rho_\infty^\kappa(0)}{\kappa}\left( \dfrac{(\Phi^\kappa_0)''}{ (\Phi_0^\kappa)'} 
	   - \Phi'_0 L^d  \mathcal K_3 \rho_\infty^\kappa(0) \right) \Bigg]\Bigg) L^d\dfrac{\tilde W(k^*)^2}{\Theta(k^*)^2} \norme{\omega_{k^*}^2}^2_2,
	\end{align*}
    with
    \begin{align*}
        \rho_\infty^\kappa(0) = \left(\bar \rho_\infty^\kappa - \dfrac{ \Phi_0^\kappa}{L^d} \right) \kappa > 0, \qquad \mathcal{K}_3 = 1-L^d\bar\rho_\infty^\kappa\kappa (\Phi_0^\kappa)' L^{\frac d2} \dfrac{\tilde W(k^*)}{\Theta(k^*)} < 0,
    \end{align*}
    and $\mathcal K_2 = \pscal{D^2_{\bar\rho \kappa} \mathcal H(0,\kappa) \left[ \omega_{k^*}\right]}{\omega_{k^*}}<0$, where $ D^2_{\bar\rho \kappa} \mathcal H(0,\kappa) \left[ \omega_{k^*} \right]$ is defined in equation \eqref{eqn:no_more_idea_for_labels_3}.
\end{theorem}

\begin{proof}
 Let us now characterise the branch of the bifurcation by computing $\kappa'(0)$ and $\kappa''(0)$. We drop the star in $k^*$ for the sake of readability and denote by $\omega_k$ the element spanning the kernel of $D_{\bar\rho} \bar {\mathcal G} (\bar\rho_\infty^\kappa,\kappa)$.
 Recall that
	\[ W\ast \omega_k =  \dfrac{1}{ (\Phi_0^\kappa)' \left (\dfrac{1}{L^d} - L^d\bar \rho_\infty^\kappa\left(\bar \rho_\infty^\kappa - \dfrac{ \Phi_0^\kappa}{L^d}\right)\kappa \right)}\omega_k =: C_1 \omega_k,   \]
	and that, as obtained in the proof of Theorem \ref{thm:main},
	\begin{equation*}
	   \omega_k - \dfrac{ (\Phi_0^\kappa)'}{L^d}\, W\ast \omega_k = - \dfrac{L^d\bar \rho_\infty\left(\bar \rho_\infty^\kappa - \dfrac{ \Phi_0^\kappa}{L^d}\right)\kappa}{\dfrac{1}{L^d} - L^d\bar \rho_\infty^\kappa\left(\bar \rho_\infty^\kappa - \dfrac{ \Phi_0^\kappa}{L^d}\right)\kappa}\,\omega_k =: - C_2\, \omega_k,
	\end{equation*}
	with $C_2$ a positive constant.
 According to Lemma \ref{lm:third} and denoting again
	\[A_3 = 1 - \dfrac{L^d(\bar \rho_\infty^\kappa)^2\kappa}{\dfrac{1}{L^d} - L^d\bar \rho_\infty^\kappa\left(\bar \rho_\infty - \dfrac{ \Phi_0^\kappa}{L^d}\right)\kappa}<0,\]
	we obtain
 \begin{equation*}
	   D^2_{\bar\rho \bar\rho} \bar{\mathcal G}(\bar \rho_\infty^\kappa,\kappa) [\omega_k, \omega_k] =  \left[ -\dfrac{(\Phi^\kappa_0)''}{ (\Phi_0^\kappa)'}  
	   + \Phi'_0 L^d  A_3\left(\bar \rho_\infty^\kappa - \dfrac{ \Phi_0^\kappa}{L^d} \right)\kappa \right]C_1 \omega_k^2
	\end{equation*}
	We have \[\pscal{\omega_k^2}{\omega_k} = \int_{\mathbb T^d}\omega_k^3(x)\diff x = 0.\]
	Therefore, $\pscal{D^2_{\bar\rho \bar\rho} \bar{\mathcal G}(\bar \rho_\infty,\kappa) [\omega_k, \omega_k]}{\omega_k}=0$ and
	\begin{equation*}
	    \kappa'(0) = -\dfrac12\dfrac{\pscal{D^2_{\bar\rho \bar\rho} \mathcal H(0,\kappa) [\omega_k, \omega_k]}{\omega_k}}{\pscal{D^2_{\bar\rho\kappa} \mathcal H(0,\kappa)[\omega_k] }{\omega_k}} = -\dfrac12\dfrac{\pscal{D^2_{\bar\rho \bar\rho} \bar{\mathcal G}(\bar \rho_\infty^\kappa,\kappa) [\omega_k, \omega_k]}{\omega_k}}{\pscal{D^2_{\bar\rho\kappa} \mathcal H(0,\kappa)[\omega_k] }{\omega_k}} = 0
	\end{equation*}
	This result was expected as it was unlikely that the bifurcation would be transcritical.
	
Now, let us compute $\kappa''(0)$. We are going to use the formula
    \begin{equation*}
        \kappa''(0) = -\dfrac13 \dfrac{\pscal{D^3_{\bar\rho\bar\rho\bar\rho} \mathcal H(0,\kappa) [\omega_k,\omega_k,\omega_k]}{\omega_k} }{\pscal{D^2_{\rho\kappa}  \mathcal H(0,\kappa)[\omega_k]}{\omega_k}},
    \end{equation*}
    from Theorem \ref{thm:characterisation}. To compute the third order derivative term, we start with the general expression in Lemma \ref{lm:third} and we use the facts that $D^3_{\bar\rho\bar\rho\bar\rho} \mathcal H(0,\kappa) [\omega_k,\omega_k,\omega_k]=D^3_{\bar\rho\bar\rho\bar\rho} \bar {\mathcal G} (\bar\rho_\infty^\kappa,\kappa) [\omega_k,\omega_k,\omega_k]$ and that $D_{\bar\rho} \bar {\mathcal G} (\bar\rho_\infty^\kappa,\kappa)[\omega_k] = 0$:
\begin{align*}
	  D^3_{\bar\rho \bar\rho  \bar\rho} \bar{\mathcal G}(\bar \rho_\infty^\kappa,\kappa) [\omega_k, \omega_k, \omega_k] = & \, - (\Phi^\kappa_0)'''  \left (\dfrac 1 {L^d} - L^d\bar \rho_\infty^\kappa \left(\bar \rho_\infty^\kappa - \dfrac{ \Phi_0^\kappa}{L^d}  \right)\kappa\right)  ( W \ast \omega_k )^3 \\
	  & + 2(\Phi^\kappa_0)'' L^d \left(  
	   \omega_k  \left(\bar \rho_\infty^\kappa - \dfrac{ \Phi_0^\kappa}{L^d} \right) + \bar \rho_\infty^\kappa\left(\omega_k - \dfrac{ (\Phi_0^\kappa)'}{L^d}\, W\ast \omega_k \right)
	   \right)\kappa (W \ast \omega_k)^2   \\
	   	&  + (\Phi^\kappa_0)' L^d \Bigg[ 
	   	  - D^2_{\bar\rho\bar \rho} \bar{\mathcal G}(\bar \rho_\infty^\kappa,\kappa)[\omega_k,\omega_k]\left(\bar \rho_\infty^\kappa - \dfrac{ \Phi_0^\kappa}{L^d} \right) \\
	   	 & \quad +
	   2\omega_k \left(\omega_k - \dfrac{ (\Phi_0^\kappa)'}{L^d}\, W\ast \omega_k \right)
	   	     - \bar \rho_\infty^\kappa \dfrac{ (\Phi_0^\kappa)''}{L^d}\, W\ast \omega_k \, W\ast \omega_k\Bigg]
	   \kappa\, W \ast \omega_k .
	\end{align*}
	Hence,	we obtain
	\begin{align*}
	  & \pscal{D^3_{\bar\rho \bar\rho  \bar\rho} \bar{\mathcal G}(\bar \rho_\infty^\kappa,\kappa) [\omega_k, \omega_k, \omega_k]}{\omega_k} \\
	    & \qquad\qquad\quad = \Bigg( - \dfrac{(\Phi^\kappa_0)'''}{(\Phi^\kappa_0)'}  {C_1}^2 
	   + 2(\Phi^\kappa_0)'' L^d 
	   A_3 
	   {C_1}^2 \left(\bar \rho_\infty^\kappa - \dfrac{ \Phi_0^\kappa}{L^d} \right) \kappa 
	   	\\  
	   	  & \qquad\qquad\qquad + (\Phi^\kappa_0)' L^d\bigg[
	   	  -\left(  2 C_2 + \bar\rho_\infty^\kappa(\Phi^\kappa_0)'' {C_1}^2 \right)\kappa
	   	  \\
	   	  & \qquad\qquad\qquad + \left(\bar \rho_\infty^\kappa - \dfrac{ \Phi_0^\kappa}{L^d} \right)\left( \dfrac{(\Phi^\kappa_0)''}{ (\Phi_0^\kappa)'} 
	   - \Phi'_0 L^d  A_3\left(\bar \rho_\infty^\kappa - \dfrac{ \Phi_0^\kappa}{L^d} \right)\kappa \right) \bigg]C_1\Bigg) \pscal{{\omega_k}^3}{\omega_k}.
	\end{align*}
	Now, notice that
    \[ C_1 = L^{\frac d2} \dfrac{\tilde W(k)}{\Theta(k)}, \qquad A_3 = 1-L^d\bar\rho_\infty^\kappa\kappa (\Phi_0^\kappa)' C_1 \qquad\mathrm{and}\qquad C_2 = 1 - \dfrac{ (\Phi_0^\kappa)'}{L^d}C_1 = 1-\dfrac{ (\Phi_0^\kappa)'\tilde W(k)}{L^{\frac d2}\Theta(k)}.  \]
    Thus,
    \begin{align*}
	  & \pscal{D^3_{\bar\rho \bar\rho  \bar\rho} \bar{\mathcal G}(\bar \rho_\infty^\kappa,\kappa) [\omega_k, \omega_k, \omega_k]}{\omega_k} \\
	  & \qquad = \Bigg( - \dfrac{(\Phi^\kappa_0)'''}{(\Phi^\kappa_0)'} 
	   + 2(\Phi^\kappa_0)'' L^d 
	   \mathcal K_3
	    \left(\bar \rho_\infty^\kappa - \dfrac{ \Phi_0^\kappa}{L^d} \right) \kappa 
	   	\\  
	   	  & \qquad\quad + (\Phi^\kappa_0)' L^d\Bigg[
	   	  -\left(  \dfrac{2L^{\frac d2}\Theta(k)}{\tilde W(k)} - 2\dfrac{ (\Phi_0^\kappa)'}{L^d} + \bar\rho_\infty^\kappa(\Phi^\kappa_0)'' L^{\frac d2} \dfrac{\tilde W(k)}{\Theta(k)} \right)\kappa\\
	   	 & \qquad\quad  + \left(\bar \rho_\infty^\kappa - \dfrac{ \Phi_0^\kappa}{L^d} \right)\left( \dfrac{(\Phi^\kappa_0)''}{ (\Phi_0^\kappa)'} 
	   - (\Phi^\kappa_0)' L^d  \mathcal K_3\left(\bar \rho_\infty^\kappa - \dfrac{ \Phi_0^\kappa}{L^d} \right)\kappa \right) \Bigg]\Bigg) L^d\dfrac{\tilde W(k)^2}{\Theta(k)^2} \norme{\omega_k^2}^2_2,
	\end{align*}
 where 
 \begin{align*}
\mathcal K_3 = 1-L^d\bar\rho_\infty^\kappa\kappa (\Phi_0^\kappa)' L^{\frac d2} \dfrac{\tilde W(k)}{\Theta(k)}.
 \end{align*}
\end{proof}

Finding the sign of $\kappa''(0)$ (which determines whether the bifurcation is subcritical or supercritical) is hard as the formula involves a several terms, some of them being of unpredictable sign and scale, like $ (\Phi_0^\kappa)'''$. However, for certain choices of $\Phi$, this convoluted formula can be simplified, like we show in the following remark.

\begin{remark}
    In the case of the ReLU function $\Phi(x)=(x)^+$, the expression for $\mathcal{K}_1$ simplifies and we can obtain the asymptotic sign of $\kappa''(0)$ for bifurcations happening close to $\kappa=+\infty$. More precisely, using the notation for constants in the previous proof, we have
	\[\kappa''(0) = -\dfrac{L^dC_1\kappa}3 \dfrac{
	   	  -  2 C_2 
	   	  + 
	    L^d  |A_3|\left(\bar \rho_\infty^\kappa - \dfrac{ \Phi_0^\kappa}{L^d} \right)^2  }{\pscal{D^2_{\bar\rho\kappa} \mathcal H(0,\kappa)[\omega_k] }{\omega_k}}\]
    and we know that $C_2>0$ and $\pscal{D^2_{\bar\rho\kappa} \mathcal H(0,\kappa)[\omega_k] }{\omega_k} < 0$ according to the proof of Theorem \ref{thm:main}. Hence, the sign of $ \kappa''(0)$ depends on the sign of
    \[ L^d  |A_3|\left(\bar \rho_\infty^\kappa - \dfrac{ \Phi_0^\kappa}{L^d} \right)^2 - 2 C_2 = \left(\bar \rho_\infty^\kappa - \dfrac{ \Phi_0^\kappa}{L^d} \right)L^d \left[ |A_3|\left(\bar \rho_\infty^\kappa - \dfrac{ \Phi_0^\kappa}{L^d} \right)-\bar\rho_\infty^\kappa\kappa C_1\right], \]
    with $C_1>0$, and so it depends upon the sign of
    \[ |A_3|\left(\bar \rho_\infty^\kappa - \dfrac{ \Phi_0^\kappa}{L^d} \right)-\bar\rho_\infty^\kappa\kappa C_1.  \]
    The positive term tends to 0 exponentially fast (Lemma \ref{lm:equiv}) when $\kappa$ tends to $+\infty$ and the negative term tends to $-\infty$ linearly (because $C_1$ and $\bar\rho_\infty$ tend to a finite positive limit by Lemmata \ref{lm:inf_bound_rho} and \ref{lm:equiv}). Therefore, past some threshold value  for $\kappa$, all the bifurcations are subcritical. It is then likely that past this value, the bifurcation branches are nonlinearly unstable. This could explain why they were not witnessed in the numerical study of the bifurcation branches in \cite{CHS}, and strengthens the hypothesis of the existence of a hysteresis phenomena as suggested by the authors. 
\end{remark}

\subsection{Higher dimensional kernels and equivariant bifurcations}\label{sec:high_dim}

In order to satisfy the hypotheses of Theorem \ref{thm:mainmain} for some value $\kappa$, we need to ensure not only that for some $k^*$ the equality \eqref{eqn:W_tilde_equal} holds, but also that this $k^*$ is unique, that is to say: for all $k\in\N^d\setminus\{k^*\}$,
\[   \frac{\tilde W(k^*)}{\Theta(k^*)}\neq  \frac{\tilde W(k)}{\Theta(k)}.  \]
This is because the hypotheses of the Crandall--Rabinowitz theorem (Theorem \ref{thm:CR}) require that the kernel of the Fredholm operator $D_{\bar\rho} \mathcal H (0,\kappa) = D_{\bar\rho} \bar{\mathcal G} (\bar\rho_\infty,\kappa)$ is one-dimensional.

However, when the kernel of the operator $D_{\bar\rho} \bar {\mathcal G}(\bar \rho_\infty,\kappa)$ is of dimension higher than one, a bifurcation can still arise. In the previous results, the choice of the space $L_S^2(\mathbb T^d)$ of coordinate-wise even functions was already playing two roles: to quotient out the symmetry implied by $W\in L_S^2(\mathbb T^d)$, and to avoid the problem of translation invariance of the stationary states.

As noted in the introduction, a natural generalisation of this procedure is to use equivariant bifurcation theory \cite{CL-equivariantbif,Dionne,FSV2022} to take full advantage of all the symmetries of $W$. This can be applied to coordinate-wise even kernels, but it is particularly useful for radially symmetric ones. The approach is as follows. First, a closed subgroup of the Euclidean group $E_d$ acts from the left on a function space by
\[ \eta \cdot h(x) = h(\eta^{-1}(x)).  \]
Here, the Euclidean group $E_d$ refers to the set of all affine isometries in $\R^d$. We still denote by $\mathcal H$ the same functional as in the previous subsection, but defined on $L^2(\mathbb T^d)$, where $\mathbb T^d$ is the square torus with length $L$, or a restriction of it. Given a lattice $\mathcal L$ on $\mathbb T^d$, we denote $\Gamma$ the largest subgroup of $E_d$ that acts on $\mathcal L$-periodic functions such that $\mathcal H(\cdot,\kappa)$ is equivariant under the action of $\Gamma$, \textit{i.e.}
\[ \forall \eta \in\Gamma, \ \forall h\in L^2(\sfrac{\mathbb T^d}{\mathcal L}), \quad \mathcal H( \eta \cdot h,\kappa) = \eta \cdot \mathcal H(h,\kappa). \]
Now, define
\[  \mathcal V = \mathrm{Ker}\big( D_{\bar\rho} \mathcal H (0,\kappa) \big) \subset L^2(\mathbb T^d).\]
 For any isotropy subgroup $\Sigma$ for the action of $\Gamma$ on $\mathcal V$, denote
\[  
\mathrm{Fix}_{\mathcal V}( \Sigma ) = \{ h \in\mathcal V \mbox{ such that } \forall \eta\in\Sigma,\ \eta \cdot h = h \}. 
\]
Then, up to technical conditions, whenever $\dim\big(\mathrm{Fix}_{\mathcal V}( \Sigma )\big) = 1$, by the equivariant branching lemma \cite[Theorem 2.3.2]{CL-equivariantbif}, there is a unique branch of steady-state solutions of \eqref{eq:4PDE} whose averages $\bar\rho$ have the symmetry $\Sigma$. 

An important consideration is that the fundamental cell of the lattice $\mathcal L$ must satisfy two matching conditions: the periodicity condition associated with the flat torus of length $L$ on which \eqref{eqn:2} is posed and the symmetry of the convolution on the torus with the kernel $W$. The easiest way to ensure both is to choose a square lattice adapted to the length $L$ of $\mathbb T^d$.

\begin{theorem}\label{thm:ebl}
    Grant Assumption \ref{as:1}. Consider $\mathcal L$ the trivial $d-$dimensional cubic lattice of the flat torus generated by the vectors $(L,0,\dots,0), \dots, (0,\dots,0,L)$. Let $\Gamma$ be a subgroup of $E_d$. Assume $\mathcal H$ is $\Gamma-$equivariant on $L^{2}(\mathbb{T}^d)$ and that there exists a multi-index $k^*\in \mathbb N^d$ such that
    \begin{align*} 
\dfrac{\tilde W (k^*)}{\Theta(k^*)} =\dfrac{1}{L^{\frac d2} (\Phi_0^\kappa)' \left( \dfrac1{L^d} - L^d\bar \rho_\infty^\kappa  \left(\bar \rho_\infty^\kappa - \dfrac{ \Phi_0^\kappa}{L^d}\right)\kappa \right)},    
\end{align*}
where $\kappa\in\R_+^*$ and $ (\Phi_0^\kappa)''\geq 0.$ Denote $\mathcal V = \mathrm{Ker}\big( D_{\bar\rho} \mathcal H (0,\kappa) \big) \subset L^{2}(\mathbb{T}^d)$. Then, for all isotropy subgroups $\Sigma$ of $\Gamma$, such that $ \mathrm{dim}\big(\mathrm{Fix}_{\mathcal V}( \Sigma ) \big) = 1$, there exists a unique branch of stationary states of \eqref{eqn:2} bifurcating from the trivial curve and whose average $\bar\rho$ in $s$ has the symmetry of $\Sigma$.
\end{theorem}

\begin{proof}
In order to prove the result, we check the hypotheses of the equivariant branching lemma as stated in \cite[Theorem 2.3.2]{CL-equivariantbif}. The regularity required for the map $\mathcal H$ can be deduced from Lemma \ref{lm:der_bar_rho}. The fact that $0$ is an isolated eigenvalue with finite multiplicity stems from the same arguments as in the proof of Theorem \ref{thm:main}. Let us now consider the functional \eqref{eq:onecompfunctionalH} on the bigger space $L^2(\mathbb T^d)$. We prove that it is a Fredholm operator from $L^2(\mathbb T^d)$ to $L^2(\mathbb T^d)$ in the same way, and diagonalise $D_{\bar\rho} \mathcal H (0,\kappa)$ on the Hilbert basis of $L^2(\mathbb T^d)$ defined in Section \ref{sec:basis}; this time we use \eqref{trigo2} for the convolution. Similar computations show that the kernel $\mathcal V$ is generated by all Fourier modes $\omega_k(x)$ associated to a multi-index $k\in\Z^d$ such that
\[ 1 - \dfrac{L^{\frac d2}(\Phi_0^\kappa)'  \tilde W(|k|)}{\Theta(|k|)}  \left (\dfrac1{L^d} - L^d\bar \rho_\infty^\kappa  \left(\bar \rho_\infty^\kappa - \dfrac{ \Phi_0^\kappa}{L^d}\right)\kappa \right) = 0,
		 \]
   where $|k|=|(k_1,\dots,k_d)| = (|k_1|,\dots,|k_d|)$. The proof of the second order condition can be done in the same way as in the proof of Theorem \ref{thm:main}.
\end{proof}

Notice that, for a chosen $W$, the kernel $\mathcal V$ is now higher dimensional than when we were diagonalizing in $L^2_S(\mathbb T^d)$.
Note also that a particular case of Theorem \ref{thm:ebl} is Theorem \ref{thm:main} when the only symmetries are reflections on each coordinate. The dimension of the Fix criterion then comes down to the uniqueness of the index $k^*$. The general result of Theorem \ref{thm:ebl} does not indicate whether the bifurcation is saddle-node, transcritical or pitchfork, but for a specific choice of symmetries of $W$ and of the isotropy group $\Sigma$, it is possible to check the needed hypotheses of Theorem 2.3.2 in \cite{CL-equivariantbif} to specify the type of bifurcation. In all the cases we have checked at hand, the bifurcations are of pitchfork type like in Theorem \ref{thm:main_2}.

\begin{remark}
In Theorem \ref{thm:ebl}, we consider the lattice generated by the vectors $(L,0,\dots,0)$, $\dots$, $(0,\dots,0,L)$ for the sake of clarity. Indeed, it allows us to obtain the same bifurcation condition \eqref{eqn:W_tilde_equal_0} as in Theorem \ref{thm:main}. 

The fact that additional bifurcation points are not appearing when we consider a finer lattice generated by the vectors $(L/n,0,\dots,0), \dots, (0,\dots,0,L/n)$, $n\in\N^*$ is not obvious due to the nonlinear aspects of the condition. However, thanks to Lemma \ref{lm:g_eta}, we can rewrite the $L$-dependency in a clearer form:
\[  L^{\frac d2} (\Phi_0^\kappa)' \left( \dfrac1{L^d} - L^d\bar \rho_\infty^\kappa  \left(\bar \rho_\infty^\kappa - \dfrac{ \Phi_0^\kappa}{L^d}\right)\kappa \right) = \dfrac1{L^\frac d2} (\Phi_0^\kappa)' g\left(  \sqrt{\frac \kappa 2  } \Phi_0^\kappa\right). \]
Hence, taking a smaller square lattice of length $L/n$ would introduce a factor $n^\frac d2$ in the right-hand side of \eqref{eqn:W_tilde_equal_0}. The corresponding rescaling of the Fourier modes (note that since the integration area is smaller in the computation of $\Phi^\kappa_0$, $W$ has to be rescaled), introduces a factor $n^\frac d2$ in the left-hand side too. Hence, the bifurcation condition remains unchanged.
\end{remark}

\begin{remark}\label{rem:exchange}
In the spirit of \cite[Remark 4.6]{CGPS20}, we can do another interesting particular case at hand. Let $\mathfrak{S}_d$ be the group of permutations of $\{1,\dots,d\}$. Introduce the equivalence relation in $\N^d$, $k \sim l$ if and only if there exists $\zeta\in \mathfrak{S}_d$ such that $k=\zeta(l)$, and denote $[k]$ the equivalence class of any $k\in \N^d$. Define the space of exchangeable coordinate functions
\[  L^2_{ex}(\mathbb T^d) = \{ u\in L_S^2(\mathbb T^d) \ | \ \forall \zeta\in\mathfrak{S}_d, \forall x\in \mathbb T^d,\ u(x) = u(\zeta(x)) \},  
\]
endowed with the Hilbert basis $(\omega_{[k]})_{[k]\in \sfrac{\N^d}{\sim}}$ defined by $ \omega_{[k]}(x) = \frac{1}{\sqrt{\mathrm{card([k])}}} \sum_{l\in [k]} \omega_l(x)$, where $\sfrac{\mathbb N ^d}{\sim}$ denotes the associated quotient set containing all the equivalence classes.  

Then, there is a bifurcation point whenever condition \eqref{eqn:W_tilde_equal_0} in Theorem \ref{thm:main} is satisfied on a unique equivalence class $[k^*]$ and $(\Phi_0^\kappa)''\geq 0$ holds. In this case, the shape of the branch is given by
\begin{equation*}
        \bar \rho_{\kappa(z)}(x) = \bar \rho_\infty^{\kappa(z)} + z \dfrac{1}{\sqrt{\mathrm{card([k])}}} \sum_{l\in [k]} \omega_l(x) + o(z), \qquad z\in(-\delta,\delta),
\end{equation*}
with the same values $\kappa'(0)=0$ and $\kappa''(0)$ as defined in Theorem \ref{thm:main_2}.
\end{remark}

The most interesting case in our application framework is dimension $d=2$. In this case, the Euclidean group associated with the trivial lattice on the torus is the compact semi-direct sum of the Dihedral group $D_4$ and the compact group of translations on $\mathbb T^2$, $D_4 \overset{.}{+} \mathbb T^2$, see for example \cite{Dionne}. If we assume that $W$ is radially symmetric, then $\mathcal H$ is $\Gamma-$equivariant with $\Gamma=D_4 \overset{.}{+} \mathbb T^2$. In order to classify the patterns we can obtain with Theorem \ref{thm:ebl}, the procedure developed in \cite{Dionne,Dionne2} consists of writing the kernel $\mathcal V \subset L^2(\mathbb T^d)$ as a direct sum of $\Gamma-$irreducible subspaces $\mathcal V = \mathcal V_1 \oplus \cdots \oplus \mathcal V_n$, which in turn implies, for any isotropy subgroup $\Sigma$ of $\Gamma$,
\[  
\mathrm{Fix}_{\mathcal V}( \Sigma ) = \mathrm{Fix}_{\mathcal V_1}( \Sigma )\oplus \cdots \oplus \mathrm{Fix}_{\mathcal V_n}( \Sigma ). 
\]
Hence, it comes down to classifying $\Gamma$-irreducible representations of $\mathcal V$ depending on its dimension. In dimension 2, this work has already been done, for example in a general context in \cite{Dionne,Dionne2} or in the context of deterministic neural field models in \cite{FSV2022,VCF2015}. Our spectral computations in the proofs of Theorem \ref{thm:main} and Theorem \ref{thm:ebl} indicate that if $W$ is radially symmetric and $\tilde W(k)/\Theta(k)$ are unique up to permutations of $k$, then $\mathrm{dim}(\mathcal V) = 4$ or $\mathrm{dim}(\mathcal V) = 8$. In dimension 8, the full classification of isotropy subgroups of $\Gamma$ acting on a $8-$dimensional representation can be found in \cite[Table 3]{Dionne2}, with dimensions of the $\mathrm{Fix}$ set for each case. In dimension 4, there is a unique translation free irreducible representation. The reader can also find in \cite[Table 3]{Dionne} a general table relating planforms (patterns arising through equivariant bifurcations) to the type of lattice, the dimension of $\mathcal V$ and the isotropy subgroups. As explained in \cite[Section 2.a]{Dionne}, it is possible to consider only translation free subgroups by lowering the dimension if necessary, which we will show below in an example for stripes.

Therefore, we can use our previous analysis of bifurcation points in terms of the Fourier modes of $W$, combined with \cite[Theorem 2.3]{Dionne}, in order to conclude that when the bifurcation conditions of Theorem \ref{thm:ebl} are satisfied, there are branches of stationary states from the spatially homogeneous one with the following symmetries: stripes, also called rolls \cite[Figure 4]{CHS}; simples squares \cite[Figure 1]{Dionne}; squares \cite[Figure 2]{Dionne}; and anti-squares \cite[Figure 3]{Dionne}. In Section \ref{sec:numerics} we comment on some of these patterns and numerically explore their presence and stability in the setting of \eqref{eqn:2}. 

In order to make the abstract procedure above more explicit, let us take $L=1$ and detail two examples for the setting in the left plot of Figure \ref{fig:radsym-vs-nonradsym} in Section \ref{sec:numerics}. Numerical results indicate that the first bifurcation points in this case are associated to the multi-indexes $k=(4,0)$, $k=(4,1)$ and $k=(3,3)$.

\begin{example}\label{ex:1}
Consider the bifurcation point $\kappa^*$ associated to the mode $k=(4,0)$. Then $\mathcal V$ is generated by the functions
\[ \omega_{(4,0)} = \sqrt 2 \cos(8 \pi x_1) , \quad \omega_{(-4,0)} = \sqrt 2 \sin(8 \pi x_1),\quad \omega_{(0,4)} = \sqrt 2 \cos(8 \pi x_2), \quad \omega_{(0,-4)} = \sqrt 2 \sin(8 \pi x_2) . \]
The subgroup $\Sigma$ generated by the rotation of angle $\pi$ and the circle $\mathbb S$ of torus translations on the $x_1$ axis (respectively the $x_2$ axis) fixes only the span of $\omega_{(0,4)}$ (resp. of $\omega_{(4,0)}$), so we can apply Theorem \ref{thm:ebl}. The branch having the symmetry of $\Sigma$, consists of horizontal (resp. vertical) stripes. Since this planform arises from a subgroup which is not translation free, it was expected that it would have a one dimensional shape. Another way to obtain this pattern is to apply Theorem \ref{thm:ebl} in $L^2_S(\mathbb T^2)$ instead of $L^2(\mathbb T^2)$; then, the kernel $\mathcal V$ is only two-dimensional, containing the cosines, and the subgroup generated by the reflection on $x_2$ (resp. on $x_1$) fixes only a one dimensional space and generates the horizontal (resp. vertical) stripes.

The translation free subgroup $\Sigma=D_4$ fixes $\omega_{(4,0)}+\omega_{(0,4)}$, and the branch with this symmetry consists of simple squares, see the left plot in Figure \ref{fig:modepatterns}. Note that this branch could also be found with the approach of exchangable coordinates outlined in Remark \ref{rem:exchange}.
\end{example}

\begin{example}\label{ex:2}
Consider the Fourier mode $k=(3,3)$. The kernel $\mathcal V$ is now generated by
\begin{align*}
    &\omega_{(3,3)} = 2 \cos(6 \pi x_1)\cos(6\pi x_2), \quad \omega_{(3,-3)} = 2 \cos(6 \pi x_1)\sin(6\pi x_2),\\
    &\omega_{(-3,3)} = 2 \sin(6 \pi x_1)\cos(6\pi x_2), \quad \omega_{(-3,-3)} = 2 \sin(6 \pi x_1)\sin(6\pi x_2).
\end{align*}  
The translation free subgroup $\Sigma=D_4$ fixes only $\omega_{(3,3)}$, and hence gives rise to a bifurcating branch with the symmetry of simple squares, see the right plot in Figure \ref{fig:modepatterns}. This branch matches the pitchfork bifurcation in $L^2_S(\mathbb T^2)$ which Theorem \ref{thm:main} yields.
\end{example}

\begin{remark}\label{rem:higherdim}
Finally, we highlight the following: \\ 
1. If $\mathrm{Fix}_{\mathcal V}( \Sigma )$ is of higher dimension than 1, there could still be a bifurcation point there, but it is not detectable with the equivariant branching lemma.\\
2. When $W$ is just component-wise even, we can still apply Theorem \ref{thm:ebl}, but $\Gamma$, and correspondingly the number of isotropy subgroups, will be smaller compared to the case where $W$ is radially symmetric.
\end{remark}

\subsection{Study of the bifurcations given a connectivity kernel}

Another important question is to determine, given a connectivity kernel $W$, how many continuous bifurcations will arise from the constant stationary state and for which values of the parameter. Under a convexity assumption on $\Phi$, we can provide some answers.

\begin{theorem}[Bifurcations given a connectivity kernel $W$]\label{thm:Given}
Grant $\Phi''\geqslant 0$ and Assumption \ref{as:1}.
For all $k\in\N^d\setminus\{0\}$ such that $\tilde W(k)/\Theta(k)$ is unique up to permutations of $k$, if
\begin{equation}\label{eqn:given_potential}
\dfrac{\tilde W(k)}{\Theta(k)}  > \dfrac{L^{\frac d 2}}{\Phi'(W_0 \rho^* + B)},
\end{equation}
where $\rho^*=\lim_{\kappa\to+\infty} \bar\rho_\infty$,
then there exists a unique $\kappa$ such that $(\bar\rho_\infty^\kappa,\kappa)$ is a bifurcation point.
\end{theorem}

\begin{proof}
According to Lemma \ref{lm:g_eta}, the bifurcation condition can be re-written in the form
\[ \dfrac{\tilde W (k)}{\Theta(k)} =  \dfrac{L^{\frac d2}}{ (\Phi_0^\kappa)' g\left(\sqrt{\frac{\kappa}{2}} \Phi_0^\kappa\right)}, \]
with the function $g$ being increasing and satisfying for all $\eta>0$,
\[ 1-\frac{2}{\pi} = g(0) < g(\eta) < \lim_{\eta\to+\infty} g(\eta) = 1. \]
By Lemma \ref{lm:der_bar_rho}, $\bar\rho_\infty$ is decreasing, and thus $ \Phi_0^\kappa$ is increasing. Since we assume $\Phi''\geqslant 0$, $\Phi'$ is increasing and thus $ (\Phi_0^\kappa)'$ is also increasing. We deduce from all this information that the function
\[ \Psi : \kappa\mapsto \dfrac{L^{\frac d2}}{ (\Phi_0^\kappa)' g\left(\sqrt{\frac{\kappa}{2}} \Phi_0^\kappa\right)}  \]
is decreasing. By continuity of $\Phi'$ and Lemma \ref{lm:kappa_small} and \ref{lm:inf_bound_rho}, we have
\[ \lim_{\kappa\to \kappa_c} \Psi(\kappa) = +\infty \qquad \mathrm{and} \qquad \lim_{\kappa\to +\infty} \Psi(\kappa) =  \dfrac{L^{\frac d 2}}{\Phi'(W_0 \rho^* + B)},  \]
where $\kappa_c = \tfrac{2|W_0|^2}{L^{2d}\pi B^2}$ and $\rho^* = \lim_{\kappa\to +\infty} \bar \rho_\infty$. Note that the function $\Psi$ is not defined on $(0,\kappa_c)$ since $ (\Phi_0^\kappa)'=0$.

Hence, for all $k\in\N^d$ satisfying the hypotheses, there is a unique intersection up to permutations of $k$, between the horizontal line $\tfrac{\tilde W (k)}{\Theta(k)}$ and the decreasing function $\Psi$ for some value $\kappa\in(\kappa_c,+\infty)$. We can then apply either Theorem \ref{thm:main} or the discussion in Subsection \ref{sec:high_dim} to prove that there is a bifurcation for this value of $\kappa$.
\end{proof}

\begin{remark} 
    If $\Phi$ is not a $C^1$ function at point 0, for example the ReLU function $\Phi=(x)^+$, then the function
    \[ \Psi : \kappa\mapsto \dfrac{L^{\frac d2}}{ (\Phi_0^\kappa)' g\left(\sqrt{\frac{\kappa}{2}} \Phi_0^\kappa\right)}  \]
    does not satisfy $\lim_{\kappa\to \kappa_c} \Psi(\kappa) = 0$. However, if we have a left limit $\Phi'_{+} = \lim_{\kappa\to0^+} \Phi'(\kappa)$, and if $\Phi''$ exists on $\R^*$, we can still state a similar result with
    \begin{equation*}   \dfrac{L^{\frac d 2}}{\Phi'(W_0 \rho^* + B)}  < \dfrac{\tilde W(k)}{\Theta(k)}  <  \dfrac{\pi L^{\frac d 2}}{(\pi-2)\Phi'_{+}}   \end{equation*}
    in lieu of \eqref{eqn:given_potential}.
\end{remark}

\begin{corollary}[Characterisation of the first bifurcation]\label{cor:first}
Grant $\Phi''\geqslant 0$ and Assumption \ref{as:1}.
Let $k^*\in\N$ be such that 
\[ \dfrac{\tilde W(k^*)}{\Theta(k^*)} = \max \left\{ \dfrac{\tilde W(k)}{\Theta(k)} \ | \ k\in\N. \right\},    \]
If 
\[\dfrac{\tilde W(k^*)}{\Theta(k^*)}  >  \dfrac{L^{\frac d 2}}{\Phi'(W_0 \rho^* + B)}\]
and if for all $k\in\N^d\setminus\{k^*\}$, either $ \tfrac{\tilde W(k^*)}{\Theta(k^*)}\neq  \tfrac{\tilde W(k)}{\Theta(k)}$ or $k$ is a permutation of $k^*$, then there exists $\kappa^*$ such that $(\bar\rho_\infty^{\kappa^*},\kappa^*)$ is the first bifurcation point and $\kappa^*$ is the unique positive number such that
\begin{equation}\label{eqn:W_tilde_equal2}
        \dfrac{\tilde W (k^*)}{\Theta(k^*)} =  \dfrac{1}{L^{\frac d2} (\Phi_0^{\kappa^*})' \left( \dfrac1{L^d} - L^d\bar \rho_\infty^{\kappa^*}  \left(\bar \rho_\infty^{\kappa^*} - \dfrac{ \Phi_0^{\kappa^*}}{L^d}\right)\kappa^* \right)}.
    \end{equation}
\end{corollary}

 Notice that the value $\kappa^*$ characterised by \eqref{eqn:W_tilde_equal2}, that yields the first bifurcation point in the above corollary, coincide with the very point where the PDE system looses its linear stability, as we discussed in Section \ref{sec:linearstability}.

\section{Bifurcations of the full four component model} \label{sec:4component}
 Utilizing the details derived during the bifurcation analysis of the one component model \eqref{eqn:2}, we now sketch the procedure for finding bifurcation points of the full four component model \eqref{eq:4PDE}.
 
 First, note that when $B$ is constant in \eqref{eq:4PDE}, the PDEs for the four different directions $\beta = 1,2,3,$ and $4$, have the same right hand side. With $W\in L^2_S(\mathbb T^d)$, one can check that this system has the same spatially homogeneous stationary states as \eqref{eqn:2}. However, our study of the functional for the mean for the one component model \eqref{eq:onecompfunctional} does not fully characterise the bifurcations for the four component model \eqref{eq:4PDE}. This is due to the fact that the argument of $\Phi $ now is different: $W \ast \bar{\rho} (x) $ is replaced by $m(x) := \tfrac{1}{4}\sum_\beta W^\beta \ast \bar \rho^\beta (x)$, see \eqref{eq:phi}. 
 
 To be more specific, each component $\rho^1, \rho^2, \rho^3$ and $\rho^4$, will satisfy the calculations in Section \ref{eq:stationarystateequ}, giving
  \begin{align}\label{eq:rhobetastat}
     \bar\rho^\beta  = \dfrac{1}{Z_\rho} \int_0^{+\infty} s \e^{ -\kappa \frac{ \left( s - \Phi(m + B) \right)^2}{2}}\diff s, \qquad Z_\rho = L^d\int_0^{+\infty} \e^{ -\kappa \frac{ \left( s - \Phi(m + B) \right)^2}{2}}\diff s,
 \end{align}
 for each $\beta$. As a consequence, the stationary states are equal, $\rho^1(x)=\rho^2(x)=\rho^3(x)=\rho^4(x)$, since the value of $\Phi(x) = \Phi(m(x)+B)$ is identical in the four different directions. 
 Assuming no spatial dependence of the stationary state and periodicity of $W$ then yields $\bar\rho^\beta = \bar\rho_\infty$, which is the unique spatially homogeneous zero of the functional $\bar{\mathcal{G}}(\bar \rho,\kappa)$ for the mean in the one component case \eqref{eq:onecompfunctional}. For spatially dependent stationary states, however, $\Phi(m(x)+B)$ in \eqref{eq:rhobetastat} and $\Phi(W \ast \bar\rho +B)$ when $\bar{\mathcal{G}}(\bar \rho,\kappa)=0$ in \eqref{eq:onecompfunctional} are not identical, which leads to minor modifications in the bifurcation conditions as we now show. 
 
We are only interested in bifurcations from the spatially homogeneous solutions, so it suffices to consider bifurcations of the spatially homogeneous stationary states of the sum $m$. The stationary states of $m$ are zeroes of the functional
 \begin{align*}
     \mathcal{Q}(m, \kappa) = m - \sum_\beta W^\beta \ast \left( \dfrac{1}{Z_\rho} \int_0^{+\infty} s \e^{ -\kappa \frac{ \left( s - \Phi(m + B) \right)^2}{2}}\diff s\right).
 \end{align*}
We define the functional $\mathcal{V}(m, \kappa) :=\mathcal{Q}(m + m_\infty^\kappa, \kappa) = \mathcal{Q}(m + W_0\bar\rho_\infty^\kappa, \kappa)$, where $m_\infty^\kappa= W_0\bar\rho_\infty^\kappa$ is the spatially homogeneous stationary state for a given $\kappa$. Like in the one component case, we wish to apply the Crandall--Rabinowitz theorem (Theorem \ref{thm:CR}) to $\mathcal{V}(m, \kappa)$. This can be done by following the proof of Theorem \ref{thm:main} step by step, and making the necessary adjustments. Note that keeping the assumptions on $W$, it makes sense to consider the same space and basis as in the proof of Theorem \ref{thm:main}. Let $h_1,h_2\in L^2_S(\mathbb T^d)$. By following the calculations leading up to, and in the proof of, Theorem \ref{thm:main}, one can check that

\begin{align}\label{eq:mderiv}
D_m \mathcal{V}(0, \kappa)[h_1] = D_m \mathcal{Q}(W_0\bar \rho_\infty^\kappa, \kappa)[h_1] = h_1 -\frac{ (\Phi_0^\kappa)'}{4} \left(\dfrac{1}{L^d} - L^d\bar \rho_\infty^\kappa\left(\bar \rho_\infty^\kappa - \frac{ \Phi_0^\kappa}{L^d}\right)\kappa\right) \sum_\beta W^\beta \ast h_1,
\end{align}
and

\begin{align} \label{eq:mkderiv}
 D_{m\kappa}^2 & \mathcal{V}(0, \kappa)[h_2] \nonumber \\
=  & \, D^2_{m\kappa} \mathcal Q(m_\infty^\kappa,\kappa)[h_2]+ D^2_{mm} \mathcal Q(m_\infty^\kappa,\kappa) \left[W_0\dfrac{\diff \bar\rho_\infty}{\diff \kappa}(\kappa), h_2\right] \nonumber \\ 
=  & \, - (\Phi_0^\kappa)'L^d \left(\bar \rho_\infty^\kappa - \frac{ \Phi_0^\kappa}{L^d}\right)\frac{ \Phi_0^\kappa}{2}\left[ \frac{1}{L^d}g\left(\sqrt{\tfrac{\kappa}{2}} \Phi_0^\kappa\right) -L^d\kappa (\bar \rho_\infty^\kappa)^2 \right] \frac{1}{4} \sum_\beta W^\beta \ast h_2 \\
& + W_0 \dfrac{\diff \bar\rho_\infty}{\diff \kappa}(\kappa) \bigg[-\frac{ (\Phi_0^\kappa)''}{L^d}g\left(\sqrt{\tfrac{\kappa}{2}} \Phi_0^\kappa\right) \nonumber \\
& \qquad \qquad \qquad \,\,\ + ( (\Phi_0^\kappa)')^2L^d\kappa\left(\bar \rho_\infty^\kappa -\frac{ \Phi_0^\kappa}{L^d}\right)\!\!\!\left(\frac{1}{L^d}g\left(\sqrt{\tfrac{\kappa}{2}} \Phi_0^\kappa\right)-L^d\kappa (\bar \rho_\infty^\kappa)^2\right)\!\bigg] \!\frac{1}{4} \!\sum_\beta W^\beta\!\ast h_2 \nonumber
\end{align}
with $g(\eta)$ as defined in Lemma \ref{lm:g_eta},
\[ \frac{1}{L^d} g\left(\sqrt{\frac{\kappa}{2}} \Phi_0^\kappa\right) = \dfrac1{L^d} - L^d\bar \rho_\infty^\kappa  \left(\bar \rho_\infty^\kappa - \dfrac{ \Phi_0^\kappa}{L^d}\right)\kappa. \]
Now, as in the proof of Theorem \ref{thm:main}, we let $h_1 = \omega_k$, where $\omega_k$ is in the orthonormal basis of $L^2_S(\mathbb T^d)$, in \eqref{eq:mderiv}. Using the fact that $W$ and the shifts $r^\beta$ are coordinate-wise even, we calculate 
\begin{align*}
\sum_\beta W^\beta \ast \omega_k = & \, \sum_\beta \int_{\mathbb T^d} W(x-y-r^\beta) \tfrac{\Theta(k)}{L^{\frac d2}}\prod_{i=1}^{d}\cos\left(\tfrac{2\pi k_i}{L} y_i\right) \diff y \\
= & \, \sum_\beta \int_{\mathbb T^d} W(x-y-r^\beta) \tfrac{\Theta(k)}{L^{\frac d2}}\prod_{i=1}^{d}\bigg[\cos\left(\tfrac{2\pi k_i}{L} (y_i+r_i^\beta-x_i\right)\cos\left(\tfrac{2\pi k_i}{L} (x_i-r_i^\beta)\right) \\
& - \sin\left(\tfrac{2\pi k_i}{L} (y_i+r_i^\beta-x_i)\right)\sin\left(\tfrac{2\pi k_i}{L} (x_i-r^\beta)\right) \bigg] \diff y \\
= &  \, \sum_\beta \int_{\mathbb T^d} W(x-y-r^\beta) \tfrac{\Theta(k)}{L^{\frac d2}}\prod_{i=1}^{d}\cos\left(\tfrac{2\pi k_i}{L} (y_i+r_i^\beta-x_i\right)\cos\left(\tfrac{2\pi k_i}{L} (x_i-r_i^\beta)\right) \diff y \\
= &  \, \sum_\beta \int_{\mathbb T^d} W(x-y-r^\beta) \prod_{i=1}^{d}\cos\left(\tfrac{2\pi k_i}{L} (y_i+r_i^\beta-x_i\right) \diff y
\tfrac{\Theta(k)}{L^{\frac d2}} \prod_{i=1}^{d}\cos\left(\tfrac{2\pi k_i}{L} (x_i-r_i^\beta)\right) \\
= &  \, \tfrac{L^{\frac d2}}{\Theta(k)}\Tilde{W}(k)
\sum_\beta \tfrac{\Theta(k)}{L^{\frac d2}} \prod_{i=1}^{d}\left[\cos\left(\tfrac{2\pi k_i}{L} x_i\right)\cos\left(\tfrac{2\pi k_i}{L} r_i^\beta\right)-\sin\left(\tfrac{2\pi k_i}{L} x_i\right)\sin\left(\tfrac{2\pi k_i}{L} r_i^\beta\right)\right] \\
= &  \, \tfrac{L^{\frac d2}}{\Theta(k)}\Tilde{W}(k) \omega_k
\sum_\beta \prod_{i=1}^{d}\cos\left(\tfrac{2\pi k_i}{L} r_i^\beta\right),
\end{align*}
where the last step follows due to the shifts being coordinate-wise even. 
Then, comparing the resulting expression with the bifurcation condition $\eqref{eqn:W_tilde_equal}$, we arrive at the slightly modified assumption 
\begin{align}\label{eq:4compbifurcation}
\exists ! \ k^* \in \mathbb N^d \quad \mathrm{s.t.} \quad 
        \dfrac{\tilde W (k^*)}{\Theta(k^*)} \frac{1}{4}\sum_\beta \prod_{i=1}^{d}\cos\left(\tfrac{2\pi k_i}{L} r_i^\beta\right) = \dfrac{1}{L^{\frac d2} (\Phi_0^\kappa)' \left( \dfrac1{L^d} - L^d\bar \rho_\infty^\kappa  \left(\bar \rho_\infty^\kappa - \dfrac{ \Phi_0^\kappa}{L^d}\right)\kappa \right)}
\end{align}
for $(m_\infty^\kappa,\kappa)=(W_0 \bar \rho_\infty^\kappa,\kappa)$ being a bifurcation point. Setting $h_2= v \in \mathrm{Ker}\big(D_{m} \mathcal V(0,\kappa)\big)$ in \eqref{eq:mkderiv}, we find that the assumptions needed for $D_{m\kappa}^2 \mathcal{V}(0, \kappa)\left[ v\right] \neq 0$ are exactly the same as for the one component model (cf. \eqref{eqn:no_more_idea_for_labels_3}). With $h_1=v$ in \eqref{eq:mderiv} one can easily check that $v \in \mathrm{Ker}\big(D_{m} \mathcal V(0,\kappa)\big)$ has to satisfy
\begin{align*}
\frac{1}{4}\sum_\beta W^\beta \ast v = \frac{1}{ (\Phi_0^\kappa)' \left(\dfrac{1}{L^d} - L^d\bar \rho_\infty^\kappa\left(\bar \rho_\infty^\kappa - \frac{ \Phi_0^\kappa}{L^d}\right)\kappa\right)}v,    
\end{align*}
such that
\begin{align*}
D_{m\kappa}^2 \mathcal{V}(0, \kappa)[v] = D_{\bar \rho \kappa}^2 \mathcal{H}(0, \kappa)[v],
\end{align*}
where $D_{\bar \rho \kappa}^2 \mathcal{H}(0, \kappa)[v]$ can be found in \eqref{eqn:no_more_idea_for_labels_3}.
Hence, Theorem \ref{thm:main} is valid, under the same assumptions on $\Phi$, for the four component model \eqref{eq:4PDE} when $\tilde W (k^*)$ in \eqref{eqn:W_tilde_equal} is exchanged with 
$$
\tilde W (k^*) \frac{1}{4}\sum_\beta\prod_{i=1}^{d}\cos\left(\tfrac{2\pi k_i}{L} r_i^\beta\right). 
$$
As highlighted in Subsection \ref{sec:linearstability} for \eqref{eqn:2}, this condition coincides with the linear stability threshold obtained in \cite[Theorem 3.6]{CHS} in the two dimensional case $d=2$ when the shifts are coordinate-wise even. 

Similarly, the higher order derivatives needed to characterise the shape of the branches in Theorem \ref{thm:main_2} are the same as in the one component case. Thus, Theorem \ref{thm:main_2} holds without modification. Finally, this can straightforwardly be extended to Theorem \ref{thm:ebl} with the slightly modified bifurcation condition.

\section{The bifurcation condition and patterns in two dimensions}\label{sec:numerics}
We now present supplementary insight about the location and shape of the bifurcations described theoretically in the previous sections in the case of two spatial dimensions ($d=2$). We also provide a brief comment on the stability of the spatially homogeneous state and on the stationary hexagonal pattern for \eqref{eqn:2} and \eqref{eq:4PDE} with radially symmetric kernels (see also \cite{CHS}).

\begin{figure}[ht]
    \centering
    \includegraphics[width=0.49\textwidth]{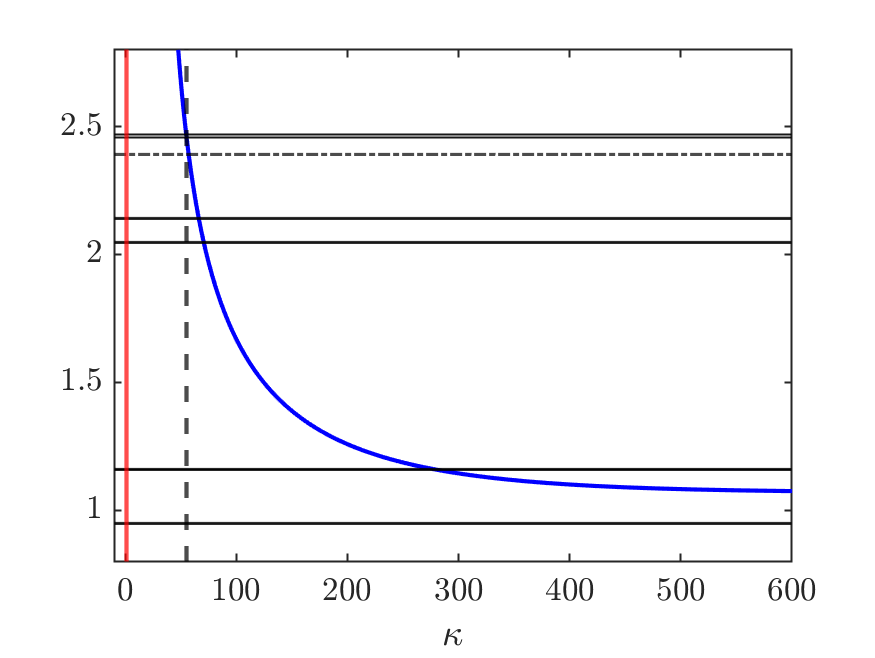}
     \includegraphics[width=0.49\textwidth]{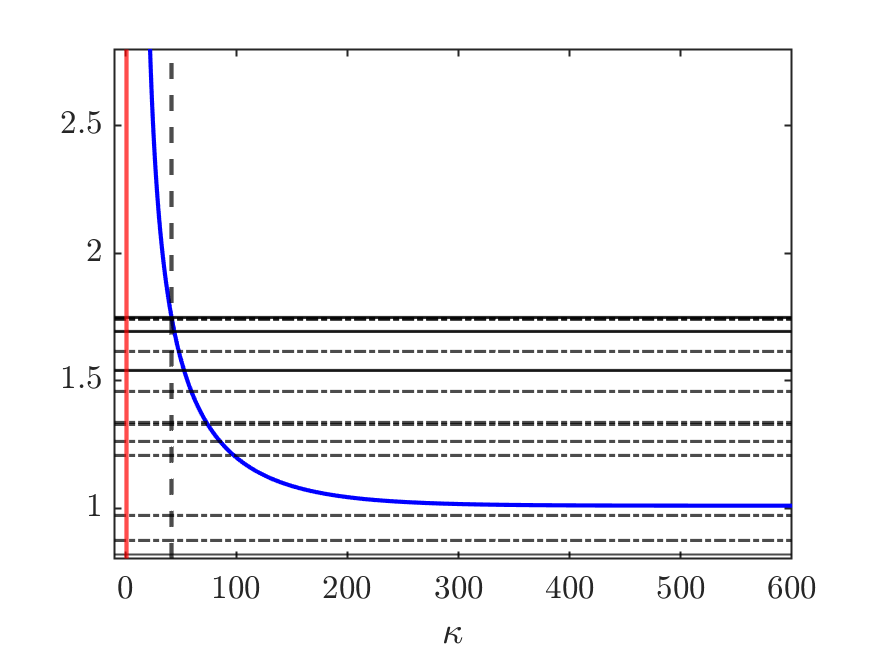}
    \caption{Bifurcation points represented as the crossing of Fourier modes (the black horizontal lines) and the right hand side of \eqref{eqn:W_tilde_equal_0} as a function of $\kappa$ (the blue line). The red line represents the critical $\kappa$-condition in Lemma \ref{condkappa}, and the vertical dashed line the condition in Remark \ref{rem:linear}. Here, $L=1$, $B=3$, $\Phi(x) = 0.5x\left(1+\tfrac{x}{\sqrt{x^2+0.1}}\right)^+$, and $W(x,y)=-0.005\cdot 2^{14}\left(1+\tanh\left(10-50\sqrt{ax^2+y^2}\right)\right)$. Left: $a=1$ (radially symmetric), right: $a=2$ (coordinate-wise even, but not radially symmetric).}
    \label{fig:radsym-vs-nonradsym}
\end{figure}

\subsection{The bifurcation condition}
Theorems \ref{thm:main}, \ref{thm:ebl}, and \ref{thm:Given} give us sufficient conditions for bifurcations, but since the right-hand side of the condition \eqref{eqn:W_tilde_equal_0} depends nonlinearly on $\kappa$, it is not easy to locate the bifurcation points theoretically. Thus, we plot in Figure \ref{fig:radsym-vs-nonradsym}, for a choice of $\Phi$ and two choices of $W$, the right-hand side of this condition in blue as a function of $\kappa$ and the left-hand side as horizontal lines representing Fourier modes. Each crossing between the decreasing blue curve and the horizontal lines yields a possible bifurcation point as described by Theorem \ref{thm:main} and Theorem \ref{thm:ebl}. A zoom of the left plot of Figure \ref{fig:radsym-vs-nonradsym} was presented in Figure \ref{fig:intro-illustration} in the introduction.

In all the plots, the vertical dashed line represents the linear stability condition from \cite{CHS}, recalled in Section \ref{sec:linearstability}, which coincides with the location of the first bifurcation (see Corollary \ref{cor:first}). The red line represents the critical value $\kappa_c = \tfrac{2|W_0|^2}{L^{2d}\pi B^2}$ (see Lemma \ref{lm:kappa_small}): no bifurcation can occur for $\kappa \leqslant \kappa_c$ for the right-hand side of the condition \eqref{eqn:W_tilde_equal_0} has the constant value $+\infty$ on $(0,\kappa_c]$.

The left plot in Figure \ref{fig:radsym-vs-nonradsym} shows the case of a radially symmetric connectivity kernel. Here, each crossing yields a bifurcation point. The crossings with the dashed black lines denote that there is a unique branch in $L^2_S(\mathbb T^d)$ at that point according to Theorem \ref{thm:main}. Note that this does not imply that this branch is the only branch in $L^2(\mathbb T^d)$ at that point or that there are no coordinate-wise even patterns at the other bifurcation points. As described in Section \ref{sec:high_dim}, each crossing in the left plot of Figure \ref{fig:radsym-vs-nonradsym} gives rise to multiple branches, each with their square symmetry, emanating from the bifurcation point. It can be seen in the right plot of Figure \ref{fig:radsym-vs-nonradsym} that a lot more bifurcation points in $L^2_S(\mathbb T^d)$ are detected with Theorem \ref{thm:main} when $W$ is in $L^2_S(\mathbb T^d)$, but is not radially symmetric, compared to the case of the radially symmetric connectivity to the left. Note that also in this case Theorem \ref{thm:ebl} yields additional bifurcation branches, see Remark \ref{rem:higherdim}.

Finally, we note that the only change of the plots in Figure \ref{fig:radsym-vs-nonradsym} in the case of the four component model \ref{eq:4PDE} is that the Fourier modes are re-scaled by a shift-dependent factor according to the bifurcation condition \eqref{eq:4compbifurcation} in Section \ref{sec:4component}.

\subsection{Patterns and numerical exploration of stability}
Here we plot examples of patterns along some of the bifurcation branches and numerically describe the loss of stability of the spatially homogeneous state of \eqref{eqn:2} through a simple bifurcation diagram in the case of the radially symmetric kernel of Figure \ref{fig:radsym-vs-nonradsym}. We do not provide colorbar scales for the plots in this section as we are only interested in the qualitative shape of the patterns, and not the quantitative aspects.

\begin{figure}[ht]
    \centering
    \includegraphics[width=0.32\textwidth]{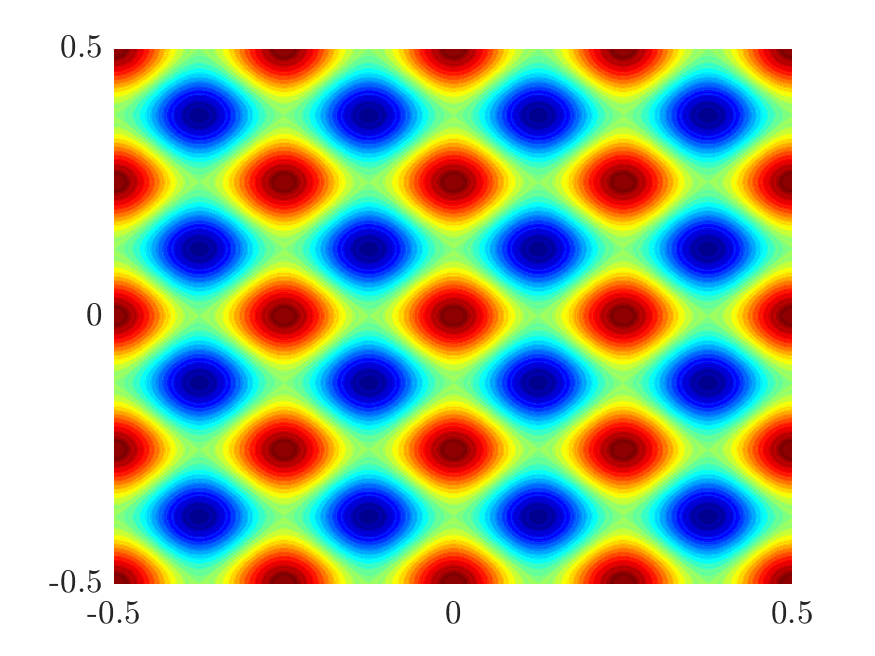}
     \includegraphics[width=0.32\textwidth]{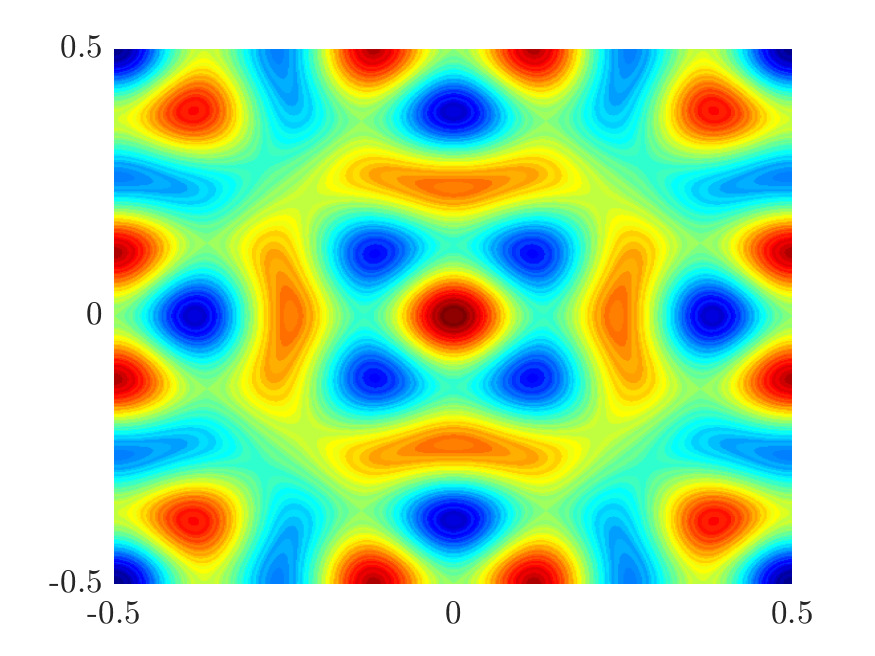}
        \includegraphics[width=0.32\textwidth]{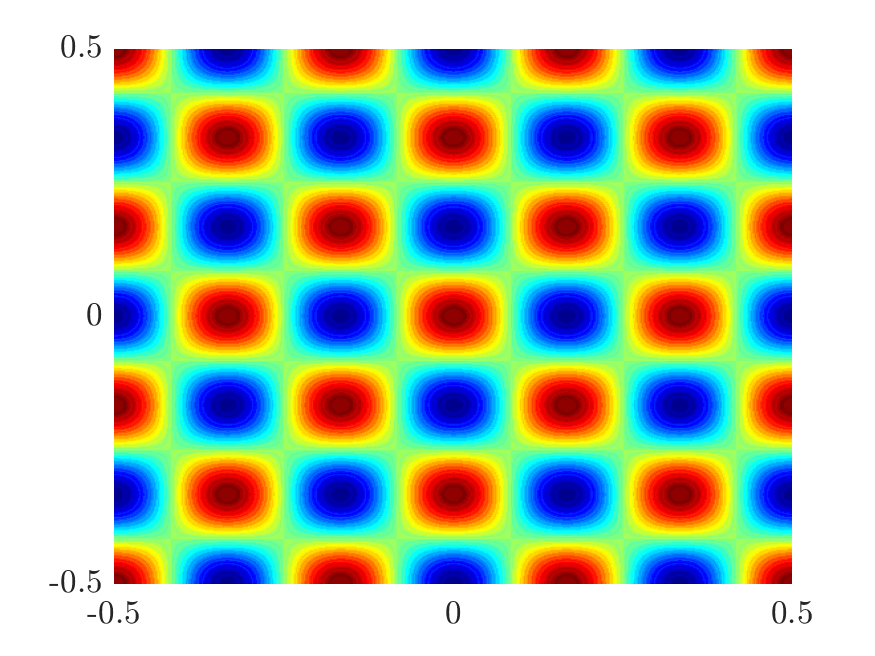}
    \caption{Patterns  in $L^2_S(\mathbb{T}^2)$, emerging along a bifurcation branch according to Theorem \ref{thm:main} and Remark \ref{rem:exchange} at each of the three first bifurcation points. The setting is as in the left plot of Figure \ref{fig:radsym-vs-nonradsym}. Left to right (first to third bifurcation point): $k^*=(0,4)$ and $(4,0)$, $k^*=(1,4)$ and $(4,1)$, and $k^*=(3,3).$}
    \label{fig:modepatterns}
\end{figure}

In Figure \ref{fig:modepatterns} the three functions $\omega(x)$ corresponding to the leftmost (first) crossings in the left plot of Figure \ref{fig:radsym-vs-nonradsym} according to Theorem \ref{thm:main} and Remark \ref{rem:exchange} are depicted. They correspond to the values or equivalence classes $[k^*] = ((4,0), (0,4))$, $[k^*]=((1,4), (4,1))$, and $k^*=(3,3)$ (from left to right in Figure \ref{fig:modepatterns}). Keeping $L=1$, the functions that we plot are
\begin{equation*}
    \omega_{[(4,0)]}(x) = \frac{1}{\sqrt 2}\big(\cos(8\pi x) + \cos(8\pi y)\big), \qquad \omega_{[(1,4)]}(x)= \cos(8\pi x)\cos(2\pi x) + \cos(2\pi x)\cos(8\pi x),
\end{equation*}
and
\begin{equation*}
     \omega_{(3,3)}(x) = 2\cos(6\pi x)\cos(6\pi y).
\end{equation*}
Note that the existence of these patterns could also be deduced from Theorem \ref{thm:ebl}, see Examples \ref{ex:1} and \ref{ex:2}.

As remarked in Section \ref{sec:high_dim}, the shape of the patterns along the different branches have been characterised in two dimensions \cite{Dionne}, and, as already noted, with a square periodic lattice, there will be branches with patterns consisting of rolls/stripes, simple squares, super squares, or anti-squares. Restricting the functional \eqref{eq:onecompfunctionalH} to the square lattice defined by the periodicity of the problem as we did in Section \ref{sec:high_dim}, the characteristic hexagonal pattern obtained through time evolution of the PDE system \eqref{eq:4PDE} in \cite{CHS} in the case of a radially symmetric connectivity kernel will thus not appear along any of the corresponding bifurcation branches. This can heuristically be explained by the fact that a hexagonal pattern cannot be fitted periodically on the square while at the same time preserve any square symmetries---the ratio between the inradius and circumradius of a hexagon is $\frac{\sqrt{3}}{2}$. 
A simple Fourier decomposition of a hexagonal pattern which periodically fits on the square $[-0.5,0.5]^2$ is
\begin{align}\label{eq:hex}
    \omega^{\mathrm{hex}}(x_1,x_2) = & \, \cos(6\pi x_1)\cos(6\pi x_2) + \cos(8\pi x_1)\cos(2\pi x_2) + \cos(2\pi x_1)\cos(8\pi x_2)\\ & 
    + \sin(6\pi x_1)\sin(6\pi x_2) - \sin(8\pi x_1)\sin(2\pi x_2) - \sin(2\pi x_1)\sin(8\pi x_2), \nonumber
\end{align}
but this is not symmetric with respect to the corresponding periodicity lattice. One can check that this pattern is coordinate-wise even with respect to axes rotated clockwise by $\frac{\pi}{12}$.

\begin{figure}[ht]
    \centering
    \includegraphics[width=0.32\textwidth]{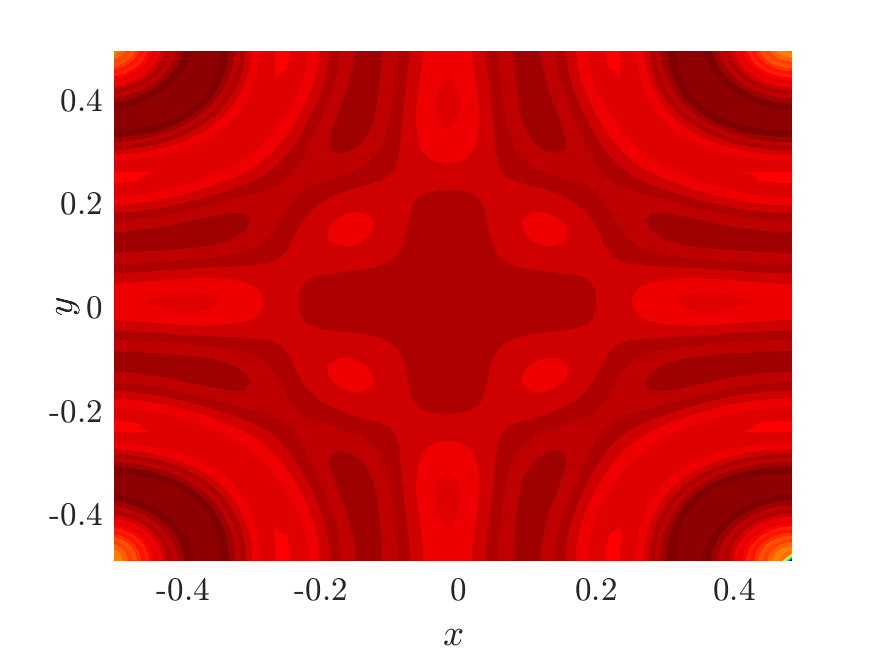}
     \includegraphics[width=0.32\textwidth]{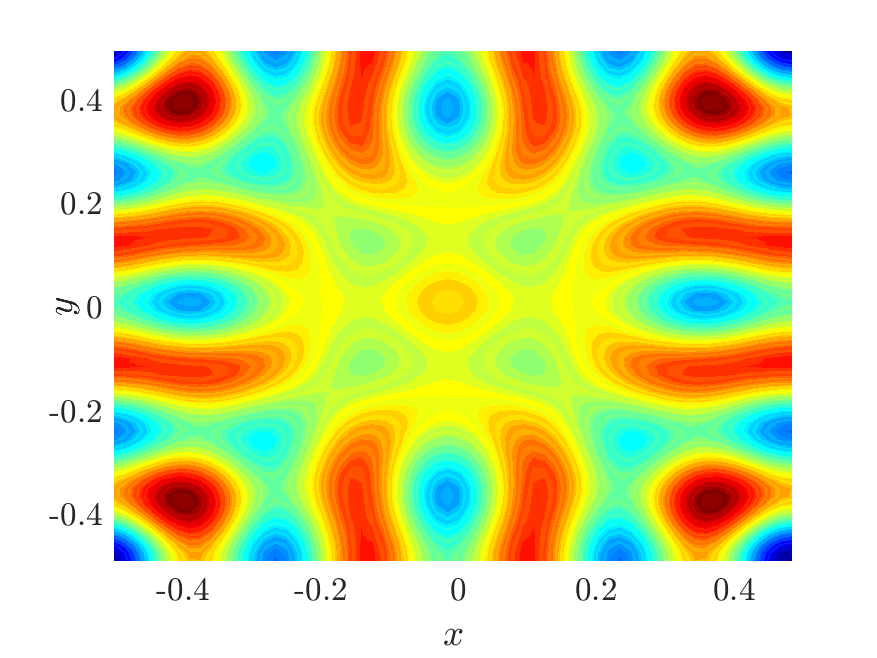}
          \includegraphics[width=0.32\textwidth]{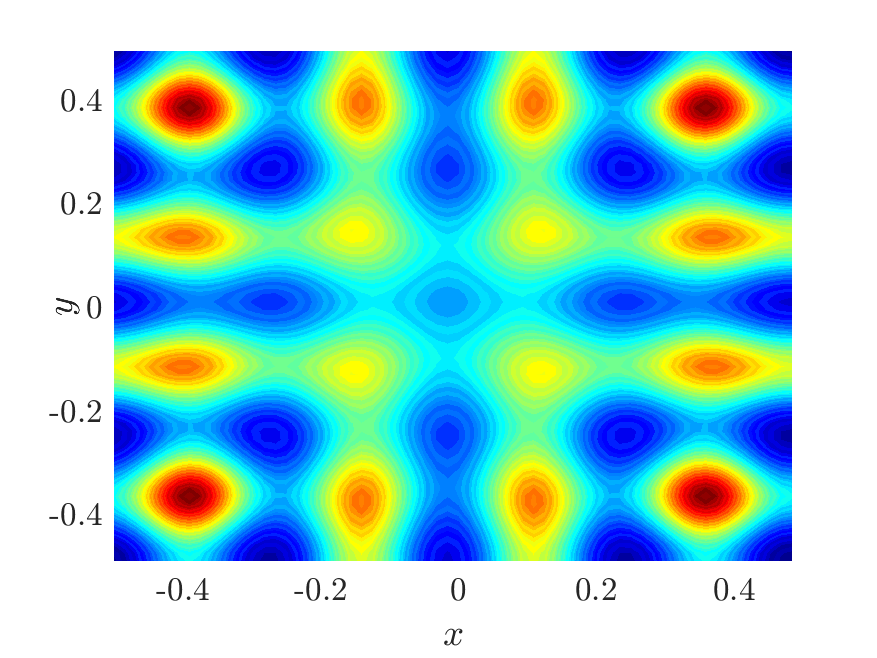}\\
    \includegraphics[width=0.32\textwidth]{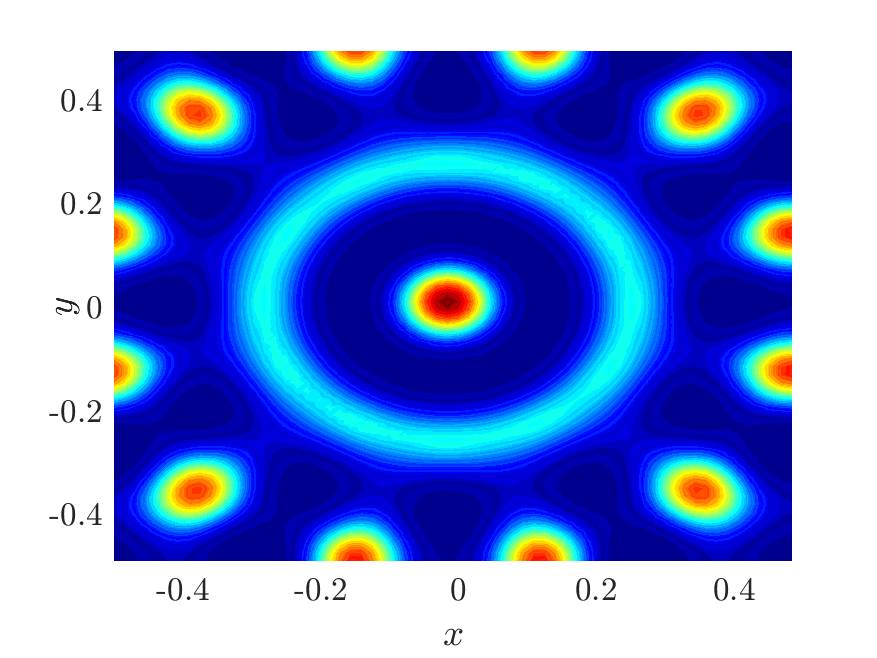}
     \includegraphics[width=0.32\textwidth]{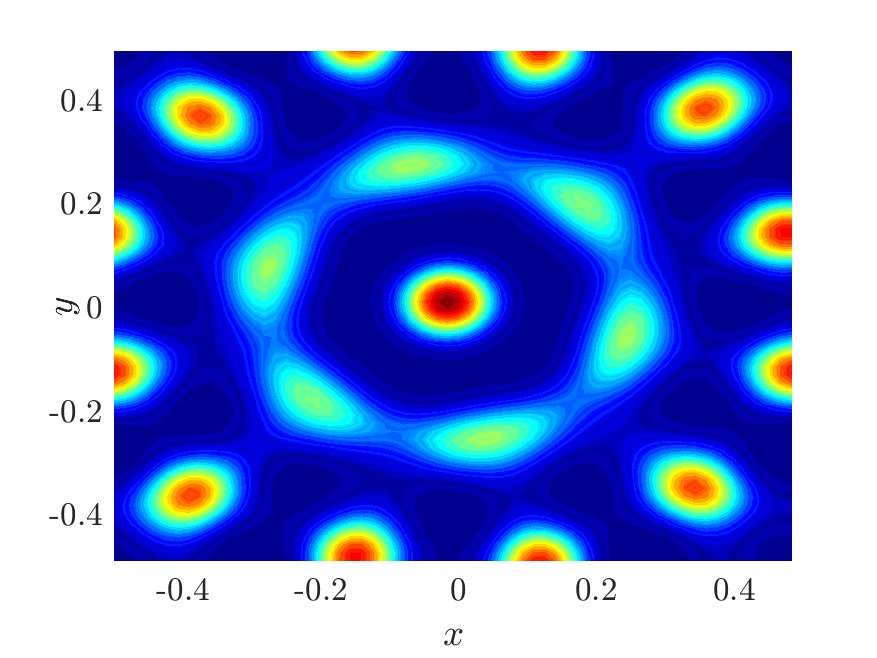}
          \includegraphics[width=0.32\textwidth]{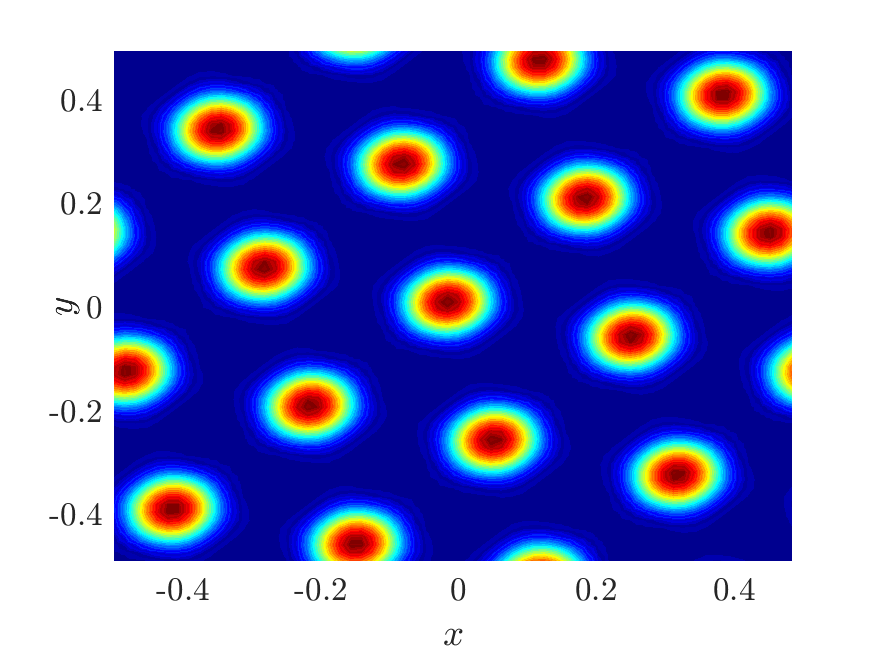}
    \caption{Time-transient patterns of the mean by solving the PDE \eqref{eqn:2} numerically with $\kappa = 55$ (close to the first three bifurcation points) and $\tau=10\,$m$s$ and otherwise keeping the settings as in the left plot of Figure \ref{fig:radsym-vs-nonradsym}. From top left to bottom right, moving horizontally, snapshots at $t=40, 220, 1500, 1810, 2190,$ and $2400\,$m$s$. }
    \label{fig:timetransient}
\end{figure}

To illustrate that the hexagonal pattern numerically appears as a stable stationary state of \eqref{eqn:2}, a typical time evolution for the PDE \eqref{eqn:2} can be found in Figure \ref{fig:timetransient}. Here $\bar\rho(x,t)$ is plotted at different times for a $\kappa$ close to the three leftmost bifurcation points of the left plot in Figure \ref{fig:radsym-vs-nonradsym}. The time series is obtained by introducing a machine precision perturbation of spatially homogeneous initial data at one position and then using the same numerical method as in \cite{CHS} on a $64\times64$-grid. We observe that at first the patterns are coordinate-wise even and go through different shapes before stabilising into the bottom left pattern of Figure \ref{fig:timetransient}. The solution stays close to this pattern for some time. Eventually a symmetry breaking event happens (bottom center of Figure \ref{fig:timetransient}) and the solution stabilises into a hexagonal pattern (bottom right of Figure \ref{fig:timetransient}) which numerically has the same coordinate-wise even symmetry axes as \eqref{eq:hex}.

As pointed out in \cite{CHS}, numerical experiments indicate that this pattern does not occur continuously as a stable bifurcation branch at a point on the spatially homogeneous branch, see Figure \ref{fig:bifurcationdiagram}. Figure \ref{fig:bifurcationdiagram} shows bifurcation diagrams with respect to $\kappa$. The diagrams are created as in \cite{CHS} for \eqref{eq:4PDE} by numerically evolving \eqref{eqn:2} up to stabilisation to a steady value for a given $\kappa$. Then, the stationary states for increasing or decreasing $\kappa$ are recursively
computed for smaller (r2l) or larger (l2r) values of $\kappa$ by taking as initial data the already computed steady state.

Indicated by a sudden jump in the difference between $\max_x \bar \rho (x)$ and $\min_x \bar \rho (x)$ in the left plot of Figure \ref{fig:bifurcationdiagram}, there is what appears to be a discontinuous phase transition between the spatially homogeneous steady state and the hexagonal steady state. This is also apparent in the bifurcation diagram of $\|\bar \rho\|_{L^2(\mathbb{T}^2)}$ to the right. Also note the different positions of the jumps for decreasing and increasing $\kappa$, which indicates a hysteresis phenomenon. Finally, the loss of linear stability (see Lemma \ref{rem:linear}) is to the right of both jumps (indicated by a red dot). This was also remarked in \cite{CHS} for \eqref{eq:4PDE}.

\begin{figure}[ht]
    \centering
    \includegraphics[width=0.45\textwidth]{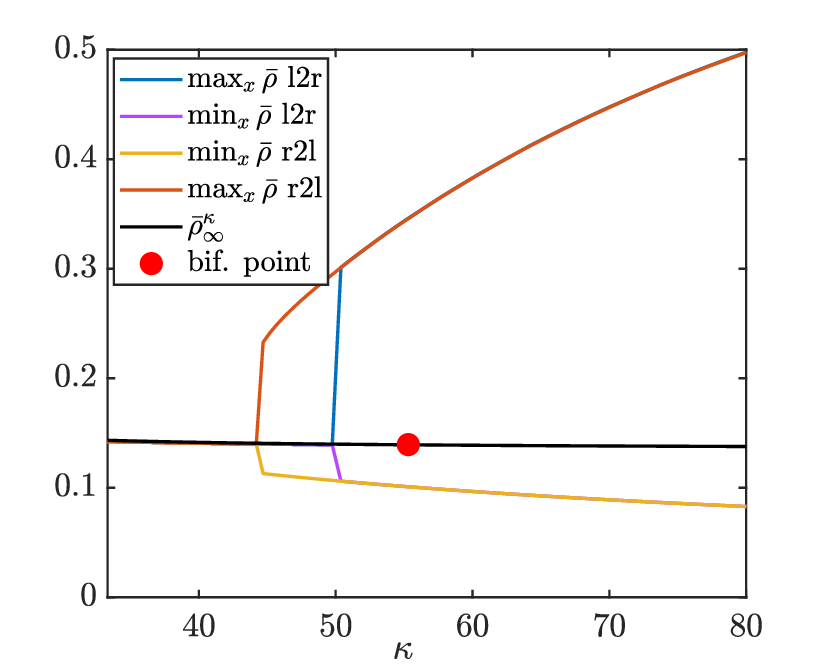}
    \includegraphics[width=0.45\textwidth,trim={0 -0.5em 0 2em},clip]{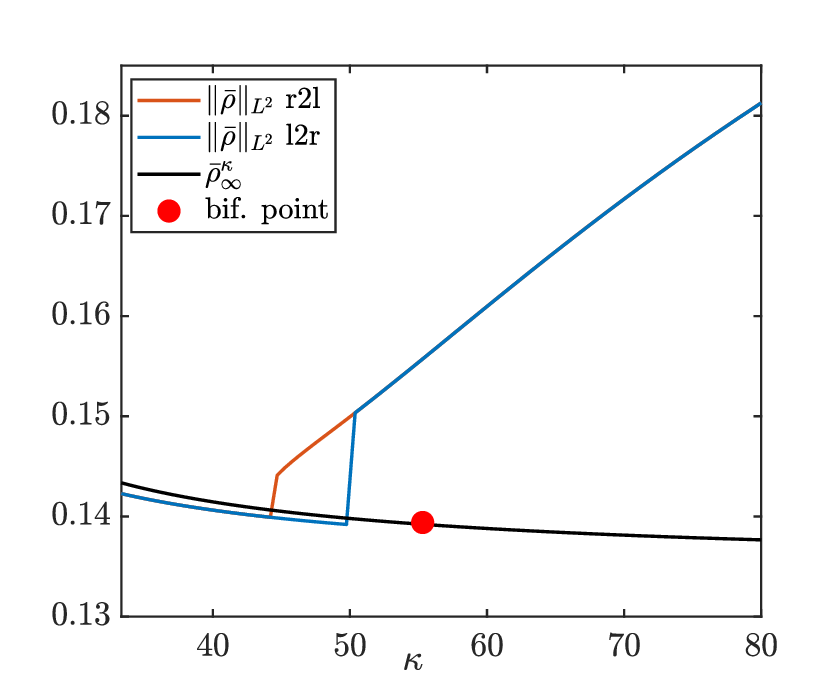}
    \caption{Bifurcation diagrams with the settings of the left plot in Figure \ref{fig:radsym-vs-nonradsym} showing the loss of stability of the spatially homogeneous state $\bar \rho_\infty^\kappa$ to the hexagonal state (see bottom right plot of Figure \ref{fig:timetransient}). Left: maximum and minimum over $x$, right: the $L^2(\mathbb T ^2)$-norm. The continuation algorithm is computed both for increasing (l2r) and decreasing (r2l) $\kappa$. The red dot marks the point of loss of linear stability.}
    \label{fig:bifurcationdiagram}
\end{figure}

\section{Perspectives}\label{sec:future}
Although having taken a crucial step towards understanding the patterning of the stochastic neural field models under consideration, we have only touched the surface of what this framework allows to explore. One direction is to be even less restrictive with the symmetry assumptions on the connectivity kernel $W$.

A further understanding of the bifurcations is needed to fully understand the emergence and stability of the hexagonal stationary state for \eqref{eqn:2} and \eqref{eq:4PDE} in the case of radially symmetric connectivity kernels. From the discussion in the previous section, it is clear that the systems possess numerically stable hexagonal steady states, and that no hexagonal steady states could be proven to exist when utilising the square periodicity of the problem. Two plausible explanations stand out regarding the emergence as the noise parameter $\sigma$ is varied. One is that there could be other bifurcations occurring along the spatially non-homogeneous branches leading to the hexagonal pattern. This conjecture is supported by the bifurcation diagrams in \cite[Figure 1]{VCF2015} for a related neural field model. This is however by no means evidence that this has to be the case for \eqref{eqn:2} as the model and bifurcation parameter in \cite{VCF2015} do differ from \eqref{eqn:2}. Furthermore, the stable patterns with square periodicity displayed in \cite{VCF2015} are only \emph{almost} hexagonal. The second is that the hexagonal pattern does in fact numerically appear as a local branch bifurcating continuously from the spatially homogeneous branch. A first thought is to again follow the approach of Dionne \cite{Dionne}, but this time restrict the functional to a hexagonal lattice. This is however complicated due to the convolution in \eqref{eqn:2}, which leads to the hexagonal pattern in Figure \ref{fig:timetransient}, having square periodicity.

Regarding the time-dependent system \eqref{eqn:2}, we have yet to provide a complete existence and well-posedness study, which is challenging due to the nonlinear boundary condition. Furthermore, nonlinear stability is a completely open and tough question. Finding a solution is crucial to fully understand the dynamics, and this process will require sophisticated analytical tools. Finally, to extend the PDE models considered here to a more realistic setting, an investigation of space correlated noise and time delays is pertinent as well.  

\subsection*{Acknowledgements}
JAC and PR were supported by the Advanced Grant Nonlocal-CPD (Nonlocal PDEs for Complex Particle Dynamics: Phase Transitions, Patterns and Synchronization) of the European Research Council Executive Agency (ERC) under the European Union's Horizon 2020 research and innovation programme (grant agreement No. 883363).
JAC was also supported by the grants EP/T022132/1 and EP/V051121/1 of the Engineering and Physical Sciences Research Council (EPSRC, UK).
Parts of this research was conducted while SS was an academic visitor at the Mathematical Institute, University of Oxford, and the author would like to thank the institution for its warm hospitality. The authors would like to thank the Isaac Newton Institute for Mathematical Sciences, Cambridge, for support and hospitality during the programme ``Frontiers in kinetic theory'' where parts of the work on this paper were undertaken, supported by EPSRC grant no EP/R014604/1. Finally, the authors would also like to thank an anonymous referee for pointing us in the direction of equivariant bifurcation theory.

\bibliographystyle{abbrv}
\bibliography{references.bib}

\appendix
\section{General results in bifurcation theory}

This appendix is adapted from the material contained in \cite{Kielhofer2012} and is meant as a brief introduction to the bifurcation theory utilised to prove Theorem \ref{thm:main}. The corresponding equivariant bifurcation theory needed in the proof of Theorem \ref{thm:ebl} can be found for example in the book by Chossat and Lauterbach \cite{CL-equivariantbif}.

Let $X,Y,Z$ be three real Banach spaces, $U\subset X$, $V \subset Y$ be two open sets and consider a continuous mapping $F:U\times V\to Z$. We provide in this appendix some abstract results about the problem
\begin{equation}\label{eqn:bifurcation_general}
    \mathcal H(x,\kappa) = 0,\qquad x\in U,\ \kappa\in V.
\end{equation}
We see $X$ as a state of configurations for a system and $Y$ as a set of parameters. Given a solution $(x_0,\kappa_0)$ to \eqref{eqn:bifurcation_general}, we want to know when a small change of the parameters around $\kappa_0$ entails a significant change in the configuration. This can happen only when the implicit function theorem cannot be applied, for example when
\[ D_x \mathcal H (x_0,\kappa_0) : X\to Z  \]
is not bijective.

Let us consider the case of a one dimensional parameter space, $V\subset Y = \R$. We make assumptions on the existence of a ``trivial'' branch of solutions to \eqref{eqn:bifurcation_general} and on the smoothness of $\mathcal H$.

\begin{hyp}\label{as:trivial}
    The open set $U$ contains 0 and for all $\kappa\in V$, $\mathcal H(0,\kappa)=0$. 
\end{hyp}
Note that if the branch is of the form $(x(s),\kappa(s))$, with $x(s)$ non-constant, it is possible to rescale the functional $\mathcal H$ to obtain a trivial branch like in the assumption. It is what we do in the beginning of the proof of Theorem \ref{thm:main}.

\begin{hyp}\label{as:smooth}
    We assume that $\mathcal H\in C^2(U\times V,Z)$. For some value $\kappa_0\in V$, we assume that $\mathcal H(\cdot,\kappa_0)$ is a Fredholm operator of index 0 such that
    \[ \dim\big(\ker(D_x \mathcal H(0,\kappa_0))\big) = \mathrm{codim}\big(\mathrm{range}(D_x \mathcal H(0,\kappa_0)\big) = 1   \]
\end{hyp}

We now state a result about the existence of a second branch of solutions crossing the trivial branch for some value $\kappa_0\in V$ of the parameter, which is one of the core results that we use in this article. It can be found in \cite[Th I.5.1.]{Kielhofer2012}.

\begin{theorem}[Crandall--Rabinowitz Theorem]\label{thm:CR}
    Grant Assumptions \ref{as:trivial} and \ref{as:smooth}. Assume that
    \begin{equation*}
        \ker\big(D_x \mathcal H(0,\kappa_0)\big) = \mathrm{span}(\omega_0),\qquad \omega_0\in U,\ \norme{\omega_0}=1 
    \end{equation*}
    and
    \begin{equation*}
        D^2_{x\kappa}  \mathcal H(0,\kappa_0) [\omega_0] \notin \mathrm{range}\big( D_x \mathcal H(0,\kappa_0) \big).
    \end{equation*}
    Then there exists a nontrivial continuously differentiable curve through $(0,\kappa_0)$
    \begin{equation}\label{eq:nontrivial}
        \{\ (x(s),\kappa(s)) \ | \ s\in(-\delta,\delta), \ (x(0),\kappa(0))=(0,\kappa_0), \delta>0 \ \},
    \end{equation}
    such that
    \begin{equation*}
        \forall s\in(-\delta,\delta),\quad \mathcal H(x(s),\kappa(s))=0,
    \end{equation*}
    and in a neighbourhood $U_1\times V_1 \subset U\times V$ of $(0,\kappa_0)$, all the solutions to \eqref{eqn:bifurcation_general} are either on the trivial solution line or on the nontrivial solution line \eqref{eq:nontrivial}.
\end{theorem}

We call the intersection $(0,\kappa_0)$ of these two curves a \textit{bifurcation point}. Note that the second order derivative $D^2_{x\kappa}  \mathcal H(0,\kappa_0): \omega\mapsto D^2_{x\kappa}  \mathcal H(0,\kappa_0)[\omega]$ while applied to the point $(0,\kappa_0)$ lies in $\mathcal L (X,Z)$ the set of bounded linear operators from $X$ to $Z$.

With more derivatives of $\mathcal H$, it is possible to harvest more information about the local behaviour of the nontrivial curve of solutions. We state here a particular case of the general results from \cite[Secs. I.5--6]{Kielhofer2012}.

\begin{theorem}[Characterisation of the branch]\label{thm:characterisation}
    Assume the hypotheses of Theorem \ref{thm:CR}, that $\mathcal H\in C^3(U\times V,Z)$, and that $X$ is a Hilbert space endowed with a scalar product $\pscal{\cdot}{\cdot}$. Then, the function $s\mapsto \kappa(s)$ is twice differentiable. The tangent vector to the nontrivial curve at the bifurcation point $(0,\kappa_0)$ is $(\omega_0, \kappa'(0))$, which means that there exists a bounded continuous function $r:\mathrm{span}(\omega_0)\times V \to L_S^2(\mathbb T^d)$ such that $r(0,0)=0$, and
    \begin{equation*}
        x(s) = s\omega_0 + r(s\,\omega_0,\kappa(s)),\qquad \lim_{s\to 0} \dfrac{\norme{r(s\,\omega_0,\kappa(s))}}{|s| + |\kappa(s) - \kappa(0)|}=0.
    \end{equation*}
    The value $\kappa'(0)$ is given by
    \begin{equation*}
        \kappa'(0) = -\dfrac12 \dfrac{\pscal{D^2_{xx} \mathcal H(0,\kappa_0) [\omega_0,\omega_0]}{\omega_0} }{\pscal{D^2_{x\kappa}  \mathcal H(0,\kappa_0)[\omega_0]}{\omega_0}}.
    \end{equation*}
    Moreover, if $\kappa'(0)=0$, then
    \begin{equation*}
        \kappa''(0) = -\dfrac13 \dfrac{\pscal{D^3_{xxx} \mathcal H(0,\kappa_0) [\omega_0,\omega_0,\omega_0]}{\omega_0} }{\pscal{D^2_{x\kappa}  \mathcal H(0,\kappa_0)[\omega_0]}{\omega_0}},
    \end{equation*}
    and
    \begin{equation*}
        x(s) = s\omega_0 + o(s), \qquad s\in(-\delta,\delta).
    \end{equation*}
\end{theorem}

When $\kappa'(0)\neq0$, the bifurcation is called \textit{transcritical}. If $\kappa'(0)=0$, it is called a pitchfork bifurcation, which is called \textit{subcritical} when $\kappa''(0)<0$ and \textit{supercritical} when $\kappa''(0)>0$.

\section{A supplementary proof}
\begin{proof}[Proof of Lemma \ref{lm:g_eta}]
    Notice that by \eqref{eqn:rho_0},
    \[\dfrac1{L^d} - L^d\bar \rho_\infty  \left(\bar \rho_\infty - \dfrac{ \Phi_0^\kappa}{L^d}\right)\kappa = \dfrac1{L^d} - L^d\bar \rho_\infty \rho_\infty(0),  \]
    and using the expression for the stationary state \eqref{eq:stationarystateequ}, we have
    \[ \rho_\infty(0) = \dfrac{\sqrt{2\kappa} \e^{-\frac\kappa2  (\Phi_0^\kappa)^2}}{ L^d \sqrt{\pi} \left(1+\mathrm{erf}(\frac{ \Phi_0^\kappa\sqrt{\kappa}}{\sqrt2})\right)}.  \]
    Moreover, using that
    \[ \int_0^{+\infty} \e^{-\frac\kappa2(s- \Phi_0^\kappa)^2}\diff s = \sqrt{\dfrac{\pi}{2\kappa}}\left(1+\mathrm{erf}\left( \dfrac{ \Phi_0^\kappa\sqrt\kappa}{\sqrt2} \right)\right),  \]
    we compute
    \begin{align}\label{eqn:rho_bar_eta} \bar \rho_\infty & = \dfrac{\displaystyle \int_0^{+\infty} (s- \Phi_0^\kappa) \e^{-\frac\kappa2(s- \Phi_0^\kappa)^2}\diff s +  \Phi_0^\kappa\int_0^{+\infty} \e^{-\frac\kappa2(s- \Phi_0^\kappa)^2}\diff s }{\displaystyle L^d \int_0^{+\infty} \e^{-\frac\kappa2(s- \Phi_0^\kappa)^2}\diff s } \nonumber \\
    & = \dfrac{\e^{-\frac\kappa2  (\Phi_0^\kappa)^2}}{\displaystyle\kappa L^d \int_0^{+\infty}\e^{-\frac\kappa2(s- \Phi_0^\kappa)^2}\diff s}   + \dfrac{ \Phi_0^\kappa}{L^d}
	    = \dfrac{\sqrt2 \e^{-\frac\kappa2  (\Phi_0^\kappa)^2}}{ L^d \sqrt{\pi\kappa} \left(1+\mathrm{erf}(\frac{ \Phi_0^\kappa\sqrt{\kappa}}{\sqrt2})\right)}   + \dfrac{ \Phi_0^\kappa}{L^d} .    \end{align}
    Hence,
    \begin{align*}
             \dfrac1{L^d} - L^d\bar \rho_\infty  \left(\bar \rho_\infty - \dfrac{ \Phi_0^\kappa}{L^d}\right)\kappa & = \dfrac1{L^d}\left( 1 -   \frac{2}{\sqrt{\pi}}\frac{\e^{-\frac\kappa2  (\Phi_0^\kappa)^2}}{1+\mathrm{erf}\left(\tfrac{ \Phi_0^\kappa\sqrt{\kappa}}{\sqrt2}\right)}\left[\frac{1}{\sqrt{\pi}}\frac{\e^{-\frac\kappa2  (\Phi_0^\kappa)^2}}{1+\mathrm{erf}\left(\tfrac{ \Phi_0^\kappa\sqrt{\kappa}}{\sqrt2}\right)}+\sqrt{\frac{\kappa}{2}} \Phi_0^\kappa\right] \right)\\
             & =  \frac{1}{L^d} g\left(\sqrt{\tfrac{\kappa}{2}} \Phi_0^\kappa\right).
    \end{align*}
    Let us denote 
    \[ f(\eta) =  \dfrac{1}{\sqrt\pi}\frac{\exp (-\eta^2)}{1+\erf (\eta)}. \]
    Since the error function $\erf(\eta)$ is differentiable on $\R$ and satisfies
    \[ \erf'(\eta)= \dfrac{2}{\sqrt\pi} \e^{-\eta^2},   \]
	then $f$ has the following property:
	\[  f'(\eta) = - 2 f(\eta) (f(\eta)+\eta).\]
    It follows that, for all $\eta\in\R_+$,
    \begin{equation*}
    g'(\eta) = 2f(\eta) \left[ 2(f(\eta) + \eta)(2f(\eta)+\eta) - 1\right].   
    \end{equation*}
    Let us denote
    \begin{equation}\label{eqn:def_h} w(\eta) = (f(\eta) + \eta)(2f(\eta)+\eta) \end{equation}
    for all $\eta\in\R_+$,
	\begin{align*}
	  w'(\eta)  & =  \big(1 - 2 f(\eta))[f(\eta)+\eta]\big)(2f(\eta)+\eta) + (f(\eta)+\eta)\big( 1-4f(\eta)[f(\eta)+\eta] \big)\\
	  & =  g(\eta)(2f(\eta)+\eta) + (2g(\eta)-1)(f(\eta)+\eta) \\
	  & =  (4g(\eta)-1) f(\eta) + (3g(\eta)-1) \eta.
	\end{align*}
	 We have
	 \[   w(0) \,=\, \dfrac2\pi \,  >\, \frac12,\]
	 so $g'(0) > 0$ and by continuity of $g'$, $g$ is increasing on a neighbourhood of 0.	Assume by contradiction that $g$ is not increasing on $\R_+$. Let $\bar\eta = \inf\{ \eta\in(0,+\infty) \ | \ g'(\eta) = 0 \}$. For all $\eta\in[0,\bar\eta]$ we have,
	\[ 4g(\eta)-1\,  >\,  3g(\eta)-1 \, \geqslant\, 3\left(1-\frac{2}{\pi}\right)-1 \,>\, 0.  \]
	Therefore, $w$ is increasing on $[0,\bar\eta]$. But then,
	\[ g'(\bar\eta) = 4f(\bar\eta) \left( w(\bar\eta) - \frac12\right) > 4f(\bar\eta) \left( w(0) - \frac12\right)  > 0, \]
    which constitutes a contradiction.
    
    Therefore, $g$ is an increasing function on $\R_+$. As a consequence, for all $\eta\in\R_+$,
    \[ 1-\frac{2}{\pi}\, =\, g(0) \, \leqslant \,  g(\eta)  \, < \, \lim_{\eta\to +\infty} g(\eta)\, =\, 1. \]
\end{proof}

\end{document}